\documentclass[11pt]{article}

\usepackage[utf8]{inputenc}
\usepackage{fullpage}
\usepackage{longtable}
\usepackage[table]{xcolor}
\usepackage{multirow,booktabs}
\usepackage{hhline}

\usepackage{verbatim}
\usepackage{todonotes}
\usepackage{algorithm}
\usepackage[noend]{algpseudocode}

\usepackage{mathtools}
\usepackage[table]{xcolor}
\usepackage{amsmath,amsthm,amsfonts, amssymb,hyperref,array,xcolor,multicol,
	verbatim,enumitem}
\usepackage[T1]{fontenc}
\usepackage{tikz}
\usetikzlibrary{positioning}
\usetikzlibrary{arrows,decorations.markings}

\newcommand{\marrow}{\marginpar[\hfill$\longrightarrow$]{$\longleftarrow$}}

\newcommand{\niceremarkcolor}[4]{\textcolor{#4}{\textsc{#1 #2:} \marrow\textsf{#3}}}
\newcommand{\lluis}[2][says]{\niceremarkcolor{Lluis}{#1}{#2}{olive}}
\newcommand{\andrew}[2][says]{\niceremarkcolor{Andrew}{#1}{#2}{blue}}

\newcommand{\overbar}[1]{\mkern 1.5mu\overline{\mkern-1.5mu#1\mkern-1.5mu}\mkern 1.5mu}

\usepackage{hyperref}

\newtheorem{theorem}{Theorem}
\newtheorem{lemma}[theorem]{Lemma}

\newtheorem{corollary}[theorem]{Corollary}
\newtheorem
{definition}[theorem]{Definition}
\newtheorem{proposition}[theorem]{Proposition}
\newtheorem{remark}[theorem]{Remark}
\newtheorem{observation}[theorem]{Observation}

\newcommand{\sg}{\mathbf{sg}}
\newcommand{\eg}{\mathbf{eg}}
\newcommand{\g}{\mathbf{g}}
\newcommand{\ee}{\mathbf{e}}
\newcommand{\ff}{\mathbf{f}}
\newcommand{\oo}{\mathbf{o}}
\newcommand{\kk}{\mathbf{k}}
\newcommand{\vv}{\mathbf{v}}
\newcommand{\cchi}{\boldsymbol{\chi}}

\makeatletter
\newcommand{\preceqdot}{\mathrel{\mathpalette\pr@ceqd@t\relax}}
\newcommand{\pr@ceqd@t}[2]{%
  \begingroup
  \sbox\z@{$#1\prec$}\sbox\tw@{$#1\preceq$}%
  \dimen@=\dimexpr\ht\tw@-\ht\z@\relax
  {\preceq}%
  \mkern-5mu
  \raisebox{\dimen@}{$\m@th#1\cdot$}%
  \endgroup
}
\makeatother

\theoremstyle{remark}

\title{Homomorphisms between graphs embedded on surfaces}

\author{Delia Garijo\thanks{Department of Applied Mathematics I, University of Seville, Spain. \texttt{dgarijo@us.es}. Supported by project PID2019-104129GB-I00/ AEI/ 10.13039/501100011033.} \and Andrew Goodall\thanks{Computer Science Institute (I\'UUK), Charles University, Prague. \texttt{andrew@iuuk.mff.cuni.cz}. Supported by GA\v{C}R 22-17398S.}  \and Llu\'{i}s Vena \thanks{Department of Mathematics, Universitat Polit\`ecnica de Catalunya, Barcelona. \texttt{lluis.vena@upc.edu}. Supported by the a Beatriu de Pin\'os Fellowship BP2018-0030 of the AGAUR, Horizon's 2020 program cofund.}}

\begin{document}

\maketitle

\begin{abstract}
We extend the notion of graph homomorphism to cellularly embedded graphs (maps) by designing operations on vertices and edges that respect the surface topology; we thus obtain the first definition of map homomorphism that preserves both the combinatorial structure (as a graph homomorphism) and the topological structure of the surface (in particular, orientability and genus). Notions such as the core of a graph and the homomorphism order on cores are then extended to maps. We also develop a purely combinatorial framework for various topological features of a map such as the contractibility of closed walks, which in particular allows us to characterize map cores. We then show that the poset of map cores ordered by the existence of a homomorphism is connected and, in contrast to graph homomorphisms, does not contain any dense interval (so it is not universal for countable posets). Finally, we give examples of a pair of cores with an infinite number of cores between them, an infinite chain of gaps, and arbitrarily large antichains with a common homomorphic image.
\end{abstract}
\tableofcontents
\section{Introduction}

Homomorphisms between sets with added structure are mappings that preserve this structure. For example, homomorphisms between graphs are mappings between their vertex sets that preserve adjacency. For multigraphs (loops and parallel edges allowed), a homomorphism is defined as a pair of mappings, one on vertices and the other on edges, which together preserve vertex-edge incidences. 
For a sense of the richness of the theory of graph homomorphisms the reader is referred to~\cite{hell04}. 

Key to the structural understanding of graph homomorphisms is the notion of a core~\cite{hell_core_1992}, and in particular the poset of cores, known as the homomorphism order. The enumeration of graph homomorphisms connects their theory to applications in statistical physics: the partition function of the $q$-state Potts model on a graph, and more generally  
partition functions of vertex-colouring models, are instances of homomorphism functions~\cite{FLS07} (given by a weighted enumeration of graph homomorphisms from the given graph to a fixed edge-weighted graph). These partition functions are in turn intimately related to the Tutte polynomial of a graph, whose evaluations include enumerations of colourings and flows. In fact,
any evaluation of the Tutte polynomial of a graph expressible as a homomorphism function is the partition function of a Potts model on the graph~\cite{GGN11}. 
For graphs cellularly embedded in a surface ({\em maps}), the analogous notions of colourings (or tensions, rather) and flows (taking non-identity values in a finite group) are counted by evaluations of the surface Tutte polynomial~\cite{GKRV18, GLRV20}, 
and have been expressed as partition functions of edge-colouring models~\cite{litjens17}. 
This leaves the question of whether these enumerations can be expressed in terms of map homomorphisms as partition functions of vertex-colouring models on the map. To even begin to answer this question we need to formulate a definition of map homomorphism that extends that of graph homomorphism while respecting the topology of the graph~embedding. We will then be able to adapt the fruitful analysis of graph homomorphisms via the poset of cores to the case of~maps.  

As we shall relate shortly, how to define homomorphisms between maps, however,  is not so clear, one difficulty consisting in the preservation of both the combinatorial structure of the graph and the topological structure of the embedding. 
Maps have several representations, some emphasizing their topological structure (cellular embeddings of graphs, ribbon graphs), some emphasizing their combinatorial structure (vertex-edge-face flags, rotation systems, graph-encoded maps, Tutte's permutation axiomatization).
What counts as a map homomorphism may thus depend on which representation is chosen: what structure is to be preserved exactly? To start on firm ground, isomorphism of maps has just one candidate for its definition no matter what representation is chosen.

We use a version of Tutte's permutation axiomatization of maps~\cite[Chapter X]{tutte01} in which three involutions defined on a set of\; ``crosses'' (obtained by quartering edges of the map) 
 describe how to get from one vertex, edge or face to an adjacent one whilst respecting the topology of the map (a formal definition is given in Definition~\ref{def:map_involutions_labelled} below). 
An isomorphism between maps in this representation is a bijection between their cross sets that commutes with the involutions defining the maps. 
By dropping the bijective condition, we obtain the definition of map homomorphism given both by Malni\v{c}, Nedela, and \v{S}koviera~\cite{malnic02} 
(restricted to the case of orientable surfaces) and by
Litjens and Sevenster~\cite{litjens17} 
 (restricted to locally bijective mappings, the context being universal covers of graphs).
Map homomorphisms defined in this way
preserve local combinatorial and topological structure (such as vertex-edge-face incidences and vertex rotations) 
but, in general, do not preserve global topological parameters such as orientability or genus.

We formulate a new definition of map homomorphism (Definition~\ref{def:maphom_full}) that ensures that homomorphisms preserve the surface topology; our definition is mainly based on a vertex identification operation that we call \emph{vertex gluing}.   
Vertex gluing, defined in terms of cross permutations (Definition~\ref{def:vertex_gluing}), has the following interpretation in the representation of a  map as a cellular embedding of a graph in a surface: two vertices can be identified under this operation if they can be moved continuously toward each other until they coincide without crossing any edges, or if they lie in different connected components; vertex gluing can be realized as the insertion of an edge (joining the vertices to be identified in a way that preserves genus and orientability) followed by its contraction. 
Any other rule of identifying vertices defined for maps generally is either a special case of vertex gluing or fails to preserve genus and orientability in some instances. We then define a map homomorphism as a sequence of such vertex gluings followed by a sequence of {\em duplicate edge gluings} (the analogue of suppressing parallel edges in graphs). Map homomorphisms in our sense preserve not only orientability and genus but also other key topological features (such as the contractibility of facial walks) and the combinatorial structure of the map (when forgetting the embedding, a homomorphism between maps gives a homomorphism between their underlying graphs).

Using our definition of map homomorphism, we define a~\emph{core} of a map (Definition~\ref{def.core}) as a minimal submap that is a homomorphic image of the whole map, analogously to how the core of a graph is defined. 
We establish several properties of map cores shared with graph cores, and characterize map cores in terms of a certain type of contractible closed walks, which, roughly speaking, separate off a disc from the remainder of the map's surface. This characterization exploits the ``locality'' of map homomorphisms, in the sense that any map homomorphism can be decomposed into a sequence of map homomorphisms that fix all but a plane submap contained within a contractible curve on the surface in which the map is embedded. After giving applications of the characterization, we show that, unlike graph homomorphisms, the map homomorphism order has no dense intervals, and is thus not universal. Finally, we produce examples of maps with an infinite number of cores between them,
an example of an infinite chain~of~gaps, and arbitrarily large antichains of cores sharing a common homomorphic image.

We would also like to highlight that another contribution of our work is a purely combinatorial formulation for various topological features of a map in its representation as a cellular embedding of a graph in a closed surface, such as the contractibility of closed walks. (A more general notion of contractibility is needed than that given by Mohar and Thomassen~\cite{mohar01}, which does not include closed walks that revisit edges within its scope, while Cabello and Mohar~\cite{CM07} give a topological definition of contractibility not straightforward to translate into combinatorial terms.)

The paper is organized as follows. 
In Section~\ref{sec:maps} we use a permutation axiomatization for maps to develop the formalism required for later constructions and proofs.  
Section~\ref{sec:del_con} focuses on the operations of deletion and contraction of edges in maps.  
By considering the effect of these operations on the various map parameters such as connectivity, genus and orientability, we arrive at eleven types of map edge (compared to the three edge types of ordinary, bridge and loop for multigraphs).
These types serve to indicate both the combinatorial and topological role an edge plays in the map.

Section~\ref{sec:map_homomomorphisms} is devoted to developing the notion of a homomorphism between maps. In Section~\ref{sec:vertex_gluing} we give a formal account of the operation of vertex gluing, 
and in Section~\ref{sec:duplicate_edges} we define the operation of duplicate edge gluing.
In Section~\ref{sec:defining_map_homomorphism} we use vertex and duplicate edge gluings to define what it means to be an epimorphism from one map onto another, and the already established notion of isomorphism between maps in order to define what it means to be a monomorphism from one map into another. By composing an epimorphism and a monomorphism we finally arrive at what it means to be a homomorphism from one map to another.

Having then in hand the formal definition of a homomorphism between maps, Section~\ref{sec:cores} is devoted to the analogue of graph cores for maps. 
In Sections~\ref{sec:cutting} and~\ref{sec:circuits-hom} topological notions such as separating and contractible curves are defined in terms of the permutation representation of maps so as to move toward the characterization of when a map is a core given in  Section~\ref{sec:char_cores}. Significant differences to graph cores then emerge in the structure of the partial order on the set of cores under the relation of ``being homomorphic to'' (Section~~\ref{sec:cores_poset}). 

\section{Maps}\label{sec:maps}





We use the combinatorial definition of maps based on Tutte's  permutation\label{def:it_perm_use} axiomatization of cellularly embedded graphs~\cite{tutte73},~\cite[Ch. X]{tutte01} (see also~\cite{Jones81, BS85}), which is adapted so as to accommodate isolated vertices (maps without edges). Also, instead of the two involutions and basic permutation of~\cite[Axioms X.1 to X.4]{tutte01}, we consider three involutions to define the map (as for generalized maps of dimension $2$ \cite{lienhardt91,lienhardt94, damlien03}, without boundary and with the artifice of empty permutation cycles to accommodate edgeless maps). We first introduce some notation.  

Permutations of a finite set $C$ are written in function notation, and the product of permutations is their composition, read from right to left: if $\alpha,\beta:C\to C$ are two permutations then $\alpha\beta$ is defined for $c\in C$ by $\alpha\beta(c)=\alpha(\beta(c))$. Parentheses are used to enclose cyclic permutations and to indicate the element being mapped; in the first case we add some spacing between the parenthesis and the elements of the cycle, and we may drop several in the second case (writing $\alpha c$ for $\alpha(c)$, note that we use Latin characters for the crosses while Greek characters for the permutations). For a sequence $\mathbf x=(x_1,x_2,\dots ,x_k)$ of distinct elements in $C$, the permutation  $(\;\;x_1 \;\;\; x_2 \;\;\; \cdots \;\;\; x_k\;\;)$ in cycle notation is denoted by $(\; \mathbf x\;)$. For a permutation $\alpha$, we let $\alpha \mathbf x=(\alpha x_1,\alpha x_2,\dots, \alpha x_k)$; then $(\;\alpha\mathbf x\;)=\alpha(\; \mathbf x\;)\alpha^{-1}$. The reverse sequence $(x_k,\dots, x_2, x_1)$ is denoted by $\overline{\mathbf x}$, and as cyclic permutations $(\; \overline{\mathbf x}\;)=(\;\mathbf x\;)^{-1}$.  Thus,  $\overline{\alpha \mathbf x}$, equal to $\alpha\overline{\mathbf x}$, stands for the sequence $(\alpha x_k,\dots, \alpha x_2,\alpha x_1)$, and $(\;\alpha\overline{\mathbf x}\;)=\alpha(\; \mathbf x\;)^{-1}\alpha^{-1}$.

\begin{definition}\label{def:map_involutions_labelled}
Let $V, E$ and $F$ be finite sets, elements of which are called {\em vertices}, {\em edges} and {\em faces}, respectively. 
A {\em map} $M=(V,E,F)$ is specified by a tuple  $(C,\alpha_0,\alpha_1,\alpha_2)$ in which $C$ is a finite set, whose elements are called \emph{crosses}; and
\begin{itemize}
 \item $\alpha_0$,$\alpha_1$,$\alpha_2:C\to C$ are involutions with no fixed points,
that is,  
$\alpha_i^2c=c$ and $\alpha_i c\neq c$ for each $i\in\{0,1,2\}$ and all $c\in C$;
\item the permutations $\alpha_0$ and $\alpha_2$ commute, that is, $\alpha_0\alpha_2=\alpha_2\alpha_0$;

\item each \emph{vertex} $v\in V$ is associated with a pair of cycles $(\; \mathbf x\;)\;(\; \alpha_2\overline{\mathbf x}\;)$ in the disjoint cycle decomposition of the permutation $\alpha_1\alpha_2$,  
where $\mathbf x$ is a sequence of crosses (we allow $\mathbf x$ to be empty, making a pair of empty cycles, which corresponds to $v$ being an isolated vertex);

\item the disjoint cycle decomposition of the involution $\alpha_1$ is supplemented by pairs of empty cycles (each pair corresponding to edgeless components  of the map);

\item each \emph{edge} $e\in E$ is associated with a pair of transpositions $(\; c \;\; \alpha_0\alpha_2 c\;)\;\;(\; \alpha_0 c \;\; \alpha_2 c\;)$ for some cross $c\in C$; 
\item each \emph{face} $z\in F$ is associated with a pair of cycles $(\; \mathbf y\;)\;(\; \alpha_0\overline{\mathbf y}\;)$  in the disjoint cycle decomposition of the permutation $\alpha_1\alpha_0$,  
where $\mathbf y$ is a sequence of crosses (we allow $\mathbf y$ to be empty, which corresponds to $z$ being the face of an isolated vertex).
\end{itemize}

\end{definition}
The vertex and face permutations of $M$ are denoted respectively by $$\tau=\alpha_1\alpha_2 \quad {\rm and} \quad \phi=\alpha_1\alpha_0.$$ A tuple $(C,\alpha_0,\alpha_1,\alpha_2)$ may be alternatively specified by giving $C,\alpha_0,\alpha_2$ and either the vertex permutation $\tau$ or the face permutation $\phi$; this is the approach taken by Tutte~\cite[Chapter~X]{tutte01}. Figure~\ref{fig:map_involutions} illustrates the involutions $\alpha_0,\alpha_1$ and $\alpha_2$, the vertex permutation $\tau$ and the face permutation $\phi$; in this and later figures it is convenient to use the ribbon graph representation for better visualization.

\begin{figure}[htb]
\centering
	\includegraphics[width=0.9\textwidth]{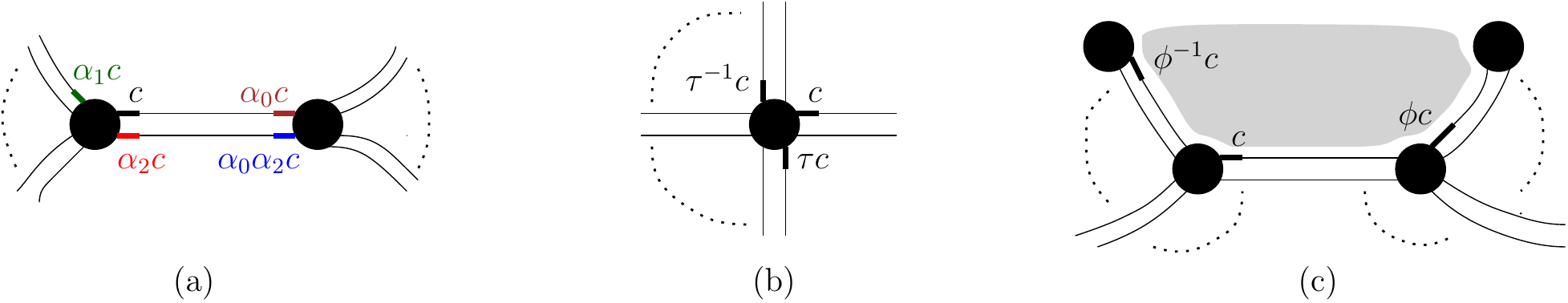}
	\caption{(a) The permutations $\alpha_0, \alpha_1, \alpha_2$ defining a map acting on a cross $c$, (b) vertex permutation $\tau=\alpha_1\alpha_2$, (c) face permutation $\phi=\alpha_1\alpha_0$. The shaded area indicates a face, while the dots indicate that there might be some additional edges in that region.}
	\label{fig:map_involutions}
\end{figure}

For a map $M=(V,E,F)$ specified by a tuple  $(C,\alpha_0,\alpha_1,\alpha_2)$ we may  write $M\equiv(C,\alpha_0,\alpha_1,\alpha_2)$ when spelling out what $V,E$ and $F$ are is not relevant to the context.  We shall also write   $v\equiv (\; \mathbf x\;)\;(\; \alpha_2\overline{\mathbf x}\;)$, $e \equiv(\;\; c \;\;\; \alpha_0\alpha_2 c\;\;)\;(\;\; \alpha_0 c \;\;\; \alpha_2 c\;\;)$ and $z \equiv (\; \mathbf y\;)\;(\; \alpha_0\overline{\mathbf y}\;)$ to indicate the cycles associated to a vertex $v$, an edge $e$, and a face $z$.

Let $v\equiv (\; \mathbf x\;)\;(\; \alpha_2\overline{\mathbf x}\;)$ be a vertex such that a cross $c$ or $\alpha_0 c$ appears in $\mathbf x$. We say that $v$ is {\em incident} with the edge $e\equiv (\;\; c \;\;\; \alpha_0\alpha_2 c\;\;)\;(\;\; \alpha_0 c \;\;\;  \alpha_2 c\;\;)$ and, similarly, edge $e$ is \emph{incident} with a face $z$ if one cross of $e$ appears in the corresponding pair of cycles of $\phi$ associated with $z$. With this definition of incidence, $\Gamma=(V,E)$ is the {\em underlying graph} of the map $M=(V,E,F)$ specified by $(C,\alpha_0,\alpha_1,\alpha_2)$. Isolated vertices of $\Gamma$ are associated with pairs of empty cycles appended to the cycle decomposition of $\tau$, and a corresponding pair of empty cycles is appended to the cycle decomposition of $\phi$ (representing the face surrounding the isolated vertex of $\Gamma$ in the given embedding). 

The {\em degree} of a vertex $v$ is the number of crosses in one of its associated cycles from the cycle decomposition of $\tau$, and the {\em degree} of the face $z$ is the number of crosses in one of its cycles in $\phi$.

 A {\em loop} of $M$ is an edge incident with just one vertex, and a {\em link} 
is an edge incident with two distinct vertices. 
We denote the number of vertices, edges and faces of $M$ by  $\vv(M), \ee(M)$ and $\ff(M)$,\label{def:it_num_vert}\label{def:it_num_edges}\label{def:it_num_faces} respectively.
Each orbit of $\langle \alpha_0,\alpha_1,\alpha_2\rangle$ acting on $C$ is a \emph{connected component} of $M$. Each pair of empty cycles of $\tau$, associated with an isolated vertex, gives its own connected component. 
The number of connected components is denoted by $\kk(M)$.\label{def:it_num_cc}
 A map is \emph{connected} if it has just one connected component (the edgeless map on one vertex is connected). 

In a connected map $M\equiv (C,\alpha_0,\alpha_1,\alpha_2)$, the number $\oo(M)$ of orbits of $\langle  \alpha_0\alpha_2, \alpha_1\alpha_2 \rangle$ acting on $C$ is either
$1$ or $2$~\cite[Theorem X.11]{tutte01}; 
 $M$ is \emph{non-orientable} if $\oo(M)=1$, and \emph{orientable} if $\oo(M)=2$.
Generally, a (not necessarily connected) map $M$ is said to be orientable if all its  components are orientable, and 
non-orientable otherwise. 
The parameter $\oo(M)$ is extended from connected maps to all maps by setting $\oo(M)=\min_i \oo(M_i)$ when $M$ has connected components $\{M_i\}$.

The \emph{Euler characteristic} of a map $M$ is defined as 
\begin{equation*}
\cchi(M)=\vv(M)-\ee(M)+\ff(M),\end{equation*}
which is used to define the \emph{Euler genus} by
\begin{equation*}
\eg(M)=2\kk(M)-\cchi(M),
\end{equation*}
the \emph{genus} by
\begin{equation*}
    \g(M)=\eg(M)/\oo(M),
\end{equation*}
and the \emph{signed genus} by
\begin{equation}\label{eq:gos}\sg(M)=[2\oo(M)-3]\g(M)=\frac{2\oo(M)-3}{\oo(M)}\eg(M)=\begin{cases} \frac{1}{2}\eg(M) & \text{ if } \oo(M)=2\\ -\eg(M) &  \text{ if }\oo(M)=1\end{cases}.\end{equation}  

The parameters $\ee,\ff,\vv,\kk,\cchi,\g$ and $\eg$ 
are additive over connected components: their value on $M$ is the sum of their values over each connected component. (The parameters of orientability $\oo$ and signed genus $\sg$ are not additive over connected components in this way.) 



Two maps $M\equiv (C,\allowbreak \alpha_0,\allowbreak \alpha_1,\allowbreak \alpha_2)$ and $M'\equiv (C', \alpha_0', \alpha_1', \alpha_2')$ are \emph{isomorphic}~\cite[X.4]{tutte01} if they have the same number of isolated vertices and there exists a bijection 
$\beta:C\to C'$ such that $\beta\alpha_ic=\alpha_i'\beta c$ for any $c\in C$ and $i\in \{0,1,2\}$. 
Note that $\beta$ determines naturally a bijection between the vertices, edges and faces of the two maps and their respective vertex and face permutation cycles.

The (geometric) \emph{dual} of a map $M\equiv(C,\alpha_0,\alpha_1,\alpha_2)$ is the map $M^{\ast}\equiv (C,\alpha_2,\alpha_1,\alpha_0)$. 
The effect of swapping $\alpha_0$ and $\alpha_2$ to form the dual gives $\phi=\tau^*$ and $\tau=\phi^*$: the rotational order of edges around a face of $M$  becomes the rotational order of edges around a vertex in $M^*$ and, dually, the rotational order of edges around a vertex of $M$ becomes the rotational order of edges around a face of $M^*$. The graph underlying $M^{\ast}$ is given by $\Gamma^{\ast}=(F,E)$ and the edge-face incidence relation.

The permutation representation of a map described in this section is better suited to our combinatorial approach  than the following equivalent topological definition: 

\begin{quote}A map is a graph $\Gamma$ embedded on a  surface $\Sigma=\Sigma_1\sqcup \cdots \sqcup \Sigma_k$ (each $\Sigma_i$ is a compact, connected $2$-manifold without boundary): vertices are points and edges are closed curves connecting pairs of vertices that do not intersect, except possibly at their endpoints (informally, the graph is drawn on the surface in such a way that two edges do not cross).
\end{quote}

  The idea behind  Definition~\ref{def:map_involutions_labelled} is to describe an embedding of a graph in a surface by explaining how to go from a point next to a given vertex, edge or face to a point next to the ``adjacent'' vertex, edge or face. In order to do this, edges are cut across and lengthwise into four parts or crosses\footnote{As Tutte explains~\cite[X.10.2]{tutte01}, the term ``cross'' was suggested by their representation in diagrams as two crossed arrows, one along an edge and the other across it, these arrows being reversed by two commuting involutions.  Bryant and Singerman~\cite{BS85} call crosses {\em blades}.
As mentioned before, the maps given by Definition~\ref{def:map_involutions_labelled} can be seen as generalized maps \cite{lienhardt91,lienhardt94, damlien03}, in which context these crosses are called {\em darts}. For Tutte, the term ``dart'' is used for a directed edge, corresponding to a pair of crosses related by the product of the two commuting involutions.} and the involutions describing adjacency act on this set of crosses (see Figure \ref{fig:map_involutions}). 

From a connected combinatorial map as  given by Definition~\ref{def:map_involutions_labelled}, we can obtain a surface~$\Sigma$ by gluing open discs along the faces (following one of the permutations of the crosses of the face); this leads to an embedding of the underlying graph $\Gamma$ of the map in~$\Sigma$ with the property that removing the vertices and edges of $\Gamma$ from $\Sigma$ leaves a set of connected components each homeomorphic to a disc. That is to say, a {\em 2-cell embedding} of $\Gamma$. Therefore, a correspondence holds between connected maps and $2$-cell embeddings of connected graphs in a surface. 
In particular, each connected component of the graph is $2$-cell embedded in its own connected surface. 
The genus and orientability of a connected surface $\Sigma$ coincide with the genus and orientability of a connected map whose underlying graph has been 2-cell-embedded in~$\Sigma$. 

 \subsection{Deletion and contraction} \label{sec:del_con}


Deletion and contraction of an edge $e$ of a map $M$ are formally defined in~\cite{GLRV20}, and denoted by $M\backslash e$ and $M/e$, respectively.\footnote{The effect of deletion and contraction of an edge on the involution 
$\alpha_1$ is described in the proof of Lemma~\ref{lem:riffle_delete_contract}.}  In the case of a link, these operations have the same effect as deletion and contraction in the underlying graph of $M$. Deleting a loop of $M$ likewise corresponds in the underlying graph to deleting the loop. However, the result of contracting a loop depends on its type, and is not the same as deleting it in the underlying graph: a loop is \emph{twisted} if a cross $c$ of the edge and $\alpha_0 c$ belong to the same permutation cycle of $\tau$, otherwise, the loop is {\em non-twisted} (a cross $c$ of the edge and $\alpha_2\alpha_0 c$ belong to the same permutation cycle of $\tau$). Contraction of a twisted loop has the effect of ``flipping over'' one side of it,  and contraction of a non-twisted loop does not correspond at all to contracting the loop in the underlying graph: indeed, contracting a non-twisted loop splits its incident vertex into two vertices. 

Just as deletion and contraction of edges in a graph are dual operations (with respect to matroid duality, the relevant matroid here being the cycle matroid of the graph), so deletion and contraction of edges in maps are dual operations (with respect to geometric duality). 
Thus~\cite[Proposition~A.5]{GLRV20},
 \begin{equation}\label{eq:dual_del_con}
 (M/e)^{\ast}=M^{\ast}\backslash e \qquad \text{ and } \qquad (M\backslash e)^{\ast}=M^{\ast} / e.
 \end{equation}

Similarly to graphs, the order in which deletion and contraction of the edges is performed is immaterial, which allows us to define deletion and contraction of a subset of edges. We let $M\backslash A$ denote the result of deleting the edges in $A$ (in any order), and $M/B$ the result of contracting the edges in $B$ (in any order).  
As shown in~\cite[Lemma~A.6]{GLRV20}, for a map $M$ and disjoint subsets $A, B$  of edges of $M$ it follows that
$(M\setminus A)/ B=(M/ B)\setminus A$.

For the purposes of this paper we only need to know the effect of these operations on the vertex and face permutations, as described in Tables~\ref{tab:del_con_tau} and~\ref{tab:del_con_phi} (actually, the definition of deletion and contraction can be retrieved from these tables). For the vertex permutation $\tau$, the effect depends on three types of edges that have already been defined: a link, a non-twisted loop, and a  twisted loop. For the face permutation $\phi$, we need three further edge types: a {\em dual link}, {\em dual twisted loop} or {\em dual non-twisted loop} of $M$, defined respectively as a link, twisted loop or non-twisted loop of the dual map $M^{\ast}$. Combinations of these give nine edge types; see Figure~\ref{fig:edge_types} for some examples.

{\footnotesize
\begin{table}[ht]
\centering
\small{
\begin{tabular}{|p{3.1cm}||p{3.5cm}|p{3.5cm}|p{3.5cm}|}
\hline
    edge type in $M$ of $e$\newline $(\,c\;\alpha_0\alpha_2 c\,)\:(\,\alpha_0 c \;\alpha_2 c\,)$ &
     link & non-twisted loop & twisted loop  \\
    \hline \hline
cycle pair(s) in $\tau$ 
&
 $(c \:\: \mathbf x) (\alpha_2 c \:\: \alpha_2 \overline{\mathbf x})$ \newline 
 $(\alpha_0\alpha_2 c \:\: \mathbf y)\:( \alpha_0 c \:\: \alpha_2 \overline{\mathbf y})$
 &
 $(c \:\: \mathbf x \:\:\alpha_0\alpha_2 c\:\: \mathbf y)$\newline $\:\:(\alpha_2 \overline{\mathbf x} \:\: \alpha_2 c \:\: \alpha_2 \overline{\mathbf y} \:\: \alpha_0 c  )$ 
 &
 $(c \:\: \mathbf x \:\: \alpha_0 c \:\: \mathbf y) $\newline $\:\:(\alpha_2 \overline{\mathbf x} \:\: \alpha_2 c \:\: \alpha_2 \overline{\mathbf y} \:\:  \alpha_0\alpha_2 c )$\\

 \hline
 \hline
 cycle pair(s) in $\tau$ \newline
 after deleting $e$ &
 $( \mathbf x )\:( \alpha_2 \overline{\mathbf x} )$ \newline
 $(\mathbf y )\:(  \alpha_2 \overline{\mathbf y})$
 &
 $(\mathbf x \:\: \mathbf y)\:(\alpha_2 \overline{\mathbf x}  \:\: \alpha_2 \overline{\mathbf y})$
 &
 $( \mathbf x \:\: \mathbf y)\:( \alpha_2 \overline{\mathbf x}  \:\: \alpha_2 \overline{\mathbf y})$
 \\
 \hline
 $\vv(M\backslash e)=$ &$\vv(M)$&$\vv(M)$&$\vv(M)$\\
 \hline
 \hline
 cycle pair(s) in $\tau$ \newline
 after contracting $e$ 
 &
 $(\mathbf x \:\: \mathbf y)\:( \alpha_2 \overline{\mathbf x} \:\: \alpha_2\overline{\mathbf y})$
&
 $( \mathbf x )\:(  \alpha_2 \overline{\mathbf x} )$\newline
 $( \mathbf y )\:(  \alpha_2 \overline{\mathbf y} )$
&
$( \mathbf x \:\: \alpha_2 \overline{\mathbf y})\:( \alpha_2 \overline{\mathbf x} \:\: \mathbf y)$\\
 \hline
 $\vv(M/e)=$ &$\vv(M)-1$&$\vv(M)+1$&$\vv(M)$\\
\hline

\end{tabular}}
\caption{\small The effect of deletion and contraction of an edge on the vertex permutation~$\tau$, in which $\mathbf x$ and $\mathbf y$ are (possibly empty) sequences of crosses. } \label{tab:del_con_tau}
\end{table}}

\begin{table}[ht]
\centering
\small{
\begin{tabular}{|p{3.1cm}||p{3.5cm}|p{3.5cm}|p{3.5cm}|}
\hline
    edge type in $M$ of $e$\newline $(\,c\;\alpha_0\alpha_2 c\,)\:(\,\alpha_0 c \;\alpha_2 c\,)$ &
     dual link & dual non-twisted loop & dual twisted loop  \\
    \hline \hline
cycle pair(s) in $\phi$ 
&
 $(c \:\: \mathbf x) (\alpha_0 c \:\: \alpha_0 \overline{\mathbf x})$ \newline 
 $(\alpha_0\alpha_2 c \:\: \mathbf y)\:( \alpha_2 c \:\: \alpha_0 \overline{\mathbf y})$
 &
 $(c \:\: \mathbf x \:\:\alpha_0\alpha_2 c\:\: \mathbf y)$\newline $\:\:(\alpha_0 \overline{\mathbf x} \:\: \alpha_0 c \:\: \alpha_0 \overline{\mathbf y} \:\: \alpha_2 c  )$ 
 &
 $(c \:\: \mathbf x \:\: \alpha_2 c \:\: \mathbf y) $\newline $\:\:(\alpha_0 \overline{\mathbf x} \:\: \alpha_0 c \:\: \alpha_0 \overline{\mathbf y} \:\:  \alpha_0\alpha_2 c )$\\

 \hline
 \hline
 cycle pair(s) in $\phi$ \newline
 after contracting $e$ &
 $( \mathbf x )\:( \alpha_0 \overline{\mathbf x} )$ \newline
 $(\mathbf y )\:(  \alpha_0 \overline{\mathbf y})$
 &
 $(\mathbf x \:\: \mathbf y)\:(\alpha_0 \overline{\mathbf x}  \:\: \alpha_0 \overline{\mathbf y})$
 &
 $( \mathbf x \:\: \mathbf y)\:( \alpha_0 \overline{\mathbf x}  \:\: \alpha_0 \overline{\mathbf y})$
 \\
 \hline
 $\ff(M/ e)=$ &$\ff(M)$&$\ff(M)$&$\ff(M)$\\
 \hline
 \hline
 cycle pair(s) in $\phi$ \newline
 after deleting $e$ 
 &
 $(\mathbf x \:\: \mathbf y)\:( \alpha_0\overline{\mathbf x} \:\: \alpha_0\overline{\mathbf y})$
&
 $( \mathbf x )\:(  \alpha_0 \overline{\mathbf x} )$\newline
 $( \mathbf y )\:(  \alpha_0 \overline{\mathbf y} )$
&
$( \mathbf x \:\: \alpha_0 \overline{\mathbf y})\:( \alpha_0 \overline{\mathbf x} \:\: \mathbf y)$\\
 \hline 
 $\ff(M\backslash e)=$ &$\ff(M)-1$&$\ff(M)+1$&$\ff(M)$\\
\hline

\end{tabular}}
\caption{\small The effect of contraction and deletion of an edge on the face permutation~$\phi$, in which $\mathbf x$ and $\mathbf y$ are (possibly empty) sequences of crosses.}
\label{tab:del_con_phi}
\end{table}

While these nine edge types are sufficient to determine the effect of deletion and contraction on the map parameters $\vv(M)$, $\ee(M)$ and $\ff(M)$, in order to determine the effect on the parameters $\kk(M)$ and $\oo(M)$, and hence, via Euler's relation, on $\g(M)$, two further edge types are needed: a {\em bridge} and a {\em dual bridge}. The defining property of a bridge $e$ is that 
\begin{equation}\label{eq:d_effect__k}\kk(M\backslash e)=\begin{cases} \kk(M)+1 & \mbox{if $e$ is a bridge,}\\
  \kk(M) & \mbox{otherwise.}\end{cases}\end{equation}
A dual bridge $e$ is defined with the same property just replacing $\kk(M\backslash e)$ by $\kk(M/ e)$.
 
A bridge is a special type of link and dual non-twisted loop (it is a link as deleting a loop on a vertex $v$ does not affect the connection of $v$ to other vertices, and it is a dual non-twisted loop as the deletion of an edge $e$ in disconnecting $M$ increases the number of faces). Similarly, 
a dual bridge is a special type of non-twisted loop and dual link. A (dual) link that is not a bridge is called a (dual) {\em ordinary} link.

Figure~\ref{fig:edge_types} shows maps that between them contain examples of all eleven edge types.
  
\begin{figure}[htb]
\centering
\includegraphics[width=0.9\textwidth]{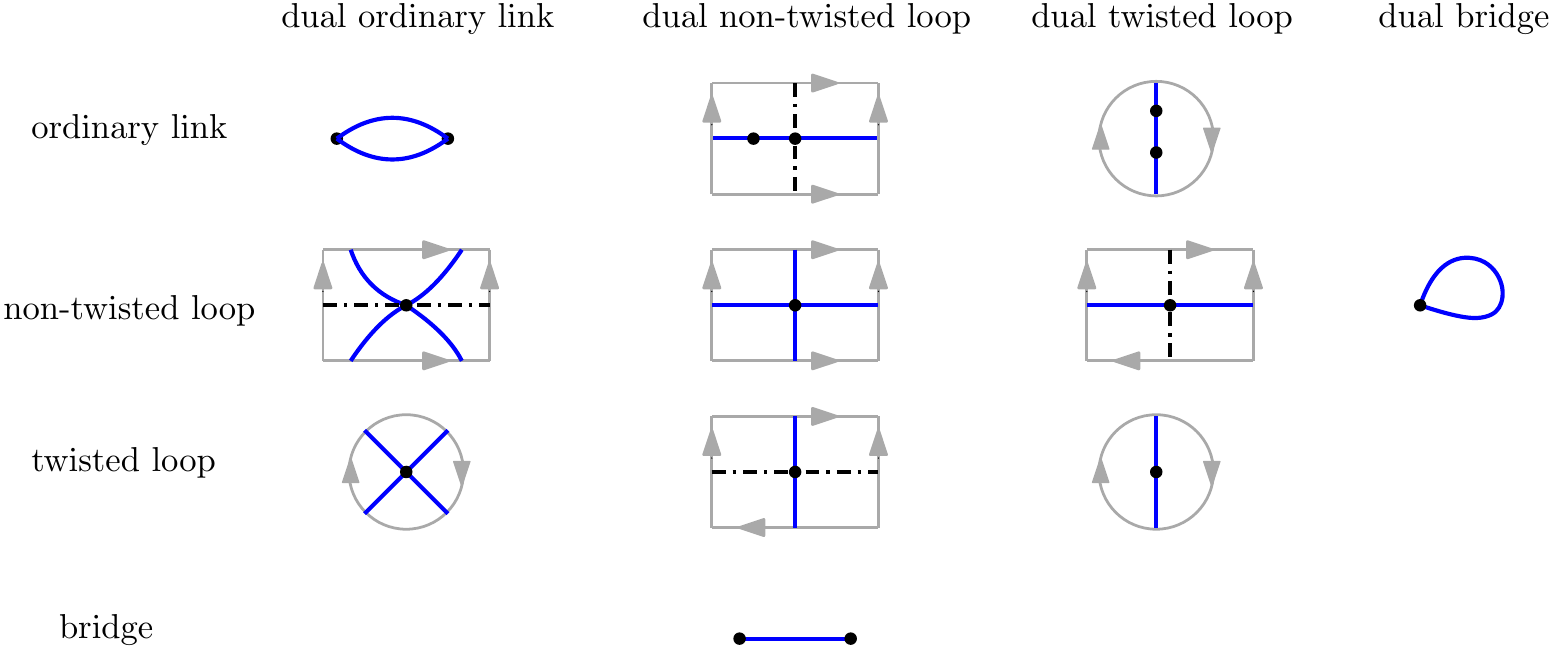}
\caption{\small Examples of maps that between them contain all eleven edge types; solid edges (in blue) are of the given type. Transpose entries in the table are dual maps (those on the diagonal self-dual).}
\label{fig:edge_types}
\end{figure}

\paragraph{Euler characteristic and Euler genus.}

The effect of deletion and contraction of an edge $e$ of $M$ on the Euler characteristic $\cchi(M)$  is recorded in Table~\ref{tab:d_effect_chi}  using the information given in Tables~\ref{tab:del_con_tau} and~\ref{tab:del_con_phi}. As can be seen from the table -- and as shown by Tutte~\cite[Theorem X.26]{tutte01} -- contracting a link does not change the Euler characteristic (dually, deleting a dual link does not change it either). 

      \begin{table}[ht]
        \centering
    
\begin{tabular}{|p{3.5cm}||>{\centering}p{0.7cm}|c|>{\centering}p{1.5cm}|c|>{\centering}p{1.2cm}|c|}
\hline
\multirow{2}{*}{
edge type 
in $M$ of $e$}& 
\multicolumn{2}{c|}{dual link} &
\multicolumn{2}{c|}{dual non-twisted loop} &
\multicolumn{2}{c|}{dual twisted loop} \\
\hhline{~------}
&
  $\backslash e$ &   $/e$ & $\backslash e$ &   $/e$ & $\backslash e$ &   $/e$ \\
  \hline \hline
link & 0& 0& +2& 0& +1 & 0 \\
\hline
non-twisted loop & 0& +2& +2& +2& +1 & +2 \\
\hline
twisted loop & 0& +1& +2& +1& +1 & +1 \\
\hline
\end{tabular}
\caption{\small Effect of deletion and contraction on the Euler characteristic: the difference $\cchi(M\backslash e)-\cchi(M)$ (in the column headed $\backslash e$) and $\cchi(M/e)-\cchi(M)$ (in the column headed $/e$) according to the nine possible edge types for $e$ (without distinguishing bridges).}
        \label{tab:d_effect_chi}
    \end{table}

   \begin{table}[ht]
{\footnotesize
\begin{tabular}{|p{2.9cm}||c|c|c|>{\centering}p{1cm}|c|>{\centering}p{1cm}|c|}
\hline
& 
\multicolumn{3}{c|}{
dual link}& \multicolumn{2}{c|}{
}&\multicolumn{2}{c|}{
}\\
\hhline{~---~~~~}
& 
dual ord. link
& 
dual ord. link & 
dual bridge
& \multicolumn{2}{c|}{\multirow{-2}{*}{
dual non-tw loop}}&
\multicolumn{2}{c|}{
\multirow{-2}{*}{
dual twisted loop}}\\
\hhline{~-------}
\multirow{-3}{=}{ edge type  in $M$ of $e$}&
$\backslash e$ &\multicolumn{2}{c|}{
$/e$} &
$\backslash e$ & 
$/e$&
$\backslash e$&
$/e$
\\
\hline \hline
 link ordinary & \multirow{2}{*}{0}&\multirow{2}{*}{0}&&-2&\multirow{2}{*}{0}&\multirow{2}{*}{-1}&\multirow{2}{*}{0}\\
 \cline{1-1}\cline{5-5}
 link bridge & &&&0&&& \\
 \hline 
 non-twisted loop & 0 & -2 & 0 & -2 & -2 & -1 & -2 \\
 \hline
 twisted loop & 0 & -1 &  & -2 & -1 & -1 & -1 \\
 \hline
\end{tabular}
}
\caption{\small Effect of deletion and contraction on the Euler genus: the difference $\eg(M\backslash e)-\eg(M)$ (in the column headed $\backslash e$) and $\eg(M/e)-\eg(M)$ (in the column headed $/e$) according to the eleven possible edge types for $e$ (non-twisted loops that are dual links are divided into those where the dual link is ordinary or a dual bridge, analogously for links that are dual non-twisted loops.)} 
        \label{tab:d_effect_s}
    \end{table}

Table~\ref{tab:d_effect_chi} together with the distinction of bridges and dual bridges yield Table~\ref{tab:d_effect_s}, which shows the effect of deletion and contraction on the Euler genus $\eg(M)$ according to the eleven edge types. The Euler genus  is unchanged when deleting a dual link or a bridge; otherwise it decreases. 



\paragraph{Orientability.}
    Tutte showed~\cite[Theorems X.26 and X.28]{tutte01} that the contraction of a link does not change orientability (dually, deleting a dual link does not change orientability either); we include this result as statement \ref{en:effect_del_con_o_2} in the following lemma (its proof is omitted).

\begin{lemma}\label{lem:effect_del_con_o}
Let $e$ be an edge of a connected map $M$.
\begin{enumerate}[label=(\roman*)]
\item\label{en:effect_del_con_o_1} If $\oo(M)=2$ then $\oo(M\backslash e)=\oo(M)=\oo(M/e)$. 
 \item\label{en:effect_del_con_o_2} If $e$ is a dual link then $\oo(M\backslash e)=\oo(M)$; if $e$ is a link then $\oo(M/e)=\oo(M)$.
\item \label{en:effect_del_con_o_3} If $e$ is a bridge or a dual bridge then $\oo(M\backslash e)=\oo(M)=\oo(M/e)$. 
 \end{enumerate}
\end{lemma}

\begin{proof}
Let $M\equiv(C,\alpha_0,\alpha_1,\alpha_2)$, and $e\equiv (\;\; c\;\;\; \alpha_0\alpha_2 c\;\;)\;(\;\; \alpha_0 c\;\;\; \alpha_2 c\;\;)$. If $\oo(M)=2$ then $\oo(M\backslash e)=2$, as deletion simply removes the crosses of $e$ from the cycles of the vertex permutation $\tau$ of $M$, and so it cannot merge the two orientation classes (orbits under the action of $\langle\tau, \alpha_0\alpha_2\rangle$).  By the same argument applied to the face permutation $\phi$ (or by duality, $\oo(M^*)=\oo(M)$) we have $\oo(M/e)=2$. This proves statement \ref{en:effect_del_con_o_1}.

When $\oo(M)=2$ statement \ref{en:effect_del_con_o_3} follows from statement \ref{en:effect_del_con_o_1}. Suppose now that edge $e$ is a bridge and $\oo(M)=1$. The group $\langle \tau,\alpha_0\alpha_2\rangle$ acting on $C$ has one orbit: the graph $\mathcal{G}_{\tau}$ on cycles of $\tau$, in which two cycles are joined by an edge if there is a cross $a$ in one such that $\alpha_0\alpha_2 a$ lies in the other, is connected (as the number of connected components of $\mathcal{G}_{\tau}$ corresponds to the number of orbits of $\langle \tau,\alpha_0\alpha_2\rangle$ acting on $C$).
The cycles  
\[(\;\;c \;\;\; \mathbf x\;\;) \;(\;\;\alpha_2 c \;\;\; \alpha_2 \overline{\mathbf x}\;\;)\qquad \text{ and }\qquad (\;\;\alpha_0\alpha_2c\;\;\; \mathbf y\;\;)\;(\;\;\alpha_0 c \;\;\; \alpha_2 \overline{\mathbf y}\;\;)\]
of $\tau$ containing crosses of  $e$ have an edge in $\mathcal{G}_{\tau}$ connecting $(\;\;c \;\;\; \mathbf x\;\;)$ and $(\;\;\alpha_0\alpha_2c\;\;\;\mathbf y\;\;)$, and an edge connecting $(\;\;\alpha_2 c \;\;\; \alpha_2 \overline{\mathbf x}\;\;)$ and $(\;\; \alpha_0 c \;\;\; \alpha_2 \overline{\mathbf y}\;\;)$.
After deleting $e$, these cycles of $\tau$ become the cycles $(\;\; \mathbf x\;\;) ( \;\;\alpha_2 \overline{\mathbf x}\;\;)$ and $(\;\; \mathbf y\;\;)\;(\;\; \alpha_2 \overline{\mathbf y}\;\;)$ of the vertex permutation $\tau '$ of $C\backslash \{c,\alpha_0\alpha_2 c, \alpha_0 c, \alpha_2 c\}$, and the other cycles of $\tau$ remain the same. 
The edge previously joining  $(\;\;c \;\;\; \mathbf x\;\;)$ to $(\;\;\alpha_0\alpha_2c\;\;\; \mathbf y\;\;)$ and that joining $(\;\;\alpha_2 c \;\;\;\alpha_2 \overline{\mathbf x}\;\;)$ to $(\;\;\alpha_0 c \;\;\; \alpha_2 \overline{\mathbf y}\;\;)$ both disappear in the new graph $\mathcal{G}_{\tau'}$.  (Since $e$ is a bridge, the crosses in $\mathbf x$ are not in the same orbit as crosses of $\mathbf y$ under the action of $\langle \alpha_0,\alpha_1,\alpha_2\rangle$ when restricted to $C\backslash\{c,\alpha_0\alpha_2 c, \alpha_0 c, \alpha_2 c\}$.) However, all other edges in $\mathcal{G}_{\tau}$ remain. Therefore, the graph $\mathcal{G}_{\tau'}$ cannot have four connected components, that is, the group $\langle \tau',\alpha_0\alpha_2\rangle$ acting on $C\backslash \{c,\alpha_0\alpha_2 c, \alpha_0 c, \alpha_2 c\}$ cannot have four orbits (there is either a single orbit of the action of $\langle \tau', \alpha_0\alpha_2\rangle$ that contains the crosses of $\mathbf x$ or a single orbit that contains the crosses of $\mathbf y$). This implies that both connected components of $M\backslash e$ cannot be orientable, and so $\oo(M\backslash e)=\oo(M)=1$.

 Finally, we prove that $M/e$ is non-orientable as the single orbit of $\langle \tau, \alpha_0\alpha_2\rangle$ acting on $C$ remains a single orbit in $M/e$. After the contraction of $e$,  the cycles of $\tau$ containing crosses of $e$ become
 $(\;\;\mathbf x \;\;\; \mathbf y\;\;)\;(\;\;\alpha_2\overline{\mathbf x} \;\;\; \alpha_2\overline{\mathbf y}\;\;)$. Then,
 the orbit containing the crosses of $\mathbf x$ coincides with that containing crosses of $\mathbf y$ under the action of $\langle \tau, \alpha_0\alpha_2\rangle$  on $C\backslash \{c,\alpha_0\alpha_2 c, \alpha_0 c, \alpha_2 c\}$. Hence, $\oo(M\backslash e)=1$.
 
 The case of a dual bridge in statement \ref{en:effect_del_con_o_3} then follows by $\oo(M^*)=\oo(M)$, $(M\backslash e)^*=M/e$, and $(M/e)^*=M\backslash e$.
\end{proof}

\paragraph{Genus and signed genus.} 
In the following two lemmas we highlight those properties of the (signed) genus with respect to deletion and contraction of an edge that we shall have need for later.

\begin{lemma} \label{l.link_equiv}
Let $e$ be an edge of a map $M$. Then,

\begin{enumerate}[label=(\roman*)]
    \item 
    \label{en:lem12_i} $\sg(M\backslash e)=\sg(M)$ if and only if 
$e$ is a dual link  or a bridge.
\item 
\label{en:lem12_ii} $\sg(M/ e)=\sg(M)$ if and only if 
$e$ is a link  or a dual bridge.
    \item 
    \label{en:lem12_iii}
    $\sg(M\backslash e)=\sg(M)=\sg(M/e)$ if and only if $e$ is a link and dual link, a bridge, or a dual bridge.
    \end{enumerate}
\end{lemma}

\begin{proof}
It suffices to prove statement \ref{en:lem12_i}, as \ref{en:lem12_ii} follows by duality and \ref{en:lem12_iii} is the conjunction of \ref{en:lem12_i} and \ref{en:lem12_ii}.
Suppose first that $e$ is a dual link or a bridge. From Table~\ref{tab:d_effect_s} it follows that $\eg(M\backslash e)=\eg(M)$, and by Lemma~\ref{lem:effect_del_con_o} we have $\oo(M\backslash e)=\oo(M)$. This implies $\sg(M\backslash e)=\sg(M)$ by Equation (\ref{eq:gos}).

Assume now that $\sg(M\backslash e)=\sg(M)$ and that $e$ is not a dual link. Then, edge $e$ is a dual loop. If $\oo(M)=2$, by Lemma~\ref{lem:effect_del_con_o}, maps $M$ and $M\backslash e$ have the same orientability. Hence $\eg(M\backslash e)=\eg(M)$, and by Table~\ref{tab:d_effect_s}, this equality holds if and only if $e$ is a bridge.

Suppose now that $\oo(M)=1$ and $\oo(M\backslash e)=2$ (otherwise we could argue as above).  Equation (\ref{eq:gos}) gives $\sg(M)=-\eg(M)$ and $\sg(M\backslash e)=\frac{1}{2}\eg(M\backslash e)$. Further, by Table~\ref{tab:d_effect_s} it follows that $\eg(M\backslash e)\in\{\eg(M)-1, \eg(M)-2\}$, which gives a contradiction as the Euler genus is an integer number. Therefore, $M$ and $M\backslash e$ must have the same orientability, and we can conclude that $e$ is a bridge with the same argument as for the case $\oo(M)=2$.
\end{proof}

By Lemma~\ref{l.link_equiv} and Table~\ref{tab:d_effect_s}, we have $\sg(M\backslash e)=\sg(M)$ if and only if $\eg(M\backslash e)=\eg(M)$.
The effect of edge deletion on the signed genus, $\sg(M)$, is otherwise to decrease its  absolute value,~$\g(M)$: 

\begin{lemma}\label{lem:effect_del_sg}
Let $e$ be an edge of a map $M$.
\begin{enumerate}[label=(\roman*)]
    \item \label{en:effect_del_sg_1} If $M$ is orientable then 
    $\g(M\backslash e)< \g(M)$ unless $e$ is a dual link or a bridge, in which case $\g(M\backslash e)=\g(M)$.
    \item \label{en:effect_del_sg_2} If $M$ and $M\backslash e$ are non-orientable then $\g(M\backslash e)< \g(M)$ unless $e$ is a dual link or a bridge, in which case $\g(M\backslash e)=\g(M)$.
    \item\label{en:effect_del_sg_3} If $M$ is non-orientable and $M\backslash e$ is orientable then  $2\g(M\backslash e)<\g(M)$. 
\end{enumerate} 
  \end{lemma}
\begin{proof}
If $M$ is orientable, by Lemma~\ref{lem:effect_del_con_o}(i),  $M\backslash e$ is orientable. Statement~\ref{en:effect_del_sg_1} then follows by Table~\ref{tab:d_effect_s} and the fact that, under the orientability assumptions, $\eg(M)=2\g(M)$ and $\eg(M\backslash e)=2\g(M\backslash e)$.   
Similarly, statement~\ref{en:effect_del_sg_2} follows by Table~\ref{tab:d_effect_s} and the fact that $\eg(M)=\g(M)$ and $\eg(M\backslash e)=\g(M\backslash e)$ (under the orientability assumptions). In both statements, by Lemma~\ref{l.link_equiv}, we have $\g(M\backslash e)=\g(M)$ when $e$ is a dual link or a bridge.

By Lemma~\ref{lem:effect_del_con_o}\ref{en:effect_del_con_o_1} we cannot have $\oo(M)=2$ and $\oo(M\backslash e)=1$, so the remaining case to consider is $\oo(M)=1$ and $\oo(M\backslash e)=2$, which is the hypothesis of statement~\ref{en:effect_del_sg_3}. By Lemma~\ref{lem:effect_del_con_o}\ref{en:effect_del_con_o_2}, in this case the edge $e$ must be a dual loop. Referring to Table~\ref{tab:d_effect_s}, this implies $\eg(M\backslash e)<\eg(M)$ unless $e$ is a bridge, in which case by Lemma~\ref{lem:effect_del_con_o}\ref{en:effect_del_con_o_3} we have $\oo(M\backslash e)=\oo(M)$, contrary to the assumption.
Under the given orientability assumptions, $\eg(M)=\g(M)$ and $\eg(M\backslash e)=2\g(M\backslash e)$.
\end{proof}

\paragraph{Euler genus and signed genus of submaps.}

Let $v$ be a vertex of a map $M=(V,E,F)$. The map $M-v$ obtained by {\em deleting} $v$ is the map that results by first deleting all edges  incident with $v$, and then removing the empty pair of cycles associated with the now isolated vertex~$v$. A map $N$ is an {\em induced submap} of $M$ if $N=M-U$ for some $U\subseteq V$; similarly, $N$ is a \emph{spanning submap} of $M$ if $N=M\backslash A$ for some $A\subseteq E$. An induced submap of a spanning submap of $M$ is called a \emph{submap} of $M$. All types of submaps are said to be {\em proper} when they are distinct from the map.

If $M'$ is a submap of $M$ then $\vv(M')\leq \vv(M)$ and $\ee(M')\leq \ee(M)$, and both the number of components and the number of faces may change in both directions.\footnote{For example, two loops embedded in the torus are incident with one common face, while deleting one of the loops leaves the other loop incident with two faces; deleting this single loop gives an isolated vertex and just one face again.} However, we next show that the Euler genus is monotonous.

\begin{lemma}\label{lem:s_submap}
Let $M'$ be a submap of a map $M$. Then,
\begin{enumerate}[label=(\roman*)]
\item\label{en:prop_s_submap_1} $\eg(M')\leq \eg(M)$, and
    \item\label{en:prop_s_submap_2} $\eg(M')=\eg(M)$ if and only if $\sg(M')=\sg(M)$.
\end{enumerate}
\end{lemma}
\begin{proof} 
Part \ref{en:prop_s_submap_1} follows from Table~\ref{tab:d_effect_s}: deleting edges does not increase the Euler genus, and neither does deleting isolated vertices. Part~\ref{en:prop_s_submap_2} follows from Table~\ref{tab:d_effect_s} and Lemma~\ref{l.link_equiv}. 
\end{proof}

\section{Map homomorphisms}\label{sec:map_homomomorphisms}

The image of a graph homomorphism  can be realized as a sequence of identifications of pairs of distinct vertices, followed by the suppression of  parallel edges (i.e., edges incident with the same pair of vertices, or same vertex in the case of parallel loops). 
  In order to define the image of a homomorphism from a map, so that restricted to the underlying graph gives a graph homomorphism, we need to define how to identify a pair of vertices on a map so as to produce another map unambiguously, and to define what it means for edges of a map to be parallel.
 When identifying vertices in maps we also need to take into account their incident faces along with the vertex and face rotations: this is where the permutation axiomatization of maps becomes essential in order to formulate a well-defined operation on maps analogous to vertex identification in graphs.

\subsection{Vertex gluing}\label{sec:vertex_gluing}
In this section we define the operation of vertex gluing in terms of cross permutations and draw on properties of edge deletion and contraction in order to establish that vertex gluing is the only general way to identify vertices in a map while preserving  genus and orientability (this property is needed, in particular, to enable restriction and composition of homomorphisms to be generally defined). 
First, we introduce an operation on maps represented by cross involutions, and go on to explain how vertex gluing can be formalized using this operation.
\begin{definition}\label{def:merging}
Let $M\equiv(C,\alpha_0,\alpha_1,\alpha_2)$ be a map and $a,b\in C$. 
The map obtained by {\em riffling} $a$ and $b$ is $M^{(\, a\: b\,)}\equiv(C,\alpha_0,\widetilde{\alpha}_1,\alpha_2)$ 
in which 
$\widetilde{\alpha}_1a=\alpha_1 b, \; \widetilde{\alpha}_1b=\alpha_1 a,$
and $\widetilde{\alpha}_1=\alpha_1$ on $C\setminus\{a,b,\alpha_1 a,\alpha_1 b\}$. 
\end{definition}

Figure \ref{fig:vertexgluing} illustrates examples of the riffling operation.
With $\widetilde{\alpha}_1a=\alpha_1 b$ and $\widetilde{\alpha}_1b=\alpha_1 a$, we have $\widetilde{\alpha}_1(\alpha_1 a)=\widetilde{\alpha}_1 ^2b=b$, and similarly $\widetilde{\alpha}_1(\alpha_1b)=a$. The involution $\widetilde{\alpha}_1$ on $C$ obtained by riffling $a$ and $b$ is the result of conjugating the involution $\alpha_1$ by the transposition $(\;\; a \;\;\; b\;\;)$, or by the transposition $(\;\; \alpha_1 a \;\;\; \alpha_1 b\;\;)$; we thus have  $M^{(\, a\: b\,)}=M^{(\, \alpha_1 a\:\: \alpha_1 b\,)}.$ 
The pair of transpositions $(\;\; a \;\;\; \alpha_1 a\;\;)\:(\;\;b \;\;\; \alpha_1 b\;\;)$ in the disjoint cycle decomposition of $\alpha_1$ is replaced by
the pair $(\;\; a \;\;\; \alpha_1 b\;\;)\:(\;\;b \;\;\; \alpha_1 a\;\;)$.  

 While any pair of distinct crosses $a,b$ of $M$ can be riffled to produce another map, only under certain conditions are genus and orientability preserved.  We next introduce some terminology to help describe what these conditions are.        
A cross $c$ is {\em incident} with a vertex $v$ (or a face $z$) if the pair of permutation cycles associated with $v$ (respectively $z$) contains $c$. 
A pair of crosses are {\em coincident} with vertex $v$ (face $z$) if they belong to the same orbit of $\tau$ (respectively $\phi$), i.e. they belong to the same permutation cycle in the pair of cycles associated with $v$ (respectively $z$).
If $a,b$ are coincident with a common face $z$, the crosses $\alpha_1 a$ and $\alpha_1 b$ appear in the conjugate inverse cycle to $(\;\; a \;\;\; \dots \;\;\; b \;\;\; \dots \;\;)$; thus $a,b$ are coincident with a common face $z$ if and only if $\alpha_1 a, \alpha_1 b$ are coincident with a common face~$z$. 


\begin{lemma}\label{lem:riff_signedgenus}
The operation of riffling crosses $a$ and $b$ 
preserves genus and orientability when one of the following conditions holds:
\begin{enumerate}[label=(\arabic*)]
    \item \label{en:riff_signedgenus_1} (i) crosses $a$ and $b$ in $M$ are incident with distinct vertices and coincident with a common face, or (ii) these crosses are incident with vertices and faces in different connected components; 
    \item \label{en:riff_signedgenus_2} (i) crosses $a$ and $b$ in $M$ are coincident with a common vertex and incident with distinct faces, or (ii) these crosses are coincident with a common vertex and face and $\kk(M^{(\, a\: b\,)})>\kk (M)$. 
\end{enumerate}
In other words, in these cases,
$\sg(M^{(\,a\:b\,)})=\sg(M).$
\end{lemma}
\begin{proof} First assume that crosses $a,b$  satisfy condition \ref{en:riff_signedgenus_1}. 
We define $M^+\equiv(C^+,\alpha_0^+,\alpha_1^+,\alpha_2^+)$ as the map obtained from $M$ by adding an edge $e\equiv (\;\; c\;\;\; \alpha_0^+\alpha_2^+c\;\;)\;(\;\;\alpha_0^+c\;\;\; \alpha_2^+c\;\;)$ 
so that $M^+\backslash e=M$ and $M^+/e=M^{(\, a\: b\,)}$. Explicitly, one can check that $M^+/e=M^{(\, a\: b\,)}$ setting 
\begin{itemize}
     \item $C^+=C\sqcup\{c,\alpha_0^+ c,\alpha_2^+ c,\alpha_0^+\alpha_2^+ c\};$
    \item $\alpha_0^+=\alpha_0$ on $C$;
    \item  $\alpha_1^+=\alpha_1$ on $C\setminus\{a,b,\alpha_1 a,\alpha_1 b\}$, and  $\alpha_1^+a=c$,  $\alpha_1^+b=\alpha_0^+\alpha_2^+ c$, $\alpha_1^+(\alpha_1 a)=\alpha_2^+c$,  $\alpha_1^+(\alpha_1 b)=\alpha_0^+ c$;
   
    \item $\alpha_2^+=\alpha_2$ on $C$.
\end{itemize}
When $a$ and $b$ belong to the same permutation cycle of $\phi$, the edge $e$ is a link and dual link in $M^+$; when they belong to different connected components of $M$, the edge $e$ is a bridge of $M^+$. By Lemma~\ref{l.link_equiv}, $\sg(M)=\sg(M^+)$ and $\sg(M^+)=\sg(M^+/e)$, whence $\sg(M)=\sg(M^{(\, a\: b\,)})$.

To see that riffling crosses satisfying condition \ref{en:riff_signedgenus_2} also preserves genus and orientability, it suffices to show that if in $M$ crosses $a$ and $b$ are coincident with a common vertex and incident with distinct faces (condition \ref{en:riff_signedgenus_2}(i)), then in $M^{(\,a\;b\,)}$ crosses $a$ and $b$ are incident with distinct vertices and coincident with a common face (condition \ref{en:riff_signedgenus_1}(i));
likewise, if in $M$ crosses $a$ and $b$ are coincident with a common vertex and common face and $\kk(M^{(\, a\: b\,)})>\kk (M)$ (condition \ref{en:riff_signedgenus_2}(ii)), then in $M^{(\,a\;b\,)}$ crosses $a$ and $b$ are incident with vertices and faces in different connected components (condition \ref{en:riff_signedgenus_1}(ii)).
For then we have, $\sg(M^{(\, a\: b\,)})=\sg\large((M^{(\, a\: b\,)})^{(\, a\: b\,)}\large)$, and the result follows as $(M^{(\,a\;b\,)})^{(\,a\;b\,)}=M$. 

It remains then to establish the effect of riffling crosses $a, b$ on the vertex and face permutations of $M$ under condition~\ref{en:riff_signedgenus_2}; as riffling $a,b$ again returns $M^{(\, a\: b\,)}$ to the original map $M$, it suffices to  establish the effect of riffling on $\tau$ and $\phi$ under condition~\ref{en:riff_signedgenus_1}. (We choose this direction as it will be useful for clarifying what is involved in the key operation of vertex gluing, defined in Definition~\ref{def:vertex_gluing} below and illustrated in Figure~\ref{fig:vertexgluing} with the same notation as here.) 

Suppose first condition~\ref{en:riff_signedgenus_1}(i)
that the crosses $a$ and $b$ are incident with distinct vertices $u$ and $v$, and coincident with a common face $z$ in $M$. 
The pair of face permutation cycles associated with $z$ takes the form 
$$\phi_z=(\;\; a \;\;\; \mathbf x \;\;\; b \;\;\; \mathbf y\;\;)\;(\;\; \alpha_0 \overline{\mathbf x} \;\;\; \alpha_0 a \;\;\; \alpha_0 \overline{\mathbf y} \;\;\; \alpha_0 b \;\;),$$
and the pairs of vertex permutation cycles associated with $u$ and $v$ take the form 
$$\tau_u=(\;\;a \;\;\; \mathbf u\;\;)\;(\;\; \alpha_2 \overline{\mathbf u}\;\;\; \alpha_2 a \;\;)\; \quad\mbox{and}\quad \;\tau_v=(\;\; b \;\;\; \mathbf v\;\;)\;(\;\;\alpha_2 \overline{\mathbf v} \;\;\; \alpha_2 b \;\;), 
$$
for  (possibly empty) cross sequences $\mathbf x,\mathbf y,\mathbf u$ and $\mathbf v$. 

After riffling $a$ and $b$, the face $z$ of $M$ is split into two faces $x,y$ and vertices $u$ and $v$ are merged into a single vertex $w$; in terms of the vertex permutation $\widetilde{\tau}=\widetilde{\alpha}_1\alpha_2$ and face permutation $\widetilde{\phi}=\widetilde{\alpha}_1\alpha_0$ of the map $M^{(\, a\: b\,)}\equiv(C,\alpha_0,\widetilde{\alpha}_1,\alpha_2)$  obtained upon merging $u$ and $v$ through $z$, 
$$\widetilde{\phi}_x=(\;\; a \;\;\; \mathbf x\;\;)\;(\;\; \alpha_0 \overline{\mathbf x} \;\;\; \alpha_0 a \;\;
)\; \quad\mbox{and}\quad \;\widetilde{\phi}_y=(\;\; b \;\;\; \mathbf y\;\;)\;(\;\; \alpha_0 \overline{\mathbf y} \;\;\; \alpha_0 b \;\;), $$
and
$$\widetilde{\tau}_w=(\;\; a \;\;\; \mathbf u \;\;\; b \;\;\; \mathbf v\;\;)\;(\;\;\alpha_2 \overline{\mathbf u} \;\;\; \alpha_2 a \;\;\; \alpha_2 \overline{\mathbf v} \;\;\; \alpha_2 b \;\;).$$
On other crosses we have $\widetilde{\tau}=\tau$ and $\widetilde{\phi}=\phi$. Thus,
riffling $a,b$ consists in merging  the vertices $u$ and $v$ into a single vertex $w$ while splitting the common face $z$ into faces $x$ and $y$. 
The crosses $a,b$ now satisfy condition~(i) of part~\ref{en:riff_signedgenus_2}. Riffling $a,b$ again returns us to condition~(i) of part~\ref{en:riff_signedgenus_1}. 

Suppose now condition~\ref{en:riff_signedgenus_1}(ii) that crosses $a,b$ are incident with vertices $u, v$ and faces $x, y$ that are in different connected components of $M$.
Here we have 
$$\tau_u=(\;\; a \;\;\; \mathbf u\;\;)\;(\;\;\alpha_2 \overline{\mathbf u} \;\;\; \alpha_2 a \;\;)\;\quad\mbox{and}\quad\;\tau_v=(\;\;  b \;\;\; \mathbf v\;\;)\;(\;\; \alpha_2 \overline{\mathbf v} \;\;\;\;\; \alpha_2 b \;\;), 
$$
and $$\phi_x=(\;\; a \;\;\; \mathbf x\;\;)\;(\;\; \alpha_0 \overline{\mathbf x} \;\;\; \alpha_0 a \;\;)\;\quad\mbox{and}\quad\;\phi_y=(\;\;b \;\;\; \mathbf y\;\;)\;(\;\; \alpha_0 \overline{\mathbf y} \;\;\; \alpha_0 b \;\;), $$
for (possibly empty) cross sequences $\mathbf x,\mathbf y,\mathbf u$ and $\mathbf v$.
Then, the result of merging $u$ and $v$ into a single vertex $w$, along with $x$ and $y$ into a single face $z$, is to produce the map $M^{(\, a\: b\,)}\equiv(C,\alpha_0,\widetilde{\alpha}_1,\alpha_2)$ in which   
$$\widetilde{\tau}_w=(\;\;\mathbf u \;\;\; a \;\;\; \mathbf v \;\;\; b\;\;)\;(\;\; \alpha_2 a \;\;\; \alpha_2 \overline{\mathbf u} \;\;\; \alpha_2 b \;\;\; \alpha_2 \overline{\mathbf v} \;\;).$$
$$\widetilde{\phi}_z=(\;\; a \;\;\; \mathbf x \;\;\; b \;\;\; \mathbf y\;\;)\;(\;\;\alpha_0 \overline{\mathbf x} \;\;\; \alpha_0 a \;\;\; \alpha_0 \overline{\mathbf y} \;\;\; \alpha_0 b \;\;).$$
The crosses $a,b$ now satisfy condition (ii) of part~\ref{en:riff_signedgenus_2}. Riffling $a,b$ again returns us to condition~(ii) of part~\ref{en:riff_signedgenus_1}. 
\end{proof}

  Condition~\ref{en:riff_signedgenus_1} in Lemma~\ref{lem:riff_signedgenus} under which genus and orientability are preserved by riffling crosses features in our definition of vertex gluing; see Figure~\ref{fig:vertexgluing}. (The inverse operation to vertex gluing of riffling crosses under condition~\ref{en:riff_signedgenus_2}, which splits a vertex while preserving genus and orientability, will feature later in Section~\ref{sec:cutting}.) 
\begin{definition}\label{def:vertex_gluing}
Let $a,b$ be crosses of a map $M$ such that either
\begin{enumerate}[label=(\roman*)]
    \item \label{en:vertex_gluing_1}\; $a$ and $b$ are coincident with a face $z$ while $a$ is incident with a vertex $u$ and $b$ is incident with a different vertex $v$;
    or
\item\label{en:vertex_gluing_2}\; $a,b$ are incident with vertices $u, v$ and faces $x, y$ belonging to different connected components. 
\end{enumerate} 
The map obtained from $M$ by {\em gluing} vertices $u$ and $v$, by (i) splitting face $z$ or by (ii) merging faces $x$ and $y$, is the map $M^{(\, a\; b\,)}$ obtained from $M$ by riffling $a$ and $b$. 

In case either $u$ or $v$ is an isolated vertex, {\em gluing} vertices $u$ and $v$ is simply the deletion of the isolated vertex (if both are isolated, then either is chosen arbitrarily to be deleted).
\end{definition}

\begin{figure}[htb]
\centering
\includegraphics[width=0.9\textwidth]{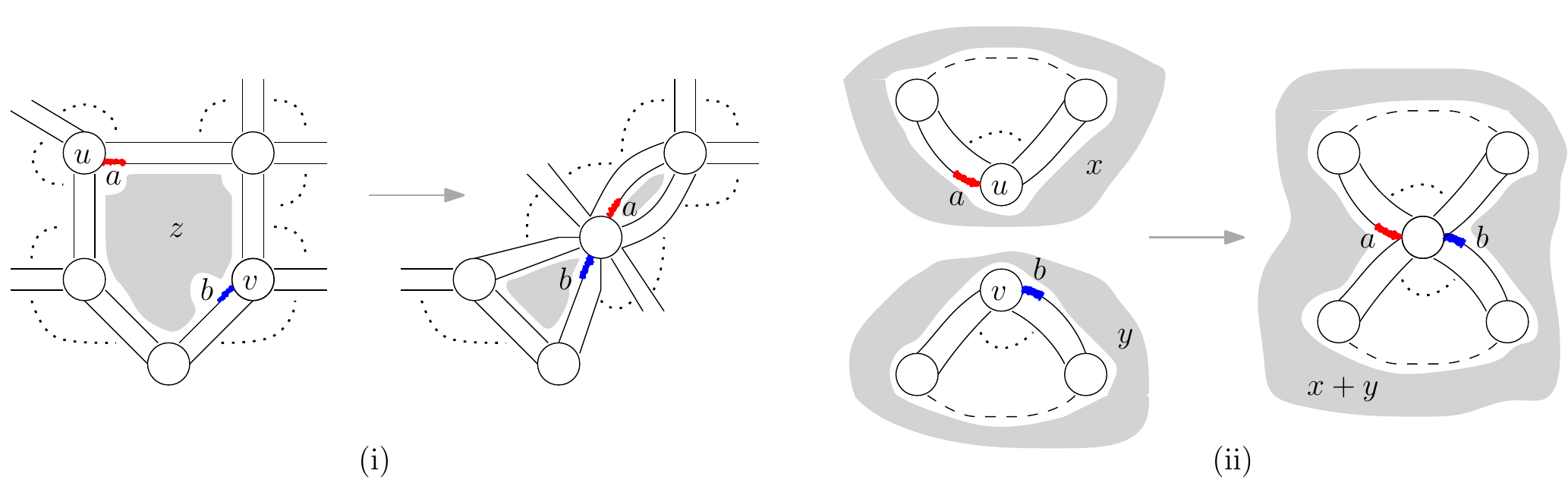}
\caption{\small Riffling the crosses $a$ and $b$ to obtain the map $M^{(\, a\: b\,)}$; this is the map that results from gluing distinct vertices $u$ and $v$, whose incident crosses $a$ and $b$ are: (i) coincident with a face $z$, (ii) incident with different faces $x,y$ in different connected components. }
\label{fig:vertexgluing}
\end{figure}

  The effect of vertex gluing on the vertex and face permutations of $M$ is given explicitly at the end of the proof of Lemma~\ref{lem:riff_signedgenus} above. 

 Having now arrived at a formal definition of vertex gluing -- an operation of vertex identification in maps that preserves genus and orientability -- we finish this section by establishing some of its properties relevant to the sequel. From
 Lemma~\ref{lem:riffle_delete_contract} to Lemma~\ref{lem:riffle_delete_dual_link_bridge}, we establish that vertex gluing commutes with deletion of a dual link or bridge. 
 Lemma~\ref{lem:riffle_reorder_condition} tells us that permuting the order of a sequence of vertex gluings gives another sequence of vertex gluings, which by Lemma~\ref{lem:riffle_reorder} results in the same map. 

\begin{lemma}\label{lem:riffle_delete_contract}
If  $e\equiv (\;\; c \;\;\;\alpha_0\alpha_2 c\;\;)\; (\;\; \alpha_ 0 c \;\;\; \alpha_2 c\;\;)$ is an edge of $M\equiv(C,\alpha_0, \alpha_1, \alpha_2)$ 
and $a,b$ are crosses of $M$ such that $\{a,b,\alpha_1 a,\alpha_1 b\}\cap \{c, \alpha_0 c,\alpha_2 c,\alpha_0\alpha_2 c\}=\emptyset$, then 
\[(M\backslash e)^{(\,a\: b\,)}=M^{(\, a\: b\,)}\backslash e,\quad \mbox{and} \quad (M/ e)^{(\,a\: b\,)}=M^{(\, a\: b\,)}/e.\]
\end{lemma}
\begin{proof} We prove the identity for deletion of $e$; the proof for its contraction is similar. 
The map $M^{(\, a\: b\,)}\equiv (C, \alpha_0, \widetilde{\alpha}_1,\alpha_2)$ has involution $\widetilde{\alpha}_1$ equal to $\alpha_1$ conjugated by $(\;\; a \;\;\; b\;\;)$.
The map $M\backslash e\equiv (C',\alpha_0',\alpha_1',\alpha_2')$ has $C'=C\setminus\{c,\alpha_0 c,\alpha_2 c,\alpha_0\alpha_2 c\}$, and $\alpha_0'=\alpha_0$ and $\alpha_2'=\alpha_2$ on $C'$. 
To obtain $\alpha_1'$ from $\alpha_1$ in terms of its disjoint cycle decomposition (product of transpositions) consider the (at most four) transpositions of $\alpha_1$ containing a cross from edge~$e$. We first describe how the transpositions of $\alpha_1'$ are obtained from those of $\alpha_1$; after this we describe the conditions under which pairs of empty cycles are added to $\alpha_1'$ (representing new isolated vertices). 

If the crosses of $e$ are in four distinct transpositions of $\alpha_1$, after removing them, merge together into a single transposition the crosses that were paired with $c$ and $\alpha_2 c$, and likewise those that were paired with $\alpha_0 c$ and $\alpha_0\alpha_2 c$.
If there are three transpositions containing crosses from $e$, then there is one transposition containing two crosses of $e$, and two containing just one cross from $e$; after removing the crosses of $e$, merge together the remaining non-empty cycles to make a single transposition in place of the orginal two. 
 Finally, if there are just two transpositions containing crosses from $e$, then the transpositions of $\alpha_1'$ are simply obtained by removing these transpositions. 
 
The conditions under which pairs of empty cycles are added to $\alpha_1'$ are as follows.  If $\alpha_1 c=\alpha_2 c$, then a pair of empty cycles is added to represent the isolated vertex that results; likewise, if  $\alpha_1\alpha_0c=\alpha_0\alpha_2 c$  then a pair of empty cycles is added to represent the isolated vertex that results. (This corresponds to deleting a link with endpoint(s) of degree one.) 
 If $\alpha_1 c=\alpha_0 c$ and $\alpha_1 \alpha_2 c=\alpha_2\alpha_0 c$, then a single pair of empty cycles is added. (This corresponds to deleting a non-twisted loop, and a single new isolated vertex is obtained.) If $\alpha_1 c=\alpha_2 \alpha_0 c$ and $\alpha_1 \alpha_2 c=\alpha_0 c$, then a single pair of empty cycles is added. (This corresponds to deleting a twisted loop, and a single new isolated vertex is obtained.)

We thus see that the involution $\alpha_1'$ of $M\backslash e$ differs from $\alpha_1$ only on $\{c, \alpha_0 c,\alpha_2 c,\alpha_0\alpha_2 c\}\cup \{\alpha_1 c, \alpha_1\alpha_0 c,\alpha_1\alpha_2 c,\alpha_1\alpha_0\alpha_2 c\}$.

Having now seen how deletion of an edge is defined in terms of the involution $\alpha_1$,\footnote{For contraction, the effect on $\alpha_1$ is the same as described for deletion with the roles of $\alpha_0$ and $\alpha_2$ switched. (Contraction produces isolated vertices for non-twisted loops, where $\alpha_1 c=\alpha_0 c$ or  $\alpha_1\alpha_2c=\alpha_0\alpha_2 c$.)} we observe that the hypothesis  that  $\{a,b,\alpha_1 a,\alpha_1 b\}\cap \{c, \alpha_0 c,\alpha_2 c,\alpha_0\alpha_2 c\}=\emptyset$ implies further
that $\{a,b,\alpha_1 a,\alpha_1 b\}\cap \{\alpha_1 c, \alpha_1\alpha_0 c,\alpha_1\alpha_2 c,\alpha_1\alpha_0\alpha_2 c\}=\emptyset$  (by applying the involution $\alpha_1$ to the two sets in the intersection).
As deleting $e$ and riffling $a,b$ thus change $\alpha_1$ on disjoint sets of crosses, and they both preserve $\alpha_0$ and $\alpha_2$ on $C'$, these operations commute.
\end{proof}
\begin{corollary}\label{cor:riffle_deletion_dual_link_bridge}
If $e\equiv (\;\; c \;\;\;\alpha_0\alpha_2 c\;\;)\: (\;\; \alpha_ 0 c \;\;\; \alpha_2 c\;\;)$ is a dual link or  bridge  of $M\equiv(C,\allowbreak\alpha_0,\allowbreak \alpha_1,\allowbreak \alpha_2)$ and $a,b$ are crosses of $M$ such that $\{a,b,\alpha_1 a,\alpha_1 b\}\cap \{c, \alpha_0 c,\alpha_2 c,\alpha_0\alpha_2 c\}=\emptyset$, then the vertex gluing of $M$ represented by riffling $a,b$ is also a vertex gluing of $M\backslash e$, again represented by riffling $a,b$. 
\end{corollary}
\begin{proof}
As $(M\backslash e)^{(\,a\: b\,)}=M^{(\, a\: b\,)}\backslash e$ by Lemma~\ref{lem:riffle_delete_contract}, we only need to check that crosses $a,b$ incident with distinct vertices and either coincident with a common face in $M$ or in different connected components of $M$ are also coincident with a common face in $M\backslash e$ or lie in different connected components of $M\backslash e$.

The effect of deleting a dual link $e$ incident with faces $x'$ and $y'$ in $M$ is to merge $x'$ and $y'$ into a single face $z'$. If $a,b$ are coincident with a common face $z$ of $M$, then the same is true in $M\backslash e$ (the same face $z$ if $z\not\in\{x',y'\}$, the merging of $x'$ and $y'$ if $z\in\{x',y'\}$). If $a,b$ belong to different connected components of $M$ and are incident with faces $x,y$ in $M$, then they belong to different connected components in $M\backslash e$ and are incident with the same faces $x,y$ in $M\backslash e$ as they are in $M$. 

Deleting a bridge $e$ incident with face $z'$ in $M$ splits $z'$ into two faces $x', y'$ belonging to different connected components. If $a,b$ are coincident with a common face $z$ of $M$, then  the same is true in $M\backslash e$ when $z\neq z'$;  if $z=z'$, crosses $a,b$ are either coincident with common face $x'$ or with common face $y'$ in $M\backslash e$, or in different connected components incident with faces $x'$ and $y'$. If $a,b$ belong to different connected components of $M$ and are incident with faces $x,y$ in $M$, then they belong to different connected components in $M\backslash e$ and are incident with the same faces $x,y$ in $M\backslash e$ as they are in $M$ unless $\{x,y\}\cap\{z\}\neq \emptyset$, in which case the face in the  component of $M$ containing $e$, say $x=z$ containing $a$, is split into two faces, one of which contains~$a$ (however, the statement regarding the fact that $a,b$ belong to different connected components of $M\setminus e$ remains valid).
\end{proof}



 Corollary~\ref{cor:riffle_deletion_dual_link_bridge} can be extended to Lemma~\ref{lem:riffle_delete_dual_link_bridge} below  to accommodate deletion of any dual link or bridge, which will allow us in Corollary~\ref{cor:riffle_spanning_submap} to induce, from a vertex gluing of $M$, a corresponding vertex gluing of a submap of the same genus and orientability.
\begin{lemma}\label{lem:riffle_delete_dual_link_bridge}
Let $e$ be either a dual link or a bridge with no endpoint of degree one of a map $M\equiv(C,\alpha_0, \alpha_1, \alpha_2)$. Let $a,b$ be crosses of $M$ 
incident with distinct vertices and either coincident with a common face or in different connected components of $M$. Then 
there are uniquely defined crosses $a', b'\in C$ 
such that $(M^{(\, a\: b\,)})\backslash e=(M\backslash e)^{(\, a'\: b'\,)}$. In particular, to the vertex gluing of $M$ represented by riffling $a,b$ corresponds a vertex gluing of $M\backslash e$ represented by riffling $a',b'$.
\end{lemma}
\begin{proof}
Let $e\equiv (\;\; c \;\;\;\alpha_0\alpha_2 c\;\;)\: (\;\; \alpha_ 0 c \;\;\; \alpha_2 c\;\;)$. When $\{a,b,\alpha_1 a,\alpha_1 b\}\cap \{c, \alpha_0 c,\alpha_2 c,\alpha_0\alpha_2 c\}=\emptyset$, by Corollary~\ref{cor:riffle_deletion_dual_link_bridge}, we can take $a'=a, b'=b$.

If $e$ is a dual link then $a$ and $b$ cannot both belong to $\{c, \alpha_0 c,\alpha_2 c,\alpha_0\alpha_2 c\}$ because they are coincident with a common face (preventing $b=\alpha_0 a$) and, as $e$ is a dual link, $c,\alpha_0 c$ belong to a different face to $\alpha_2 c, \alpha_0\alpha_2 c$ (preventing $b$ from being one of these crosses). 
Suppose then, without loss of generality, that $a=c$ is incident with face $x$ of $M$ and $b\not\in \{c, \alpha_0 c,\alpha_2 c,\alpha_0\alpha_2 c\}$ is either a cross coincident with $a$ on face $x$ or a cross belonging to a different connected component of $M$. 
If we set $a'=\alpha_1\alpha_2 a=\tau a$  then $(M^{(\, a\: b\,)})\backslash e=(M\backslash e)^{(\, a'\: b\,)}$. 

Now suppose that $e$ is a bridge incident with face $z$ and that $a=c$. Possibly $b=\alpha_0\alpha_2 c$, a cross coincident with $a$ on face $z$, but in any event $b\not\in\{\alpha_0 c, \alpha_2 c\}$. 
When $b\neq \alpha_0\alpha_2 c$, the same choice of $a'=\alpha_1\alpha_2 a$ gives $(M^{(\, a\: b\,)})\backslash e=(M\backslash e)^{(\, a'\: b\,)}$, unless $a$ is incident with an endpoint of $e$ of degree one. When $b=\alpha_0\alpha_2 c$, crosses $a'=\alpha_1\alpha_2 a$ and $b'=\alpha_1\alpha_2 b$ belong to different connected components of $M\backslash e$, the edge $e$ is a loop in $M^{(\, a\: b\,)}$, and we have  $(M^{(\, a\: b\,)})\backslash e=(M\backslash e)^{(\, a'\: b'\,)}=M/e$.

In all cases, the cross $a'$ is equal to $a$ when $a\not\in\{c, \alpha_0 c,\alpha_2 c,\alpha_0\alpha_2 c\}$ and to $\tau a$ when $a\in \{c, \alpha_0 c,\alpha_2 c,\alpha_0\alpha_2 c\}$; and likewise $b'$ is set equal to $b$ or $\tau b$ according as it does not or does belong to the crosses of $e$.    
\end{proof}

 The following corollary gives a way to restrict a sequence of vertex gluings to a spanning submap, and will be used in Section~\ref{sec:defining_map_homomorphism} to define the restriction of a map homomorphism to a submap. 

\begin{corollary}\label{cor:riffle_spanning_submap}
Let $M\backslash A$ be a spanning submap of $M$ such that $\sg(M\backslash A)=\sg(M)$, and let $a,b$ be crosses of $M$ incident with distinct vertices $u,v$ and either coincident with a common face or in different connected components. If neither $u$ nor $v$ are isolated vertices in $M\backslash A$, then there are crosses $a',b'$ of $M\backslash A$ such that 
$M^{(\, a\: b\,)}\backslash A=(M\backslash A)^{(\, a'\: b'\,)}$.
\end{corollary}
\begin{proof} Lemma~\ref{lem:riffle_delete_dual_link_bridge} and induction yield the result as any spanning submap of the same signed genus as $M$ is,  by Lemma~\ref{l.link_equiv}(i), obtained by a sequence of deletions of dual links or bridges.  \end{proof}

 The following lemma records the effect of switching the order  in which two riffles are composed. 
 \begin{lemma}\label{lem:riffle_reorder}
 For two pairs of distinct crosses $a,b$ and $a',b'$ of a map $M\equiv(C,\alpha_0, \alpha_1, \alpha_2)$, the following statements hold:

\begin{enumerate}[label=(\roman*)]
    \item 
    \label{en:riffle_reorder_1}
If $a,b,a',b'$ are distinct,  then
 $(M^{(\, a\:b\,)})^{(\, a'\:b'\,)}=(M^{(\, a'\:b'\,)})^{(\, a\:b\,)}$,
 and the same identity holds with $(\;\; \alpha_1 a \;\;\; \alpha_1 b\;\;)$ in place of $(\;\; a \;\;\; b\;\;)$ or with $(\;\; \alpha_1 a' \;\;\; \alpha_1 b'\;\;)$  in place of $(\;\; a' \;\;\; b'\;\;)$. 
\item 
\label{en:riffle_reorder_2}
If $a=a'$ and $b\neq b'$, then 
$(M^{(\, a\:b\,)})^{(\, a\:b'\,)}=(M^{(\, a\:b'\,)})^{(\, \alpha_1 a\:\;\alpha_1b\,)}$,
and the same identity holds with $(\;\; \alpha_1 a \;\;\;\alpha_1 b\;\;)$  in place of $(\;\; a \;\;\;  b\;\;)$ or with $(\;\; \alpha_1 a \;\;\; \alpha_1 b'\;\;)$  in place of $(\;\;  a \;\;\;  b'\;\;)$.
\item 
\label{en:riffle_reorder_3}
If $a=a'$ and $b=b'$, then $(M^{(\, a\: b\,)})^{(\, a\: b\,)}=M$,
and the same identity holds with either transposition $(\;\;  a \;\;\; b\;\;)$ (possibly both) replaced by $(\;\; \alpha_1 a \;\;\; \alpha_1 b\;\;)$.
\end{enumerate}
\end{lemma}
\begin{proof} 



Suppose first that no cross among $a',b',\alpha_1 a',\alpha_1 b'$ is a cross among $a,b,\alpha_1 a, \alpha_1 b$. Then, $(\;\; a \;\;\; b\;\;)\;\;(\;\; a' \;\;\; b'\;\;)=(\;\; a' \;\;\; b'\;\;)\;\;(\;\; a \;\;\; b\;\;)$, and thus conjugating $\alpha_1$ by $(\;\; a \;\;\; b\;\;)$ and then by $(\;\; a' \;\;\; b'\;\;)$ gives the same permutation as when conjugating first by $(\;\; a' \;\;\; b'\;\;)$ and then by  $(\;\; a \;\;\; b\;\;)$. Hence, $(M^{(\, a\:b\,)})^{(\, a'\:b'\,)}=(M^{(\, a'\:b'\,)})^{(\, a\:b\,)}$, which proves~\ref{en:riffle_reorder_1}.
The same argument shows that the identity holds with $\alpha_1 a, \alpha_1 b$ replacing $a,b$ or with $\alpha_1 a', \alpha_1 b'$  replacing $a',b'$.

Suppose now that $a= a'$ and $b\neq b'$.  
 The involution 
 \begin{equation}\label{eq:zerosecond}
\alpha_1=(\;\; a \;\;\; \alpha_1 a\;\;)\;\;(\;\; b \;\;\; \alpha_1 b\;\;)\;\;(\;\; b' \;\;\; \alpha_1 b'\;\;)\:\dots\end{equation}
after conjugating by $(\;\; a \;\;\; b'\;\;)$ or by $(\;\; \alpha_1 a \;\;\; \alpha_1 b'\;\;)$  becomes 
$(\;\; a \;\;\; \alpha_1 b'\;\;)\;\;(\;\; b \;\;\; \alpha_1 b\;\;)\;\;(\;\; b' \;\;\; \alpha_1 a\;\;)\:\dots$
Then, conjugating by $(\;\; \alpha_1 a \;\;\; \alpha_1 b\;\;)$ gives 
\begin{equation}\label{eq:firstsecond}(\;\; a \;\;\; \alpha_1 b'\;\;)\;\;(\;\; b \;\;\; \alpha_1 a\;\;)\;\;(\;\; b' \;\;\; \alpha_1 b\;\;)\:\dots\end{equation}
\noindent Conjugating now the involution $\alpha_1$ in~\eqref{eq:zerosecond} by $(\;\; a \;\;\; b\;\;)$ or by $(\;\; \alpha_1 a \;\;\; \alpha_1 b\;\;)$  gives 
\[(\;\; b \;\;\; \alpha_1 a\;\;)\;\;(\;\; a \;\;\; \alpha_1 b\;\;)\;\;(\;\; b' \;\;\; \alpha_1 b'\;\;)\:\dots,\]
and then conjugating by $(\;\; a \;\;\;  b'\;\;)$  gives 
$(\;\; b \;\;\; \alpha_1 a\;\;)\;\;(\;\; b' \;\;\; \alpha_1 b\;\;)\;\;(\;\; a \;\;\; \alpha_1 b'\;\;)\:\dots,$
which is the same involution  as~\eqref{eq:firstsecond}. This establishes~\ref{en:riffle_reorder_2}. 

Finally, when $a=a'$ and $b=b'$, conjugating $\alpha_1$ by $(\;\; a \;\;\; b\;\;)$ and then again by $(\;\; a \;\;\; b\;\;)$ is to conjugate by the product of transpositions  $(\;\; a \;\;\; b\;\;)\;\; (\;\; a \;\;\; b\;\;)$, equal to the identity permutation. This yields~\ref{en:riffle_reorder_3}. 
\end{proof}


Before stating the next lemma, we define a notion needed for its proof and used further in Section~\ref{sec:cutting} below.     
Distinct crosses $a,b, a', b'$  are {\em interlacing} in a cycle of crosses $\gamma$ 
 if  $\gamma=(\;\;a \;\;\; \mathbf x \;\;\; a' \;\;\; \allowbreak \mathbf y \;\;\; \allowbreak b \;\;\; \allowbreak \mathbf z \;\;\; b' \;\;\; \mathbf w\;\;)$  or $\gamma=(\;\; a \;\;\; \mathbf x \;\;\; b' \;\;\; \mathbf y \;\;\; b \;\;\; \mathbf z \;\;\; a' \;\;\; \mathbf w\;\;)$, where $\mathbf x, \mathbf y, \mathbf z$ and~$\mathbf w$ are (possibly empty) sequences of crosses.

\begin{lemma}\label{lem:riffle_reorder_condition}
Let $a,b$ and $a',b'$ be two pairs of distinct crosses of a map $M\equiv(C,\alpha_0, \alpha_1, \alpha_2)$, and let $M^{(\,a\:b\,)}$ be a vertex gluing of $M$ and $(M^{(\,a\:b\,)})^{(\,a'\:b'\,)}$ a vertex gluing of $M^{(\,a\:b\,)}$. Then, either
\begin{enumerate}[label=(\roman*)]
    \item \label{en:ppp.1} $M^{(\,a'\:b'\,)}$ is a vertex gluing of $M$ and $(M^{(\,a'\:b'\,)})^{(\,a\:b\,)}$ is a vertex gluing of $M^{(\,a'\:b'\,)}$, and $(M^{(\,a\:b\,)})^{(\,a'\:b'\,)}=\allowbreak (M^{(\,a'\:b'\,)})^{(\,a\:b\,)}$, or
    \item \label{en:ppp.2} $M^{(\,\alpha_1 a'\:\alpha_1 b'\,)}$ is a vertex gluing of $M$ and $(M^{(\,\alpha_1 a'\:\alpha_1 b'\,)})^{(\,a\:b\,)}$ is a vertex gluing of $M^{(\,\alpha_1 a'\:\alpha_1 b'\,)}$, and $(M^{(\,a\:b\,)})^{(\,a'\:b'\,)}=(M^{(\,\alpha_1 a'\:\alpha_1 b'\,)})^{(\,a\:b\,)}$, or
    \item \label{en:ppp.3} $M^{(\,a'\:b'\,)}$ is a vertex gluing of $M$ and $(M^{(\,a'\:b'\,)})^{(\,\alpha_1 a\:\alpha_1b\,)}$ is a vertex gluing of $M^{(\,a'\: b'\,)}$, and $(M^{(\,a\:b\,)})^{(\,a'\:b'\,)}=(M^{(\,a'\:b'\,)})^{(\,\alpha_1 a\:\alpha_1b\,)}$. 
\end{enumerate}
\end{lemma}

\begin{proof} 
The final equalities of statements \ref{en:ppp.1}, \ref{en:ppp.2} and \ref{en:ppp.3}  follow from Lemma~\ref{lem:riffle_reorder}. To establish the remainder of the lemma we shall show that if 
$a,b$ are either
\begin{enumerate}[label=(h\arabic*)]
    \item
    \label{en:st11} incident with distinct vertices 
    and coincident with a common face of $M$, or
    \item
    \label{en:st21} incident with vertices and faces in different connected components of $M$,
    \end{enumerate}
    and $a',b'$ are either
\begin{enumerate}[label=(h\arabic*')]
    \item 
    \label{en:st12} incident with distinct vertices 
    and coincident with a common face of $M^{(\,a\:b\,)}$, or
    \item 
    \label{en:st22} incident with vertices and faces in different connected components of $M^{(\,a\:b\,)}$,
    \end{enumerate}
then $a',b'$ are either
\begin{enumerate}[label=(c\arabic*)]
    \item
    \label{en:st31} incident with distinct vertices 
    and coincident with a common face of $M$, or
    \item
    \label{en:st32} incident with vertices and faces in different connected components of $M$,
    \end{enumerate}
    and $a,b$ are either
\begin{enumerate}[label=(c\arabic*')]
    \item
    \label{en:st41} incident with distinct vertices 
    and coincident with a common face of $M^{(\,a'\:b'\,)}$, or
    \item
    \label{en:st42} incident with vertices and faces in different connected components of $M^{(\,a'\:b'\,)}$,
    \end{enumerate}
\noindent with the appropriate changes for cases \ref{en:ppp.2} and \ref{en:ppp.3} of the statement: For \ref{en:ppp.2} we  replace $a'$ by $\alpha_1 a'$ and $b'$ by $\alpha_1 b'$ in \ref{en:st31}, \ref{en:st32}, \ref{en:st41} and \ref{en:st42}; 
for \ref{en:ppp.3} we  replace $a$ by $\alpha_1 a$ and $b$ by $\alpha_1 b$ in \ref{en:st31}, \ref{en:st32}, \ref{en:st41} and \ref{en:st42}.

    We take in turn the four cases according to the position of crosses $a,b$ in $M$ and the position of $a',b'$ in $M^{(\, a\: b\,)}$. We notice that cases \ref{en:ppp.2} and \ref{en:ppp.3} only occur in the first case of the proof below. 

\emph{Case \ref{en:st11} and \ref{en:st12}:}
Riffling $a,b$ in $M$ splits the face $z$ in $M$ containing $a,b$ into two faces $x,y$, the first containing $a$ and the second $b$. By \ref{en:st12}, there is a face $z'$ in $M^{(\,a\:b\,)}$ containing  $a',b'$. 

If $z'\neq x$ and $z'\neq y$, then the face $z'$ is also a face in $M$, and we obtain \ref{en:st31} (for if the vertices of $a'$ and $b'$ are distinct in $M^{(\,a\:b\,)}$, then they are also distinct in $M$ as, since we are performing a vertex gluing, there is a surjective mapping from the vertices of $M$ onto the vertices of  $M^{(\,a\:b\,)}$).
Now, $a$ and $b$ are also coincident with the same face in $M^{(\,a'\:b'\,)}$, by the assumption that $z'\neq x$ and $z'\neq y$ and \ref{en:st12}; furthermore, they are also incident with distinct vertices in $M^{(\,a'\:b'\,)}$ as we have just glued the vertex incident with $a'$ and the vertex incident with $b'$, and if the two distinct vertices incident with $a$ and $b$ in $M$ now become the same in $M^{(\,a'\:b'\,)}$, this means that they were also glued together in $M^{(\,a\:b\,)}$, thus contradicting the assumption \ref{en:st12}. This shows \ref{en:st41} as stated.

If $z'=x$ (the case $z'=y$ is argued similarly), then we conclude that $\{a,b,a',b'\}$ are all incident with the same face $z$ in $M$. Now, the condition \ref{en:st12} implies that $\{a,b\}\neq \{a',b'\}$, for otherwise, $a'$ and $b'$ would lie in two different faces, contradicting \ref{en:st12}. Consider the  cycle of crosses $\gamma=(\;\;a\;\;\;\alpha_0 a\;\;\;\phi a \;\;\; \alpha_0\phi a \;\;\; \cdots \;\;\;\alpha_1 a \;\;)$ interleaving the two permutation cycles of the face $z$ (one of them being in its order, and the other in reverse).
Assume first that $|\{a,b,a',b'\}|=4$, then the four crosses cannot be interlacing in $\gamma$ for otherwise \ref{en:st12} would not hold. Then we obtain \ref{en:st31} and \ref{en:st41}.
Now assume that $|\{a,b,a',b'\}|=3$, say without loss of generality that $a'=a$.
Then there are two cases for $\gamma$, either
$\gamma_1=(\;\;a=a'\;\;\; \mathbf{x}\;\;\; b\;\;\; \mathbf{y}\;\;\; b'\;\;\; \mathbf{z}\;\;)$  or $\gamma_2=(\;\;a=a'\;\;\; \mathbf{x}\;\;\; b'\;\;\; \mathbf{y}\;\;\; b\;\;\; \mathbf{z}\;\;)$ where $\mathbf x,\mathbf y,\mathbf z$ are non-empty sequences of crosses.\footnote{The cyclicity of $\gamma$ and the fact that $\{a',b'\}$ and $\{a,b\}$
are coincident with the face makes the relative position of $\{a',\alpha_1a\}$ and $\{b',\alpha_1b'\}$ the same within $\gamma_1$ and $\gamma_2$.}
In the first case,  writing $\mathbf{y}=(\mathbf y',\alpha_1 b')$ and $\mathbf z=(\mathbf z',\alpha_1 a')$, we have $\gamma_1=(\;\; \alpha_1 a' \;\;\; a\;\;\; \mathbf{x}\;\;\;  b\;\;\; \mathbf{y}'\;\;\; \alpha_1 b'\;\;\;b' \;\;\; \mathbf{z}'\;\;)$, so
$\{\alpha_1 a',\alpha_1 b'\}$ and $\{a,b\}$ are not interlacing, and we conclude \ref{en:st31} and  \ref{en:st41} hold with $\alpha_1 a'$ and $\alpha_1 b'$ (thus obtaining part \ref{en:ppp.2} of the statement).
In the second case,  writing $\mathbf{y}=(\mathbf y',\alpha_1 b)$ and $\mathbf z=(\mathbf z',\alpha_1 a)$, we have $\gamma_{2}=(\;\; \alpha_1 a \;\;\; a'\;\;\; \mathbf{x} \;\;\; b'\;\;\; \mathbf{y}'\;\;\;\alpha_1b\;\;\;b \;\;\; \mathbf{z}'\;\;)$, so  
$\{\alpha_1 a,\alpha_1 b\}$ and $\{a',b'\}$ are not interlacing, and we have \ref{en:st31} and  \ref{en:st41}  with $\alpha_1 a$ and $\alpha_1 b$ (thus obtaining part \ref{en:ppp.3} of the statement).


\emph{Case \ref{en:st11} and \ref{en:st22}:}
As riffling $a,b$ in $M$ preserves connected components, crosses $a',b'$ must also belong to different connected components of $M$. Thus \ref{en:st32} follows. We now show that  \ref{en:st41} holds. When riffling $a',b'$ in $M$, the only alteration to the connected component of $M$ containing $a$ and $b$ may be the addition of some crosses to one of its vertices. Thus, if $a$ and $b$ are incident with distinct vertices in $M$, they are also incident with distinct vertices in $M^{(\,a'\:b'\,)}$. The effect on the faces of this connected component  is the addition of crosses to one of its faces. Thus, if $a$ and $b$ are coincident with a single face in $M$, they are also coincident with a face in $M^{(\,a'\:b'\,)}$. This shows~\ref{en:st41}.

\emph{Case \ref{en:st21} and \ref{en:st12}:}
The riffling  of $a,b$ in $M$ merges their respective faces $x, y$ into a face $z$, with which $a,b$ are coincident in $M^{(\, a\:b\,)}$. 
If the face $z'$ of $M^{(\,a\:b\,)}$ containing $a'$ and $b'$ is not equal to $z$, then $a',b'$ remain coincident with this face in $M$, and we have \ref{en:st31} and \ref{en:st42}.
If $z'=z$ then there are two possibilities. The first is that $a',b'$ are coincident with a common face of $M$ (either $x$ or $y$), so \ref{en:st31} holds; also  \ref{en:st42} holds since gluing the vertices of $a'$ and $b'$ in one connected component does not affect the other connected component. 
The second possibility is that 
$a'$ is incident with the same face of $M$ as $a$ or $b$, say the face $x$ (but $a$ and $a'$ are not necessarily coincident), and $b'$ is incident with the face $y$. In this case \ref{en:st32} holds. Now, in $M^{(\,a\:b\,)}$, crosses $a'$ and $b'$ are incident with the same face as $a$ and $b$, and~\ref{en:st12} implies that $a'$ and $b'$ are in fact coincident with this face.
Crosses $a$ and $b$ are in the same face permutation cycle as $a'$ and $b'$ if both pairs $\{a,a'\}$ and $\{b,b'\}$ are coincident on their faces of $M$, and different ones if both these pairs are incident but not coincident on their faces of $M$. (Hypothesis~\ref{en:st12} prevents just one of these pairs being coincident and the other not.) 
Furthermore, hypothesis \ref{en:st12} also implies that $a'$ and $b'$ are incident with distinct vertices in $M^{(\,a\:b\,)}$. This shows \ref{en:st41}. 

\emph{Case \ref{en:st21} and \ref{en:st22}:}
Riffling $a$ and $b$ in $M$ merges the two connected components of $a$ and $b$ into one, while the connected components of $a'$ and $b'$ are not merged. 
Thus, the crosses $\{a,b,a',b'\}$ span at least three connected components in $M$. This shows that \ref{en:st32} and \ref{en:st42} hold.
\end{proof}

\subsection{Duplicate edges}\label{sec:duplicate_edges}

If $e$ and $f$ are distinct parallel edges of $\Gamma$, then the graphs obtained after deleting $e$ and $f$, respectively, are isomorphic via  
the mapping fixing vertices and fixing all edges except $e$ and $f$, which are swapped. 
Parallel edges are interchangeable when it comes to the existence of homomorphisms: there is a homomorphism from $\Gamma\backslash e$ to a graph $\Gamma'$ if and only if there is a homomorphism from $\Gamma\backslash f$ to $\Gamma'$, and a redundancy of one of the parallel edges in that there is a homomorphism from $\Gamma$ to $\Gamma'$ if and only if there is a homomorphism from $\Gamma\backslash e$ to $\Gamma'$. The same statements hold for homomorphisms from the graph $\Gamma'$ rather than to it. 

There are, however, non-isomorphic maps with the same underlying graph which differ only in the placement of parallel edges. For example,  a plane 4-cycle with a chord and an edge added in parallel to the chord, in one way bounding a face of degree two and a face of degree three, and in the other bounding two faces of degree three. 

In order to carry over to maps the property of parallel edges in graphs being indistinguishable when it comes to homomorphisms, we need to add a topological constraint. 
\begin{definition}\label{def:duplicate_edges}
Two edges of a map $M$ are {\em duplicate} if they are incident with a common face of degree two. 
\end{definition}

  Figure \ref{fig:map_parallel} illustrates different examples of duplicate and non-duplicate edges. One can see that duplicate edges can be merged into one while remaining in the surface of the map (much as pairs of vertices on a common face can be glued while remaining in the surface of the map). 

\begin{figure}[htb]
\centering
\includegraphics[width=0.9\textwidth]{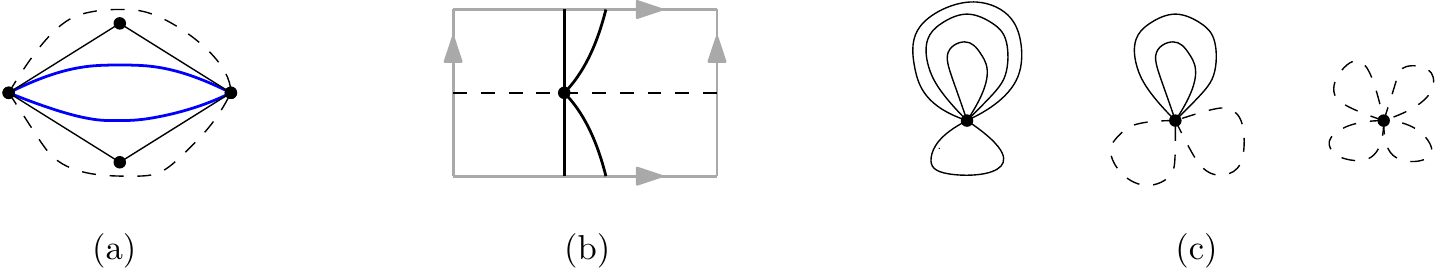}
\caption{\small (a) Two pairs (thicker blue and dashed) of duplicate edges in a plane map, (b) two duplicate loops in the torus (thicker edges), and a loop that is not duplicate with the others (dashed), (c) three plane embeddings of four loops on a vertex, the ones dashed are not duplicate with any other loop, the non-dashed ones are duplicate loops. }
 \label{fig:map_parallel}
\end{figure}

\begin{definition}\label{def:edge_gluing}
 Let $M\equiv (C,\alpha_0,\alpha_1, \alpha_2)$ be a map containing crosses $a$ and $b=\phi a$  such that $e \equiv (\;\; a \;\;\; \alpha_0\alpha_2 a\;\;)\;(\;\; \alpha_0 a \;\;\; \alpha_2 a\;\;)$ and $f\equiv (\;\; b \;\;\;\alpha_0\alpha_2 b\;\;)\;(\;\;\alpha_0 b \;\;\; \alpha_2 b\;\;)$ are duplicate edges, bounding the face $(\;\; a \;\;\; b\;\;)\;(\;\; \alpha_0 a\;\;\;\alpha_0 b\;\;)$ of degree two. The map obtained by {\em gluing} $e$ and $f$ is $M^{[\,a\;b\,]}\equiv(\overbar{C},\overbar{\alpha}_0,\overbar{\alpha}_1,\overbar{\alpha}_2)$, in which 
\begin{itemize}
\item $\overbar{C}=C\backslash\{a,\alpha_0 a,b,\alpha_0 b\}$;
    \item $\overbar{\alpha}_0=\alpha_0$ and $\overbar{\alpha}_1=\alpha_1$ on $\overbar{C}$;  \item $\overbar{\alpha}_2=\alpha_2$ on $\overbar{C}\backslash\{\alpha_2a,\alpha_0\alpha_2 a,\alpha_2 b,\alpha_0\alpha_2 b\}$, and
    \[\overbar{\alpha}_2(\alpha_2 a)=\alpha_0\alpha_2 b, \quad \overbar{\alpha}_2(\alpha_0\alpha_2 a)=\alpha_2 b.\]
   
\end{itemize}
Equivalently, we obtain $M^{[\, a\; b\,]}$ from $M$ by preserving the three involutions  and setting $\alpha_0\alpha_2 b\leftarrow a$ and $\alpha_0\alpha_2 a\leftarrow b$ (so that $\alpha_2 b\leftarrow \alpha_0 a$ and $\alpha_2 a\leftarrow \alpha_0 b$). 
\end{definition}

\begin{remark}\label{lem:duplicate_iso} The map $M^{[\,a\;b\,]}$ given by Definition~\ref{def:edge_gluing} is isomorphic to $M\backslash f$ upon setting $\alpha_0\alpha_2 b\leftarrow a$, $\alpha_2 b\leftarrow\alpha_0 a$, and to $M\backslash e$ upon setting $\alpha_0\alpha_2 a\leftarrow b$, $\alpha_2 a\leftarrow\alpha_0 b$.  
\end{remark}

Definition~\ref{def:edge_gluing} extends to pairs of edges $e,f$ that lie on a common face of larger degree, or that belong to different connected components, but we shall only apply it to pairs of edges that are duplicate: the more  general operation of gluing $e$ and $f$ can be defined as a composition of two vertex gluings (if $e$ and $f$ share no endpoint) or one vertex gluing (if $e$ and $f$ share a single endpoint) and a duplicate edge gluing. 

Edges that are duplicate  in $M$ are incident with a common vertex of degree two in $M^*$. The maximal induced paths in $M^*$ have edge sets equal to the equivalence classes of the relation defined on $M$ by taking the transitive closure of the relation of being duplicate: this equivalence relation refines that of parallel edges in the underlying graph of $M$.   

\begin{lemma}\label{lem:duplicate_glue_del}
If $e,f$ are duplicate edges in $M$, then 
$\sg(M\backslash e)=\sg(M\backslash f)=\sg(M)$.
\end{lemma}
\begin{proof}
The first equality  follows by Remark~\ref{lem:duplicate_iso}, and the second equality by Lemma~\ref{l.link_equiv}~(i) as duplicate edges are dual links. 
\end{proof}

  The following lemma tells us that duplicate edges remain duplicate under vertex gluing  (except possibly when the vertex gluing involves the face bounded by the duplicate edges), and that gluing duplicates can be done before or after a vertex gluing without changing the resulting map. 
\begin{lemma} \label{lem:vertex_edge_gluing_order}
 Let $M\equiv (C,\alpha_0,\alpha_1, \alpha_2)$ be a map containing crosses $a$ and $b=\phi a$  such that $e \equiv (\;\; a \;\;\; \alpha_0\alpha_2 a\;\;)\;(\;\; \alpha_0 a \;\;\; \alpha_2 a\;\;)$ and $f\equiv (\;\; b \;\;\;\alpha_0\alpha_2 b\;\;)\;(\;\;\alpha_0 b \;\;\; \alpha_2 b\;\;)$ are duplicate edges. 

Suppose that 
 $a',b'$ are crosses of $M$ such that  $\{a',b',\alpha_1 a',\alpha_1 b'\}$ is disjoint from $\{a,\alpha_0 a, b,\alpha_0 b\}$. Then $e$ and $f$ are duplicate in $M^{(\, a'\: b'\,)}$ and
\[(M^{[\,a\;b\,]})^{(\,a'\: b'\,)}=(M^{(\, a'\: b'\,)})^{[\,a\;b\,]}.\]
\end{lemma}
\begin{proof}
By the disjointedness condition on crosses $a',b',\alpha_1 a',\alpha_1 b'$, riffling $a',b'$ fixes $a,b,\alpha_1a=\alpha_0 b$ and $\alpha_1 b=\alpha_0 a$, 
and therefore the face of degree two $(\;\; a \;\;\; b\;\;)\; (\;\; \alpha_0 a \;\;\; \alpha_0 b\;\;)$ is unchanged, i.e. the edges $e$ and $f$ remain duplicate  in $M^{(\, a'\: b'\,)}$.

Riffling $a',b'$ fixes $\alpha_0$ and $\alpha_2$, and only changes $\alpha_1$ on the four crosses $a',b',\alpha_1 a', \alpha_1 b'$, which by asssumption all belong to $\overbar{C}=C\setminus\{a,\alpha_0 a,b,\alpha_0 b\}$. The map $M^{[\, a\; b\,]}$ is obtained by restricting  $\alpha_0, \alpha_1$ to $\overbar{C}$, and restricting $\alpha_2$ to $\overbar{C}$ while changing its values on $\alpha_2 a, \alpha_0\alpha_2 a, \alpha_2 b, \alpha_0\alpha_2 b$ from $a,\alpha_0 a, b,\alpha_0 b$ to $\alpha_0\alpha_2 b, \alpha_2 b,\alpha_0\alpha_2 a, \alpha_2 a$, respectively. 
The involutions $\alpha_0,\alpha_1,\alpha_2$ are thus by riffling/duplicate edge gluing affected as follows: 
\begin{center}
\begin{tabular}{|l|c|c|}
\hline
 & riffling & duplicate edge gluing \\
 \hline
 $\alpha_0$ & fix & restrict to $\overbar{C}$\\
 $\alpha_1$ & change on $\{a',\alpha_1 a', b', \alpha_1 b'\}$ & restrict to $\overbar{C}$\\
 $\alpha_2$ & fix & restrict to $\overbar{C}$ and change on $\{\alpha_2 a, \alpha_0\alpha_2 a, \alpha_2 b, \alpha_0\alpha_2 b\}$.\\
 \hline
\end{tabular}
\end{center}

For each involution, duplicate edge gluing followed by riffling has the same effect as riffling followed by duplicate edge gluing.   (For $\alpha_1$, the condition that the four crosses $a',b',\alpha_1 a', \alpha_1 b'$ all belong to $\overbar{C}=C\setminus\{a,\alpha_0 a,b,\alpha_0 b\}$ is needed so that the restriction to $\overbar{C}$ can be followed by the changes on $\{a',b',\alpha_1 a', \alpha_1 b'\}$.) 
\end{proof}

 In a similar way to how we ``remove parallel edges'' of a graph while maintaining at least one edge in each parallel class, we can glue pairs of duplicate edges in a map iteratively until no further duplicates remain. Duplicate edges $e,f$ of $M$ are interchangeable in the sense that there is a isomorphism from $M\backslash e$ to $M\backslash f$ that fixes all vertices and all edges apart from $e$ or~$f$.


\subsection{Homomorphisms}\label{sec:defining_map_homomorphism}

We have seen how a sequence of vertex gluings followed by a sequence of duplicate edge gluings takes one map onto another map of the same signed genus: the mapping induced on the underlying graphs of the maps is an epimorphism (surjective graph homomorphism). This motivates, in particular, the following definition of an epimorphism between maps.

\begin{definition}\label{def:epimorphism_gluing}
 Let $M$ and $N$ be maps of the same signed genus. 
 
 An \emph{epimorphism} $n:M\twoheadrightarrow N$ from $M$ onto $N$ is an isomorphism of $N$ to a map obtained from $M$ by performing a sequence of vertex gluings  and a sequence of duplicate edge gluings. 

A {\em monomorphism} from $M$ into $N$ is an isomorphism from $M$ to a submap of $N$.

\end{definition}

  In an epimorphism either sequence -- of vertex gluings or of duplicate edge gluings --  may be empty. Notice that Lemmas~\ref{lem:riffle_reorder}, \ref{lem:riffle_reorder_condition} and \ref{lem:vertex_edge_gluing_order} together imply that vertex gluings can be done first, followed by duplicate edge gluings, and the order within each sequence does not matter (that edge gluings commute with each other  easily follows by interpreting it as edge deletion.) 


Before giving our definition of map homomorphism, we record the following proposition, which will be used in Section~\ref{sec:cores}.  

\begin{proposition}\label{prop:epi_mono_iso}
Let $M, M'$ be two maps such that there exist epimorphisms $M\twoheadrightarrow M'$ and $M'\twoheadrightarrow M$ (resp,. monomorphisms $M\rightarrowtail M'$ and $M'\rightarrowtail M$). Then, $M$ and $M'$ are isomorphic.
\end{proposition}
\begin{proof}
If there are epimorphisms both ways between $M$ and $M'$, then there can be neither vertex gluings (which reduce the number of vertices) nor duplicate edge gluings (which reduce the number of edges). Since $M'$ is thus obtained from $M$ by empty sequences of vertex gluings and duplicate edge gluings, there is an isomorphism between $M'$ and $M$. 

If there are monomorphisms both ways between $M$ and $M'$, there are isomorphisms between $M$ and a submap of $M'$ and between $M'$ and a submap of $M$. 
As the maps are finite there is thus an isomorphism between $M$ and $M'$. 
\end{proof}

\begin{definition} \label{def:maphom_full} Let $M,M'$ be maps of the same signed genus. 
A \emph{map homomorphism} from $M$ to $M'$ is a composition  of an epimorphism  from $M$ onto a submap $N$ of $M'$ and a monomorphism from $N$ into $M'$.
In symbols, a map homomorphism $h:M\to M'$ is the composition $m\circ n$ of an epimorphism 
$n:M\twoheadrightarrow N$ and a monomorphism $m:N\rightarrowtail M'$.  
\end{definition}
  A homomorphism from one map to another gives a graph homomorphism between their underlying graphs: a  map homomorphism can thus be regarded as a graph homomorphism that satisfies further topological constraints specified by an embedding of the graph as a map.

A homomorphism $h:M\to M'$ given by the composition $h=m\circ n$, has {\em image},  denoted by $h(M)$, equal to the image of its constituent epimorphism $n$; i.e., $h(M)=N$. Abusing notation a little further, a function on crosses is associated with $h$, and we let $h(c)$ denote the cross of $M'$ to which cross $c$ of $M$ is sent by $h$. When the map homomorphism $h$ is given solely by a sequence of vertex gluings, as represented by riffles, $h$ fixes each cross (as a map on crosses); it is the involution $\alpha_1$ that is altered. When the map homomorphism $h$ consists solely of a single duplicate edge gluing, then it sends four crosses to four other crosses, as described in Definition~\ref{def:edge_gluing}. Further, the isomorphism defining the monomorphism from $h(M)$ into $M'$ consists of a bijection between crosses of $h(M)$ and a subset of crosses of $M'$.   

Associated with a homomorphism $h:M\to M'$, there is a function from vertices of $M$ to vertices of $M'$: the crosses incident with a given vertex $v$ of $M$ are sent by $h$ to crosses incident with a unique vertex $v'$ in $M'$. 
Likewise, a function is defined by $h:M\to M'$ from edges of $M$ to  edges of $M'$: the four crosses incident with an edge $e$ of $M$ are sent by $h$ to four crosses incident with an edge $e'$ in $M'$. 
These two functions on vertices and on edges define a graph homomorphism between the underlying graphs of $M$ and~$M'$.

\paragraph{\textbf{Restriction of a homomorphism.}} 
Lemmas~\ref{lem:riffle_delete_dual_link_bridge} and~\ref{cor:riffle_spanning_submap} yield a procedure for converting a sequence of vertex gluings applied to $M$ into a sequence of vertex gluings applied to a spanning submap $M\backslash A$ of the same genus and orientability as $M$.

\begin{definition}\label{def:restriction_hom_sequential}
Let $M$ be a map with vertex permutation $\tau$, and $M\backslash A$ a spanning submap of $M$ of the same signed genus. Let
 $h=m\circ n$ be a map homomorphism from $M$ to $M'$. 
The {\em restriction} of $h$ to  $M\backslash A$, denoted by  $h_{|M\backslash A}:M\backslash A\to M'$, is given by restricting $n$ and $m$ as follows:

\,

  Restriction of $n$:  first, modify its sequence of vertex gluings to vertex gluings of $M\backslash A$ by replacing a vertex gluing represented by a riffle of crosses $a,b$ in $M$ incident with vertices $u,v$ by

\begin{itemize}
    \item
 a vertex gluing of $M\backslash A$ represented by a riffle of $\tau^i a,\tau^j b$ in $M\backslash A$ (where $i,j\geq 0$ are minimal with respect to $\tau^i a, \tau^j b$ being crosses of $M\backslash A$, if such $i,j$ exist); 
\item deleting isolated vertex $u$ (if $\tau^i a$ is not a cross of $M\backslash A$ for any $i$, but $\tau^j b$ is a cross of $M\backslash A$ for some $j$),  or isolated vertex $v$ (if $\tau^jb$ is not a cross of $M\backslash A$ for any $j$, but $\tau^i a$ is a cross of $M\backslash A$ for some $i$);  
\item deleting one of the isolated vertices $u$ and $v$ (if $\tau^i a, \tau^jb$ are not crosses of $M\backslash A$ for any $i,j$).
\end{itemize}
Second, modify its sequence of duplicate edge gluings by removing any gluings of duplicate edges, one or both of which belong to~$A$. 

\,

  Restriction of $m$: we restrict the isomorphism that defines $m$ to the image of the submap $M\backslash A$ under the restricted epimorphism~$n$. 
\end{definition}
  The restriction $h_{|M\backslash A}$ involves gluing the same vertex pairs as in the sequence of vertex gluings defining the constituent epimorphism of~$h$ (what may differ are the faces merged or split by the vertex gluings): the graph homomorphism corresponding to $h$ from the underlying graph of $M$ to that of $M'$ is restricted to the underlying graph of $M\backslash A$.   

Having introduced the restriction to a spanning submap $M\backslash A$, we next define the restriction of a homomorphism $h:M\to M'$ to a submap   
$N=(M\backslash A)-U$, where $U$ is a subset of isolated vertices of $M\backslash A$. To obtain $h_{|N}$ we just perform the restriction $h_{|M\backslash A}$ with the modification that in the vertex gluing sequence, when there is a choice between deleting a vertex in $U$ and one outside $U$, the vertex in $U$ is deleted, and after which any remaining vertices in $U$ are deleted. If $N$ has the same signed genus as $M$, then so does $M\backslash A$, so we have the following result. 
\begin{proposition}\label{prop:restriction}
Let $N$ be a submap of a map $M$ of the same signed genus as $M$, and let $h:M\to M'$ be a homomorphism from $M$ to a map $M'$. Then, the restriction $h_{|N}:N\to M'$ is a homomorphism.     
\end{proposition}

 
 In light of the preceeding proposition, one may ask if it is possible to define a restriction $h_{|N}$ for any submap $N$ even if $N$ is a submap with a different signed genus as $M$; this is, in general, not possible, as Figure~\ref{fig:submap_diff_signed_genus} shows; if $N$ is a plane submap, then we are not allowed to follow the sequence of vertex gluings that $h$ demands. This is the reason for the condition that $\sg(N)= \sg(M)$.

\begin{figure}[htb]
\centering
\includegraphics[width=0.55\textwidth]{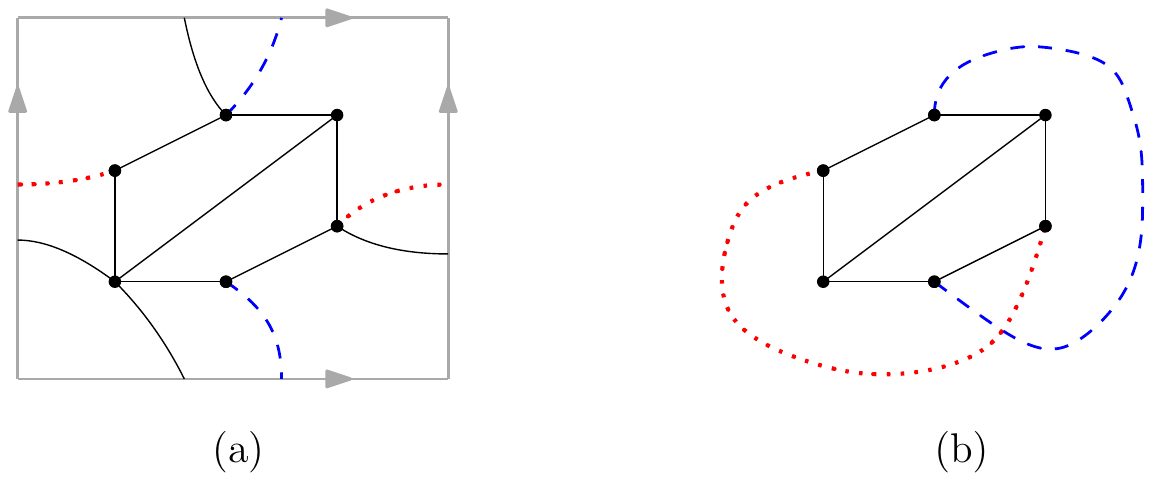}
\caption{\small (a) A map $M$ embedded in the torus, the dotted red and dashed blue lines indicate gluings of the corresponding vertices, (b) a plane submap $N$ of $M$ in which the vertex gluings of (a) cannot be performed.}
\label{fig:submap_diff_signed_genus}
\end{figure}

\paragraph{\textbf{Composition of homomorphisms.}} Let $M,M',M''$ be maps of the same signed genus, and let $h:M\to M'$ and $h':M'\to M''$ be homomorphisms.  The image $h(M)$ is a submap of $M'$ of the same signed genus which implies, by Proposition \ref{prop:restriction}, that the restriction  $h'_{|h(M)}: h(M) \to M''$  is a homomorphism. The composition of $h'_{|h(M)}$ with $h$ defines the {\em composition} of $h$ and $h'$, which is a homomorphism $h'\circ h:M\to M''$.

Decomposing $h=m\circ n$ and  $h'=m'\circ n'$ into their constituent epimorphism and monomorphism, the composition $h'\circ h$ has epimorphism $n'_{|h(M)}\circ n$, whose vertex gluing sequence is the concatenation of the sequence of vertex gluings defining $n$ and those defining the restriction of $n'$ to $h(M)$. The sequence of duplicate edge gluings is 
formed from the concatenation of the sequence of duplicate edge gluings defining $n$ and the  inverse image under $n'_{|h(M)}$ of the edges in the sequence of duplicate edge gluings defining the restriction of $n'$ to $h(M)$, where edges in the inverse image are taken in an arbitrary order. 

Since the composition of map homomorphisms is again a map homomorphism,  the existence of a homomorphism between maps defines a transitive relation; as the identity mapping from $M$ to itself is a homomorphism, this relation defines a quasi-order on maps. In Section~\ref{sec:cores_poset} a partial order is derived from this quasi-order by defining the analogue of graph cores for maps.


\section{Cores}\label{sec:cores}

A core of a graph $\Gamma$ \cite{hell_core_1992} is a  subgraph $\Delta$ with the property that, for any graph $\Gamma'$, there exists a  homomorphism from $\Gamma$ to $\Gamma'$  if and only if there exists a homomorphism from $\Delta$ to $\Gamma'$, and no proper subgraph of $\Delta$ has this property. Equivalently,  $\Delta$ is a core of $\Gamma$ if it is a minimal subgraph of $\Gamma$ (with respect to containment) for which there exists a graph homomorphism from $\Gamma$ to $\Delta$. A core of a graph is, up to isomorphism, unique.

We define cores for maps analogously to how cores are defined for graphs. 

\begin{definition}\label{def.core}
 A submap $N$ of a map $M$ is a \emph{core} of $M$ if there is a homomorphism $M\to N$ but no homomorphism $M\to N'$ for any proper submap $N'$ of $N$.
\end{definition}

 As for graphs, we say that a map is a \emph{core} when it is its own core. 
A core of a map $M$, as a homomorphic image of $M$, has the same signed genus as $M$, and has no duplicate edges.
If a graph is not connected, then its core may also not be connected (consider, for instance, the disjoint union of  $K_3$ with the graph obtained from $K_4$ by twice subdividing each  edge of a perfect matching, thus having girth four, 
which is a core as $K_3$ is not a homomorphic image of the subdivided $K_4$, and both connected components are themselves graph cores); the case is similar for maps (a plane embedding of the previous graph is a map core, as any embedding of a graph core is a map core). For this reason, we focus on the quasi-ordered set of connected maps.

\subsection{Basic properties}

For a map $M$ with core $N$, a homomorphism from $M$ to $N$ must be surjective (otherwise $M$ would be mapped to a strict submap of $N$, contradicting that $N$ is a core).
As for graphs, the core of a map is unique.

\begin{proposition}\label{p.cu}
If $N$ and $N'$ are cores of a map $M$, then $N$ is isomorphic to $N'$.
\end{proposition}
\begin{proof}
We have $\sg(N)=\sg(M)=\sg(N')$, and there are homomorphisms
$h:M\to N$ and $h':M\to N'$. By Proposition \ref{prop:restriction}, 
the restrictions $h_{|N'}: N'\to N$ and $h'_{|N}: N\to N'$ are homomorphisms. In addition, they are surjective since otherwise the compositions $h_{|N'}\circ h'$ and $h'_{|N}\circ h$ would be 
homomorphisms from $M$ to proper submaps of $N$ and $N'$, respectively,  contradicting the fact that $N$ and $N'$ are cores. The result now follows from Proposition~\ref{prop:epi_mono_iso}.
\end{proof}

  In the following proposition we establish that some basic properties of graph cores~\cite{hell_core_1992}  have their counterparts for map cores.

\begin{proposition}\label{p.prop_cores}
The following statements hold  for the core $N$ of a map $M$:
\begin{enumerate}[label=(\roman*)]
\item \label{en.core1}
$N$ is a core (i.e., $N$ is its own core).

\item \label{en.core4}
There is a homomorphism from $M$ to $N$ whose restriction to $N$ is the identity.
\item \label{en.core2} 
$N$ is an induced submap of $M$ after gluing its duplicate edges.
\item \label{en.core3} 
Given another map $M'$ with core $N'$, there is a homomorphism from $M$ to $M'$ if and only if there is a homomorphism from $N$ to  $N'$. 
\end{enumerate}
\end{proposition}

\begin{proof}
\ref{en.core1} Suppose that there is a homomorphism $h'$ from $N$ to a proper submap $N'$ of $N$. We can compose $h'$ with a homomorphism $h:M\to N$ to obtain a homomorphism $h'\circ h:M\to N'$, which contradicts the fact that $N$ is a core of $M$. Hence, $N$ is a core.

\ref{en.core4}
Let $h:M\to N$, and consider the restriction $h_{|N}:N\to N$ which, by Proposition~\ref{prop:restriction}, is a homomorphism as $\sg(N)=\sg(M)$. In addition, it is  surjective (since $N$ is a core), and so $h_{|N}$ is an isomorphism (by Proposition~\ref{prop:epi_mono_iso} using the identity map as the other epimorphism). Thus, we can compose $h$ with the inverse of $h_{|N}$ (also an isomorphism) to obtain a homomorphism from $M$ to $N$, which is the identity on~$N$.

\ref{en.core2} 
Duplicate edges in any map can be glued to obtain a proper submap as homomorphic image and so $N$ cannot contain duplicate edges.

Let $N'$ be the submap induced in $M$ by the vertices of $N$ (so $N$ is a spanning submap of $N'$). As $\sg(N)=\sg(M)$, by Lemma~\ref{lem:s_submap}, we have $\sg(N')=\sg(M)$. Let $h:M\to N$ be a homomorphism that is the identity when restricted to $N$, which we may assume by~\ref{en.core4}. Consider the restriction $h_{|N'}:N'\to N$, which is an epimorphism by Proposition~\ref{prop:restriction} and the facts that $N$ is a core of $M$ and  $h$ is the identity on $N$. Definition~\ref{def:epimorphism_gluing} 
then establishes that $N$ is isomorphic to a map obtained from $N'$ after a sequence of vertex gluings followed by a sequence of duplicate edge gluings. Since $V(N)=V(N')$,  the sequence of vertex gluings 
is empty, leaving only a sequence of duplicate edge gluings. Thus $N$ is the induced submap $N'$ with duplicate edges glued until none remain; see Figure \ref{fig.a7777} for an example.  
Note that the glued edges are duplicated in $N'$, but not in $M$.

\ref{en.core3}  Let $h: M \to N$ be a homomorphism from $M$ to its core $N$. Suppose first that there is a homomorphism $k: N \to N'$. Then the composition $k\circ h:M\to N'$ is a homomorphism, 
and $N'$ is a submap of $M'$ with $\sg(M')=\sg(N')$. Hence $k\circ h:M\to M'$ is a homomorphism. 


Suppose conversely that there is a homomorphism $\ell:M\to M'$, and let $h': M'\to N'$ be a homomorphism from $M'$ to its core $N'$. Since $N$ is a submap of $M$ with its same signed genus, there is the monomorphism $\iota:N\to M$ just seeing $N$ as a submap in $M$; then the composition $h'\circ \ell\circ\iota$ is a homomorphism between $N$ and $N'$ as all the maps have the same signed genus.
\end{proof}

\begin{figure}[htb]
\centering
\includegraphics[width=0.5\textwidth]{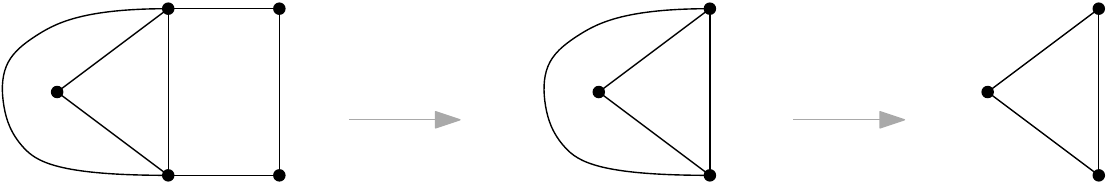}
\caption{\small A map (left) whose core is a triangle, which is not an induced submap: the duplicate edge of the submap induced on the three vertices of the triangle needs to be removed. (Notice that this edge becomes duplicate only after removing the rightmost vertices.)}
\label{fig.a7777}
\end{figure}  

 The statement of Proposition~\ref{p.prop_cores}\ref{en.core4} holds for any submap $N$ of $M$ with the same signed genus as $M$ (not necessarily a core), as we establish in the following lemma. This fact will be useful in Section \ref{sec:circuits-hom} as it ensures that there is an endomorphism of $M$ with image $N$ if and only if there is such an endomorphism fixing $N$.

\begin{lemma} \label{rem:ext_to_sub}
 Let $M$ be a map, and let $N$ be a submap of $M$ of the same signed genus. If there is an epimorphism $n:M\twoheadrightarrow N$,  then there is an epimorphism from $M$ onto $N$ whose restriction to $N$ is the identity. 
\end{lemma}
\begin{proof}
If $N$ is the core of $M$ and there is an epimorphism $n:M\twoheadrightarrow  N$, the result follows analogously to Proposition~\ref{p.prop_cores}\ref{en.core4}. Suppose then that $N$ is not a core, and let $N_0\subset N$ be the core of $N$, and so the core of $M$ (the epimorphism $n:M\twoheadrightarrow  N$ can be composed with an epimorphism $N\twoheadrightarrow N_0$). By Proposition~\ref{p.prop_cores}\ref{en.core4}, there exists an epimorphism $n_0:M\twoheadrightarrow N_0$ whose restriction to $N_0$ is the identity.
This epimorphism is defined by a sequence of vertex gluings that gives a map $N_0\cup \{e_1, \ldots, e_s\}$, where the $e_i$'s are duplicate edges that must be glued to obtain $N_0$. Each time that a vertex gluing is performed, we identify two distinct vertices generating a map that has one less vertex. In this identification we always try to keep the vertices in $N_0$ or in $N\setminus N_0$; note that the two identified vertices cannot both belong to $V(N_0)$ as $n_0$ 
is the identity on~$N_0$. The sequence of vertex gluings and the 
duplicate edge gluings can be viewed in the reverse order: we construct a map $N_1$ from $N_0$ by not performing the last vertex gluing of the sequence that removes a vertex from $N\setminus V(N_0)$ and not gluing those edges $e_i$ that are no longer duplicate with another edge in $N_1$.  This creates an epimorphism $n_1:M\twoheadrightarrow N_1$ whose restriction to $N_1$ is the identity.

Iterating this procedure, we construct a sequence of maps $N_0\subset N_1 \subset \ldots \subset N_k\subseteq N$, where $N_i$ is obtained from $N_{i-1}$ by not performing the last vertex gluing (in the sequence to obtain $N_{i-1}$) that removes a vertex from $V(N)\setminus V(N_{i-1})$, and by not gluing those edges $e_i$ that are no longer duplicate in $N_{i}$. We have $V(N_k)=V(N)$; then $N$ is obtained from $N_k$ by not gluing duplicate edges of $N$ (which may have been glued to make $N_k$ as they had duplicates). Associated with each $N_i$ there is an epimorphism $n_i:M\twoheadrightarrow N_i$ whose restriction to $N_i$ is the identity. Thus, the desired epimorphism from $M$ onto $N$ whose restriction to $N$ is the identity is easily obtained from $n_k:M\twoheadrightarrow N_k$ when go from $N_k$ to $N$.
\end{proof}

\subsection{Separating and contractible  cross-circuits
} \label{sec:cutting}

The aim of this section is to translate the topological notions of a simple (i.e. non-self-intersecting) closed curve, a separating curve and a contractible curve (or at least contractible curves of a certain type) into the language of cross permutations defining a map.

In topology, a closed curve in a surface is said to be \emph{contractible} if it can be homotopically mapped to a point; 
equivalently, the curve is contained within a region of the surface homeomorphic to an open disc. 
There have been several translations of this  notion into the context of maps, where closed curves are described by walks, which are defined combinatorially rather than topologically.

Viewing a map $M$ as an embedding of a graph $\Gamma$ into a surface $\Sigma$, Mohar and Thomassen \cite{mohar01} define a cycle of $\Gamma$ (a simple circuit, entering and leaving each vertex at most once) 
to be contractible in $M$ if the closed curve that it defines in $\Sigma$ 
is contractible.
One shortcoming of this definition for our purposes is that it only applies to cycles in $\Gamma$ and not to closed walks of $\Gamma$ that revisit vertices or edges.  This shortcoming is circumvented by the definition offered by Cabello and Mohar \cite{CM07}, who define a closed walk of $\Gamma$ 
to be contractible in $M$ if the closed curve that it defines in $\Sigma$ 
can be $\epsilon$-perturbed into a contractible closed curve in $\Sigma$.  Such a perturbation ensures that there are only finitely many self-intersections, and edges that are traversed more than once or vertices that are visited more than once by the walk are represented by disjoint curves or points and are thus distinguished from each other.
Problematic from our perspective is that both these definitions of contractibility rely on topological properties of curves in $\Sigma$ determined by closed walks of $\Gamma$ rather than properties of the map $M$ as a combinatorial object. 

In order to remain within the realm of maps as defined by cross involutions, we begin by introducing an analogous object to a {\em circuit} in a graph; in a graph a circuit is a closed walk that repeats no edges, in a map it is a closed walk that repeats no crosses (Definition~\ref{def:closed_walk}). In this way, we extend the scope of the term contractibility from the graph cycles of Mohar and Thomassen to cross-circuits; facial walks are then all contractible, even if they traverse an edge more than once, corresponding to the defining property of faces as enclosing topological discs.

\begin{definition}\label{def:closed_walk}
Let $M\equiv (C,\alpha_0,\alpha_1,\alpha_2)$ be a map with vertex permutation $\tau$ and face permutation~$\phi$. 
A {\em cross-circuit} $\kappa$ of $M$ is a pair of cycles,  
\[\kappa=(\;\; c_0 \;\;\; c_1 \;\;\; \cdots \;\;\; c_{\ell-1}\;\;)
\;(\;\; \alpha_0c_{\ell-1} \;\;\; \cdots \;\;\; \alpha_0 c_1 \;\;\; \alpha_0 c_0\;\;),\]
 with the property that the $2\ell$ crosses among $\{c_i\}\cup \{\alpha_0 c_i\}$ are pairwise distinct and for each $i$ there is a $j$ such that  $c_i=\tau^j\phi c_{i-1}$ (indices $i$ taken modulo $\ell$).\end{definition}

  The two cycles of the cross-circuit $\kappa$ represent its traversal  in opposite directions.  Its length is denoted by $\ell(\kappa)$. A facial walk traversing face $z$ is represented by the cross-circuit~$\phi_z$.

Recall that a closed walk in a graph is a sequence of edges which joins a sequence of vertices, starting and ending at the same vertex; for a circuit, edges are not allowed to be repeated. For orientable maps, we can find a combinatorial representation using half-edges (considering the set of half-edges to be the set of crosses in the equivalence classes modulo $\langle \alpha_0\alpha_2,\phi \rangle$ thus having two half-edges per each edge in the graph, instead of four crosses). Thus, in a cross-circuit along an orientable map, we would traverse each half-edge at most once, so each edge at most twice, once in each direction. If a map is non-orientable, we add an extra pair of half-edges per edge in the combinatorial representation to make a total of four crosses; then a cross-circuit only uses each cross from $\{c,\alpha_0c\}$ at most once, so each edge is used at most twice.

The choice of cross for a vertex-edge incidence reflects, not only the direction of travel (we go from $c$ ``to $\alpha_0 c$'' and then ``to $\alpha_1\alpha_0c$''), but also which ``side'' of the edge the walk is on, $c$ being on one side and $\alpha_2 c$ on the other. For instance, which direction a loop is traversed is determined by which of $c, \alpha_2 c$ or $\alpha_0 c$, $\alpha_0\alpha_2 c$ is chosen to represent the vertex-edge incidence, and hence the side of the edge as well.
A closed walk in a map is thus given not only by the sequence of vertices and edges, but by the side along which the edges are traversed -- this is determined by how the face permutation $\phi$ instructs the walk to take a step along an edge, which explains its role in Definition~\ref{def:closed_walk}.

A cross-circuit $(\;\; c_0 \;\;\; c_1 \;\;\; \cdots \;\;\; c_{\ell-1}\;\;)\;(\;\;\alpha_0c_{\ell-1} \;\;\; \cdots \;\;\; \alpha_0 c_1 \;\;\; \alpha_0 c_0\;\;)$ is said to be \emph{non-self-intersecting} 
if there are no pair of distinct indices $(i,j)$ such that the pairs of crosses $\alpha_0 c_i,c_{i+1}$ and $\alpha_0 c_j,c_{j+1}$ interlace in a cycle 
$(\;\; c \;\;\;  \alpha_1 c \;\;\; \tau c \;\;\; \tau\alpha_1 c\cdots \;\;\; \tau^{-1} c \;\;\; \tau^{-1} \alpha_1c \;\;)$
consisting of all the crosses incident with a given vertex. 

We next present a procedure of cutting a map around a non-self-intersecting cross-circuit, which leans on two operations that are first defined: edge splitting and vertex splitting.

\begin{definition}\label{def:edge_splitting}
Let $e\equiv(\;\; c \;\;\; \alpha_0\alpha_2 c\;\;)\;(\;\;\alpha_0 c \;\;\; \alpha_2 c\;\;)$ be an edge of a map $M\equiv (C,\alpha_0,\alpha_1,\alpha_2)$. The map $M'\equiv(C',\alpha_0',\alpha_1',\alpha_2')$ obtained from $M$ by {\em splitting edge $e$} into two distinct edges
bounding a new face of degree two 
is defined as follows:
\begin{itemize}
\item $C'=C\sqcup\{\alpha_2'c$, $\alpha_2'\alpha_0 c$, $\alpha_2'\alpha_2c,\alpha_2'\alpha_0\alpha_2 c\}$;
\item 
$\alpha_0'=\alpha_0$ on $C$, and
\begin{align}
\alpha_0'\,(\alpha_2'c)&=\alpha_2'\alpha_0 c\:, & \alpha_0'\, (\alpha_2'\alpha_0c) &= \alpha_2'c\: , \nonumber \\
\alpha_0'\,(\alpha_2'\alpha_0\alpha_2 c)&=\alpha_2'\alpha_2c\: , & \alpha_0'\,(\alpha_2'\alpha_2c)&=\alpha_2'\alpha_0\alpha_2 c \: ; \nonumber
\end{align}

\item 
$\alpha_2'=\alpha_2$ on $C\setminus\{c, \alpha_0 c, \alpha_2 c, \alpha_0\alpha_2 c\}$, and
\begin{align}
\alpha_2'\, (\alpha_2' c)&=c\: , & \alpha_2'\,(\alpha_2'\alpha_0 c)&=\alpha_0c\: , \nonumber \\
\alpha_2'\,(\alpha_2'\alpha_0\alpha_2 c)&=\alpha_0\alpha_2 c\: , & \alpha_2'\,(\alpha_2'\alpha_2c)&=\alpha_2 c\: ; \nonumber 
\end{align}
\item 
$\alpha_1'=\alpha_1$ on $C$, and
\begin{align}
\alpha_1'\, (\alpha_2'c)&= \alpha_2'\alpha_2 c\: ,  &\alpha_1'\, (\alpha_2'\alpha_2 c)&= \alpha_2'c\: , \nonumber \\
\alpha_1'\,(\alpha_2'\alpha_0c)&=\alpha_2'\alpha_0\alpha_2 c\: , &\alpha_1'\,(\alpha_2'\alpha_0\alpha_2 c)&=\alpha_2'\alpha_0c \: .\nonumber
\end{align}
\end{itemize}
\end{definition}

  Edge splitting is illustrated in Figure \ref{edge_vertex_splitting}(a), where it can be easily seen that this operation
is the inverse to edge gluing (setting $a=\alpha_2'c$ and $b=\alpha_2'\, (\alpha_0\alpha_2 c)$ in Definition~\ref{def:edge_gluing}). 

\begin{definition}\label{def:vertex_splitting}
Let $a,b$ be a pair of distinct crosses that are  coincident with a vertex $v$ of a map $M\equiv(C,\alpha_0,\alpha_1,\alpha_2)$. The map $M^{(a\;b)}$ obtained by riffling $a,b$ is the result of  \emph{splitting} the vertex $v$ through $a,b$.  
\end{definition}

  Vertex splitting is depicted in Figure~\ref{edge_vertex_splitting}(b). This operation is inverse to vertex gluing when the pair of crosses $a,b$ satisfy condition~\ref{en:riff_signedgenus_2} of Lemma~\ref{lem:riff_signedgenus} (see Figure~\ref{fig:vertexgluing}).
Tutte~\cite[Ch. X]{tutte01} defines a related operation of vertex splitting, which follows the operation in Definition~\ref{def:vertex_splitting} by  the insertion of a link joining the two vertices produced by splitting in our sense, in a way that preserves genus an orientability (exactly how this edge is inserted is described in the proof of Lemma~\ref{lem:riff_signedgenus} above).
In the proof of Lemma~\ref{lem:core_connection} the reader can find how Tutte's vertex splitting operation can be recovered using Definition~\ref{def:vertex_splitting}, vertex-gluing, duplicate edge-gluing, and edge addition.

\begin{figure}[htb]
\centering
\includegraphics[width=0.7\textwidth]{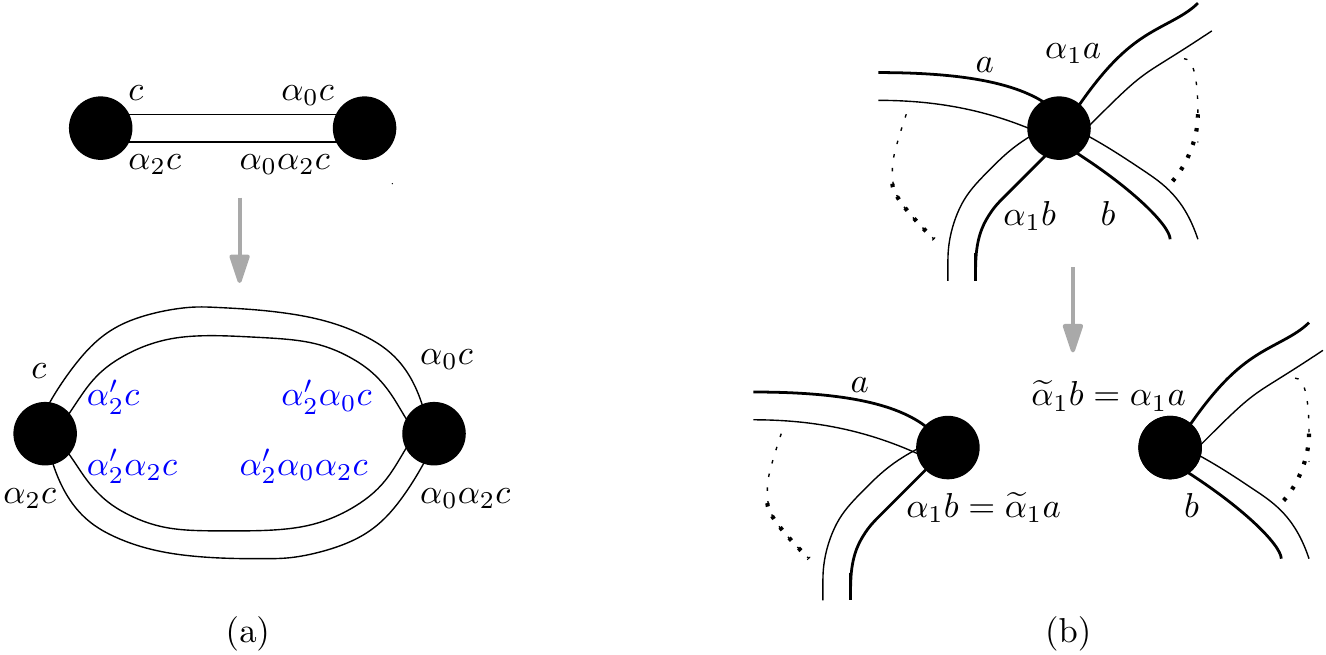}
\caption{\small (a) Edge splitting: the new crosses are labeled in blue, (b) Splitting  a vertex of a map $M$ through crosses $a, b$ to make the map $M^{(a\;b)}$.}
\label{edge_vertex_splitting}
\end{figure}  

Informally, the procedure of cutting around a non-self-intersecting cross-circuit $\kappa$, which we define next, is performed as follows. While traversing $\kappa$ we split each edge as we pass through it (in total, we split each edge as many times as it is traversed). Then, crosses of the split edges are associated with new faces of degree two (generated by the edge splitting). The property of non-self-intersecting allows us to merge these new degree-two faces when splitting vertices: as we proceed, these faces merge together to make a face of larger even degree until the last step, when the circuit closes, and the merged face is joined back on itself and the last vertex-splitting produces two new faces $x$ and $y$. If this last vertex-splitting disconnects the map, $x$ and $y$ are in different connected components (in which case the last vertex splitting, as all the previous ones, is inverse to vertex gluing, and preserves genus and orientability). Otherwise, the two vertices produced by this last splitting are not incident with a common face (neither do they belong to different connected components); one is incident with face $x$ and the other with face $y$.

\begin{definition} \label{def:cutting}
Let $M\equiv(C,\alpha_0,\alpha_1,\alpha_2)$ be a map, and let  \[\kappa=(\;\; c_0 \;\;\; c_1 \;\;\; \cdots \;\;\; c_{\ell-1}\;\;)\;\;(\;\;\alpha_0c_{\ell-1} \;\;\; \cdots \;\;\; \alpha_0 c_1 \;\;\; \alpha_0 c_0\;\;)\] be a non-self-intersecting cross-circuit of $M$. 
The map $M\ltimes \kappa$ obtained from $M$ by \emph{cutting around} $\kappa$  is the result of: (i) splitting the edges of $M$ containing $c_0,c_1,\ldots,c_{\ell-1}$ in turn, producing a map $M'\equiv (C',\alpha_0',\alpha_1',\alpha_2')$, (ii) applying the sequence of vertex-splittings in $M'$ realized by riffling crosses 
$\alpha_2'\alpha_0 c_i$ and $\alpha_1'\alpha_2'c_{i+1}$ for $i=0,\dots, \ell-1$ 
(indices modulo $\ell$). 
\end{definition}

Note that $M\ltimes \kappa$ has the same set of crosses as $M'$ in Definition~\ref{def:cutting}.   Figures~\ref{fig:prefacial_cross_circuit} and~\ref{fig:prefacial_cross_circuit(2)} illustrate examples of the procedure of cutting around a  non-self-intersecting cross-circuit $\kappa$. The new facial walks produced by cutting around  $\kappa$ are next recorded explicitly.

\begin{lemma}\label{lem:facial_walks_cut}
Let $M\equiv(C,\alpha_0,\alpha_1,\alpha_2)$ be a connected map, and let \[\kappa=(\;\;c_0 \;\;\; c_1 \;\;\; \cdots \;\;\; c_{\ell-1}\;\;)\;(\;\; \alpha_0c_{\ell-1} \;\;\; \cdots \;\;\; \alpha_0 c_1 \;\;\; \alpha_0 c_0\;\;)\]
 be a non-self-intersecting cross-circuit of $M$. If  $M'\equiv(C',\alpha_0',\alpha_1',\alpha_2')$ is obtained from $M$ by splitting the edges containing $c_0, c_1, \dots, c_{\ell-1}$ in turn as in Definition~\ref{def:cutting},  then the two facial walks of $M\ltimes \kappa$ that are not facial walks of $M$ are
 \[
\begin{cases}
    \kappa_{\ltimes}=(\;\; \alpha_2'c_{0} \;\;\; \alpha_2' c_{1}\;\;\; \dots \;\;\; \alpha_2'c_{\ell-1}\;\;)\;(\;\; \alpha_0'\alpha_2' c_{\ell-1} \;\;\;  \dots \;\;\; \alpha_0'\alpha_2'  c_1 \;\;\; \alpha_0'\alpha_2'c_{0} \;\;), \\
    \kappa^{\ltimes}=(\;\; \alpha_1'\alpha_2'c_{0} \;\;\; \alpha_1'\alpha_2' c_{1}\;\;\; \dots \;\;\; \alpha_1'\alpha_2'c_{\ell-1}\;\;)\;(\;\;\alpha_0'\alpha_1'\alpha_2'  c_{\ell-1} \;\;\;  \dots \;\;\; \alpha_0'\alpha_1'\alpha_2'  c_1 \;\;\; \alpha_0'\alpha_1'\alpha_2'  c_{0} \;\;).
\end{cases}\]

\end{lemma}

\begin{proof}
This follows by recording the effect of splitting edges to form the map $M'$, and  splitting vertices visited by $\kappa$. 
If the vertex being split is not the last one, the two faces incident with the two riffled crosses are different; for the last vertex we have crosses coincident with a common vertex and with a common face. See Figure~\ref{edge_vertex_splitting}. Vertex splitting after edge splitting means that we slice between $\alpha_2'c$ and  $\alpha_2'\alpha_2c$, and between $\alpha_2'\alpha_0c$ and  $\alpha_2'\alpha_0\alpha_2 c$ in Figure~\ref{edge_vertex_splitting}(a) all along the edges of $\kappa$.
\end{proof}


  The effect of vertex splitting on 
the signed genus in the cases that arise when we cut around
a non-self-intersecting cross-circuit is presented in the following lemma. 


\begin{lemma} \label{lem:vs_imp}
Let $M\equiv(C,\alpha_0,\alpha_1,\alpha_2)$ be a connected map, and let $a,b$ be crosses of $M$ that are coincident with a vertex $w$.  After vertex splitting $w$ through $a,b$ to make the map $M^{(\,a\:b\,)}$:
\begin{enumerate}[label=(\roman*)]
    \item  \label{en:vs_imp1} If $a,b$ are incident with distinct faces, then $\kk(M^{(\,a\:b\,)})=\kk(M)$ and  $\sg(M^{(\,a\:b\,)})=\sg(M)$.
    \item \label{en:vs_imp2} If $a,b$ are coincident with a common face and $\kk(M^{(\,a\:b\,)})>\kk(M)$, then  $\sg(M^{(\,a\:b\,)})=\sg(M)$.
    \item  \label{en:vs_imp3} If $a,b$ are coincident with a common face and $\kk(M^{(\,a\:b\,)})=\kk(M)$, then  $\sg(M^{(\,a\:b\,)})\neq \sg(M)$.
\end{enumerate}
\end{lemma}
\begin{proof}
\ref{en:vs_imp1} In this case, crosses $a,b$ are in $M^{(\, a\: b\,)}$ incident with different vertices and coincident with a common face, so, by Lemma~\ref{lem:riff_signedgenus} and since $[M^{(\, a\: b\,)}]^{(\, a\: b\,)}=M$, we have   $\sg(M^{(\, a\: b\,)})=\sg(M)$. That $\mathbf k(M^{(\, a\: b\,)})=\mathbf k(M)$ follows from the fact that  the Euler characteristic and  Euler genus are unchanged. 

\ref{en:vs_imp2} 
 The assumption on the number of connected components implies that splitting $w$ into $u$ and $v$ disconnects the component of $M^{(\, a\: b\,)}$ containing $u$ from the one containing $v$. 
Therefore, 
 crosses $a,b$  are in  $M^{(\, a\: b\,)}$ incident with vertices and faces in different connected components. By Lemma~\ref{lem:riff_signedgenus}, the result follows. 
 
\ref{en:vs_imp3} From $\vv(M^{(\, a\: b\,)})=\vv(M)+1$, $\ee(M^{(\, a\: b\,)})=\ee(M)$ and $\ff(M^{(\, a\: b\,)})=\ff(M)+1$, it follows that $\cchi(M^{(\, a\: b\,)})=\cchi(M)+2$. 
The latter, together with the assumption $\kk(M^{(\, a\: b\,)})=\kk(M)$, implies that  $\eg(M^{(\, a\: b\,)})=\eg(M)+2$.
Hence,

\[ \mbox{if}\hspace{0.7cm}\begin{cases}\oo(M^{(\, a\: b\,)})=\oo(M)=2 \\
\oo(M^{(\, a\: b\,)})=\oo(M)=1 \\
\oo(M^{(\, a\: b\,)})=2, \oo(M)=1\\
\oo(M^{(\, a\: b\,)})=1, \oo(M)=2\end{cases}\hspace{0.7cm}\mbox{then}\hspace{0.7cm}\begin{cases}\sg(M^{(\, a\: b\,)})=\sg(M)+1 \\ 
\sg(M^{(\, a\: b\,)})=\sg(M)-2 \\ 
\sg(M^{(\, a\: b\,)})=1-\frac{\sg(M)}{2} \\ 
\sg(M^{(\, a\: b\,)})=-2\sg(M)-2 \qquad \qedhere 
\end{cases}
\]
\end{proof}

  The map
$M\ltimes \kappa$ is produced from $M$ by first splitting (or duplicating) the edges of $\kappa$, which does not change the signed genus of $M$, and then by a sequence of vertex-splittings. When $\kappa$ is a non-self-intersecting cross-circuit of $M$, by Lemma~\ref{lem:vs_imp}\ref{en:vs_imp1}  
splitting its vertices visited in turn does not change the number of connected components until possibly when splitting the last vertex. In the notation of Definition~\ref{def:cutting}, this last vertex-splitting is realized by riffling the crosses $\alpha_2'\alpha_0 c_\ell$ and $\alpha_2'c_0$, which are coincident with a common vertex and with a common face, and so  $\textbf{k}(M)\leq  \textbf{k}(M\ltimes \kappa)\leq \textbf{k}(M)+1.$
When $\textbf{k}(M\ltimes \kappa)=\textbf{k}(M)$, the non-self-intersecting cross-circuit $\kappa$ is said to be \emph{non-separating}, and it is \emph{separating} if $\textbf{k}(M\ltimes \kappa) =\textbf{k}(M)+1$.  Lemma~\ref{lem:vs_imp}\ref{en:vs_imp3} and   Lemma~\ref{lem:vs_imp}\ref{en:vs_imp2} give the effect on the signed genus for these two cases.

When $\kappa$ is a separating cross-circuit of a connected map $M$, the map $M\ltimes\kappa$ is the disjoint union of two maps:
\begin{itemize}
\item a connected map $M_\kappa$ containing the crosses of $\kappa$ and the facial walk $\kappa_{\ltimes}$, and
\item a connected map $M^\kappa$ containing the crosses of the facial walk $\kappa^{\ltimes}$.\end{itemize}
The facial walks $\kappa_{\ltimes}$ and $\kappa^{\ltimes}$ are both described in Lemma~\ref{lem:facial_walks_cut};  see also Figures~\ref{fig:prefacial_cross_circuit} and~\ref{fig:prefacial_cross_circuit(2)}.  We say that $\kappa$ is \emph{contractible} if  $M_\kappa$ or $M^\kappa$ is a plane map. For non-connected maps, a cross-circuit is said to be contractible if it has this property in its own connected component. This notion of contractibility includes, as a special case, the contractible cycles as defined by Mohar and Thomassen \cite{mohar01}, and is consistent with the definition given by Cabello and Mohar \cite{CM07}  when restricting the latter to those curves that can be $\epsilon$-perturbed so as not to be self-intersecting. (Our term ``contractible'' assumes that the cross-circuit is separating and therefore non-self-intersecting.)

\begin{definition}\label{def:prefacial_cross_circuit}
A  contractible cross-circuit $\kappa$ of a connected map $M$ is said to be  \emph{prefacial} if $M_{\kappa}$ is plane and the facial walk $\kappa_{\ltimes}$  traverses a cycle in the underlying graph of $M_{\kappa}$.
\end{definition} 


Asking a cross-circuit $\kappa$ to be prefacial adds to the contractibility constraint that it is contractible in a certain direction ($M_\kappa$ is plane), and the constraint on the facial walk $\kappa_\ltimes$ adds an irreducibility condition.     
Figure~\ref{fig:prefacial_cross_circuit} illustrates a contractible cross-circuit that is not prefacial because its direction of contraction is not toward the plane part. An example of a contractible cross-circuit $\kappa$ with $M_\kappa$  plane and not prefacial is given by traversing two plane loops on a vertex along the half-edges ``inside'' the loops. (Or take $\kappa$ given by $(\; a\;\; \phi a\;\; b\;\;\phi b\;\;\phi^2b\;)$ in the right-hand map in Figure~\ref{fig:vertexgluing}(a).)
  Figure~\ref{fig:prefacial_cross_circuit(2)} illustrates a prefacial cross-circuit.

\begin{figure}[ht]
    \centering
    \includegraphics[scale=0.7]{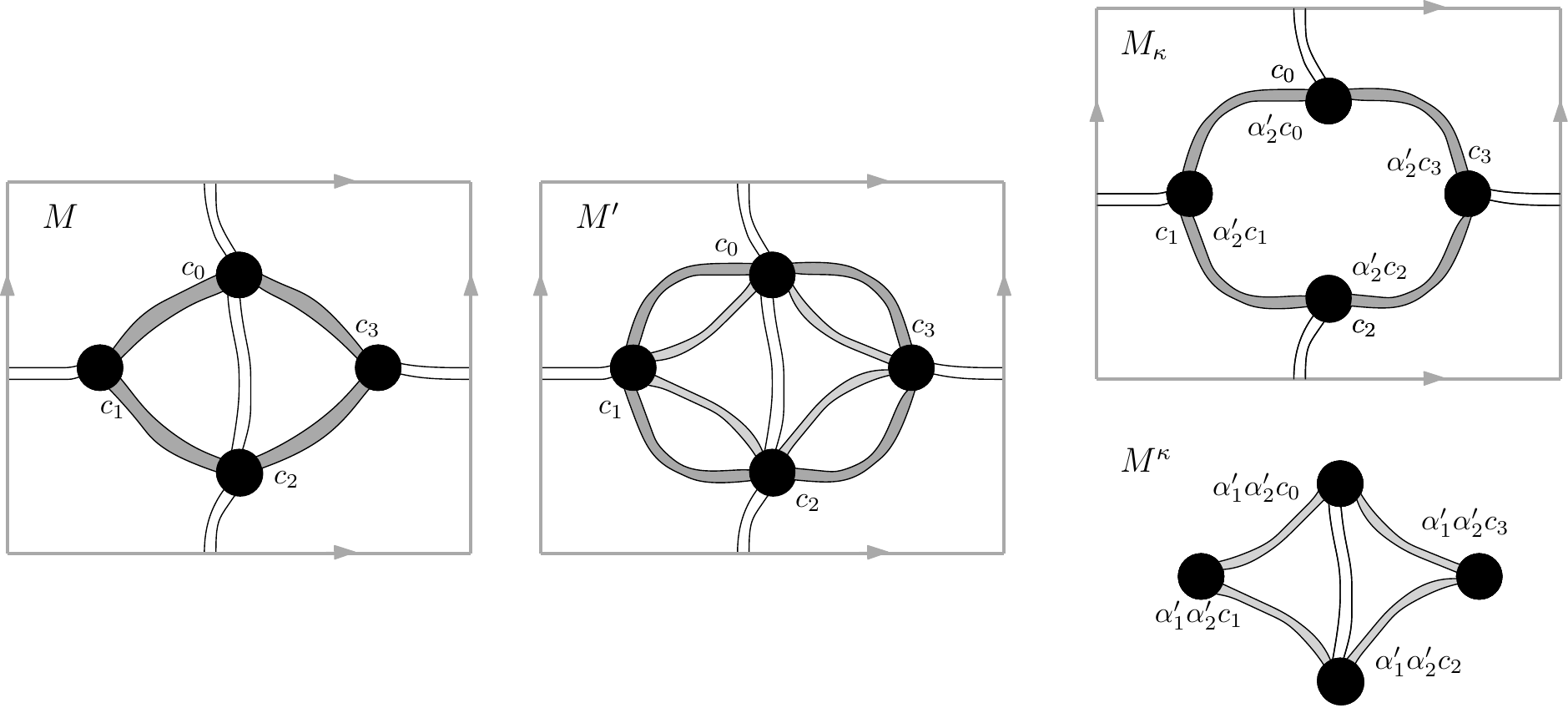}
    \caption{\small A cross-circuit $\kappa$ in a map $M$ (embedded in the torus) given by $(\; c_0 \;\; c_1 \;\; c_2 \;\; c_3 \;)$. Splitting the edges containing crosses of $\kappa$ (edges in dark grey) gives the map $M'$ (edges in light grey come from some edge splitting). The map $M\ltimes \kappa$ is then obtained by a sequence of vertex-splittings in $M'$: it is the disjoint union of $M_{\kappa}$ and $M^{\kappa}$. The cross-circuit $\kappa$ is contractible ($M^{\kappa}$ is plane) and not prefacial ($M_{\kappa}$ is not plane).
    }
    \label{fig:prefacial_cross_circuit}
\end{figure}

\begin{figure}[ht]
    \centering
    \includegraphics[scale=0.7]{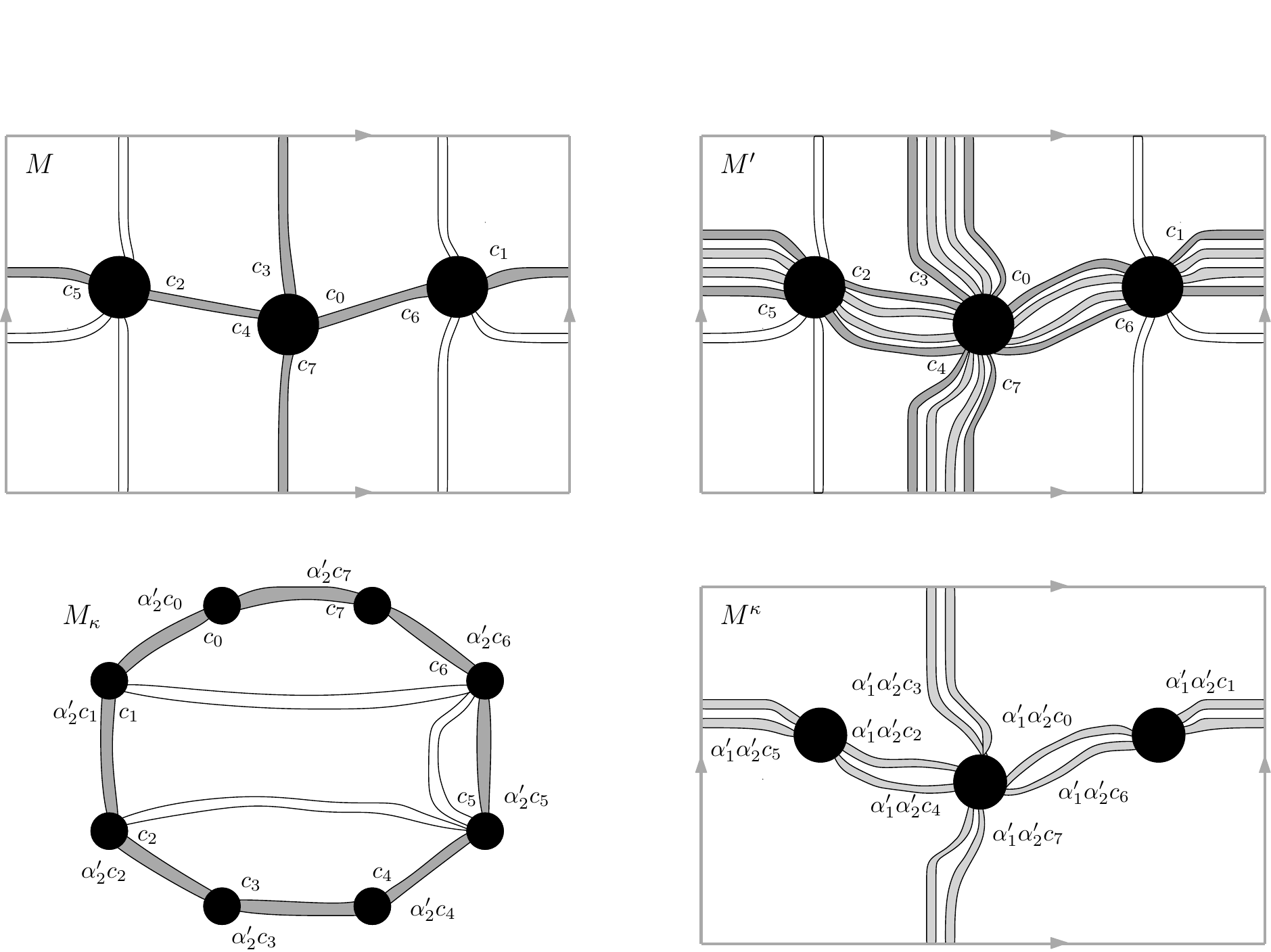}
   \caption{\small A cross-circuit $\kappa$ in a map $M$ (embedded in the torus) given by $(\; c_0 \;\; c_1 \;\; c_2 \;\; c_3 \;\; c_4 \;\; c_5 \;\; c_6 \;\; c_7\;)$. The map $M'$ is obtained 
by splitting the edges containing crosses of $\kappa$ (edges in dark grey); each edge is split twice as it contains two crosses in the cycle defining $\kappa$ (edges in light grey come from some edge splitting).  The map $M\ltimes \kappa$ is then obtained by a sequence of vertex-splittings in $M'$: it is the disjoint union of $M_{\kappa}$ and $M^{\kappa}$. The cross-circuit $\kappa$ is prefacial.}
    \label{fig:prefacial_cross_circuit(2)}
\end{figure}

\subsection{Cross-circuits and homomorphisms}\label{sec:circuits-hom}

In this section, we present technical tools on the interaction between map homomorphisms and  different types of cross-circuits; this leads to the characterization of map cores (Theorem~\ref{thm:core_char}). We begin by showing that there is a correspondence between prefacial cross-circuits of a connected map and facial walks of any connected submap of the same signed genus. Prefacial cross-circuits are then used to properly partition the set of edges of a map so that homomorphisms will act independently on each subset of the partition. 

\begin{lemma}\label{lem:cont_submap}
Let $M$ be a connected map, and let $N$ be a connected submap of $M$ with the same signed genus. Then,
\begin{enumerate}[label=(\roman*)]
    \item \label{en.count_sub_1} Each facial walk $\kappa$ in $N$ is a prefacial cross-circuit of $M$.
    \item \label{en.count_sub_2} Conversely, for each prefacial cross-circuit $\kappa$, there exists a submap $N(\kappa)$ with the same signed genus as $M$ such that $\kappa$ is a facial walk of $N(\kappa)$.
    \item \label{en.count_sub_3} 
    Each edge in $E(M)\setminus E(N)$ belongs to a unique $M_{\kappa}\setminus E(\kappa_{\ltimes})$ for some  facial walk $\kappa$ in $N$.
\end{enumerate}
\end{lemma}

\begin{proof}
For \ref{en.count_sub_1} and \ref{en.count_sub_3} it suffices to prove that if we add the edges $e$ of $E(M)\setminus E(N)$ to $N$ one by one (so that the resulting map is always connected), each edge $e$ either subdivides an existing face or adds a new vertex. Then, a small inductive argument over the facial walks of the successive maps establishes the result.

As the orientability of the submaps generated by adding the edges $e$ remains constant, so does the signed genus and the Euler genus (see Lemma~\ref{lem:s_submap}). Then, all these submaps have the same genus, which, by Lemma~\ref{lem:effect_del_sg}, implies that edge $e$ is either a dual link (hence dividing a face into two) or a bridge (adding a new vertex), as desired.

Part \ref{en.count_sub_2} follows from removing in $M$ the edges and vertices in $M_{\kappa}$ that do not belong to $\kappa$; this operation creates a submap $N(\kappa)$ of $M$ satisfying the conditions of the statement. 
\end{proof}


 \begin{lemma}\label{lem:lift_homomorphism}
 Let $\kappa$ be a separating cross-circuit of a connected map $M$.
 \begin{enumerate}[label=(\roman*)]
     \item \label{en:lift_hom_1}A homomorphism $h_\kappa:M_\kappa\to M_\kappa$ fixing the crosses of $\kappa$ induces a homomorphism $h:M\to M$ with the following properties:
     \begin{itemize}
         \item $\kappa$ is a separating cross-circuit of $h(M)$,
         \item $h(M)\ltimes\kappa$ consists of the map $M^\kappa$ and the connected map $h_\kappa(M_\kappa)$,
         \item $h$ is the identity on $M\setminus [E(M_{\kappa})\setminus E(\kappa)]$.
     \end{itemize}
     
   \item \label{en:lift_hom_2}  The cross-circuit $\phi_z$ corresponding to a face $z$ in $M$ is a facial walk of either $M_\kappa$ or $M^\kappa$.
Conversely, a facial walk of $M_\kappa$ (resp. $M^{\kappa}$) other than $\kappa_\ltimes$ (resp. $\kappa^{\ltimes}$)  defines a facial walk of~$M$. 
 \end{enumerate}
 \end{lemma}
  \begin{proof}
\ref{en:lift_hom_1} Since the crosses of $\kappa$ are fixed by $h_{\kappa}$, the rifflings associated with the vertex gluings that $h_{\kappa}$ defines in $M_{\kappa}$ do not involve any cross from $\kappa_{\ltimes}$. Further, we may assume that no cross from $\kappa_{\ltimes}$ is removed when gluing duplicate edges (in the notation of Definition~\ref{def:edge_gluing} at least one of the crosses $a$ and $b$ does not belong to $\kappa_{\ltimes}$; if one of them belongs, say $a$, we can say  that the edge containing $a$ still belongs to the map). Similarly, we may consider the identity map on $M^{\kappa}$, and conclude that  no vertex or duplicate edge gluings involve crosses from $\kappa^{\ltimes}$. 

As $\kappa$ is separating, Lemma~\ref{lem:vs_imp}\ref{en:vs_imp2}\ref{en:vs_imp1} 
give that $\sg(M)=\sg(M\ltimes \kappa)=\sg(M_{\kappa}\sqcup M^{\kappa})$. Thus, we can reverse the cutting operations that form $M\ltimes \kappa$ by gluing the appropriate vertices (in the  notation of Definition~\ref{def:cutting} these vertex gluings correspond to riffling back all the crosses $(\alpha_2'\alpha_0 c_i,\alpha_1'\alpha_2' c_{i-1})$, indices modulo $\ell$) and then gluing the edges of $\kappa_{\ltimes}$ with the edges of $\kappa^{\ltimes}$. This creates a homomorphism $h:M\to M$ satisfying the stated properties.

\ref{en:lift_hom_2} Cutting around a non-self-intersecting cross-circuit~$\kappa$ creates, by Lemma~\ref{lem:facial_walks_cut}, two additional faces in $M\ltimes \kappa$, and these precisely involve those crosses added in the process of cutting around; the remaining facial walks not involving the added crosses to $M\ltimes \kappa$ are unchanged. As in addition additional separating hypothesis on $\kappa$, $M\ltimes \kappa=M_{\kappa}\sqcup M^{\kappa}$ and the result follows.
 \end{proof}

 A face $z$ of a map $M$ that is a face of $M_{\kappa}$ (in the sense that the pair of cross cycle permutations of $\phi_z$ defines  a facial walk of both $M$ and $M_\kappa$) is said to be {\em contained in} $\kappa$ (and $\kappa$ {\em contains}~$z$).


\begin{observation}\label{obs.parity_faces}
A  prefacial cross-circuit $\kappa$ of odd length is either a facial walk or contains an odd number of odd-degree faces. If $\kappa$ has even length, it contains an even number of odd-degree faces. The argument to prove the odd-length statement  is the following (similar for even length): $M_{\kappa}$ is plane, and has an odd-degree face defined by $\kappa_{\ltimes}$. Considering the dual map~$M_{\kappa}^*$, there must be an odd number of other faces of odd degree in $M_\kappa$, each of which by Lemma~\ref{lem:lift_homomorphism}\ref{en:lift_hom_2} is a face of $M$ contained in $\kappa$ (by the handshaking lemma, the number of odd-degree vertices is even). 
\end{observation}

Fixed a  prefacial cross-circuit $\kappa$ of a map $M$, there is a partial order $\preceq$ on contractible cross-circuits of $M_\kappa$ defined by: $\lambda\preceq \mu$ if 
$\lambda$ determines a cross-circuit in $M_\mu$ and all the crosses in $M_{\lambda}$ belong to $M_{\mu}$ (perhaps with the exception of $\lambda_{\ltimes}$). This order has maximum element $\kappa_{\ltimes}$ and, for  each face $z$ of $M_\kappa$, one of its minimal elements is $\phi_z$.

We next use the partial order $\preceq$ and bring  together the previous results in this section to establish when a whole plane map can be sent (in a map-homomorphic way) to its surrounding  prefacial cross-circuit. When no confusion may arise, we shall use $M\to \lambda$ to
indicate that a map $M$ is mapped to the map induced by the edges of a prefacial cross-circuit $\lambda$.

\begin{lemma}\label{lem:technical_1}
  Let $\kappa$ be a 
prefacial cross-circuit of a connected map $M$. Assume that there exists a face $z$ of $M_\kappa$ such that:
\begin{enumerate}[label=(\roman*)]
\item $z$ is different from the face defined by $\kappa_{\ltimes}$ in $M_\kappa$;
\item \label{en.tech_1} the degree of $z$ 
is at least the degree of the face traversed by $\kappa_{\ltimes}$, and both faces have degree of the same parity;
\item \label{en.tech_3} every other face of $M_\kappa$ has even degree;
\item \label{en.tech_4}
 $\ell(\kappa) \leq\ell(\lambda)$ for every prefacial cross-circuit $\lambda$ of $M_\kappa$ such that $\phi_z\preceq\lambda\preceq\kappa$. 
    \end{enumerate}
Then there is an endomorphism of $M_\kappa$ that maps $M_\kappa$ to the map induced by the edges of $\kappa$ (equivalently $\kappa_{\ltimes}$) and it is the identity over $\kappa$ and $\kappa_{\ltimes}$.
\end{lemma}

  The following facts are used in the proof of the preceding lemma.

\begin{observation} \label{obs.some_obs_contract}
Under the hypotheses of Lemma~\ref{lem:technical_1} it holds that:
\begin{itemize}
    \item $M_{\kappa}$ and all its connected submaps are plane, and hence orientable. Thus, we can indicate a single permutation cycle associated to faces, vertices, and contractible cross-circuits. 
    \item Any cross-circuit of $M_{\kappa}$ that is a cycle in the underlying graph of $M_{\kappa}$ is a contractible cross-circuit (as $M_{\kappa}$ is plane).
    \item Any  prefacial cross-circuit $\lambda$ such that $\phi_z\preceq\lambda\preceq\kappa$ has the same parity as $\kappa$ and $\phi_z$ (this follows from assumptions \ref{en.tech_1} and \ref{en.tech_3} by a parity argument involving the fact that if an even contractible cross-circuit contains an odd degree face, then it contains at least two of them (Observation~\ref{obs.parity_faces})).

\end{itemize}

\end{observation}

\begin{proof}[Proof of Lemma~\ref{lem:technical_1}.]
We show the result inductively on the number of edges and vertices inside $M_{\kappa}$.
The base case occurs when $\phi_z= \kappa$, and so $M_{\kappa}$ is a plane cycle; the statement follows by using the identity map as a map homomorphism. Suppose now that $\phi_z\neq \kappa$; the induction argument follows a case analysis.

\vspace{0.2cm}

  \emph{Case 1: There is no prefacial cross-circuit $\lambda$ with $\phi_z\preceq\lambda\preceq\kappa$ and $\phi_z\neq\lambda\neq\kappa$.} 
  Let $z'$ be a face different from $z$ containing crosses of $\kappa$ in $M_{\kappa}$ (it exists as  $\phi_z\neq \kappa$). We shall assume that $\phi_{z'}\neq \kappa$,  otherwise $\kappa$ would only contain $z'$ and not $z$ as well. Let  $\phi_{z'}=(\;\;c_0\;\;\;\cdots\;\;\;c_s\;\;)$ such that $c_0 \in \kappa$ and $c_1\notin \kappa$. Note that $z'$ must contain crosses in $\kappa$ and outside $\kappa$: if all the crosses were in $\kappa$, it would be $\kappa=\phi_{z'}$ as $\kappa_{\ltimes}$ is a cycle in the underlying graph of the plane map; if all the crosses were outside $\kappa$, one could find a prefacial cross-circuit $\lambda$ between  $\phi_z$ and $\kappa$ (using the connectedness of $M$ and $M_{\kappa}$.).
 

\vspace{0.2cm}

  \emph{Case 1.1: If the vertex adjacent to 
  $c_2$ belongs to $\kappa$ and is the same as the vertex of $c_0$}, then $c_2=c_0$; otherwise we can find a  prefacial cross-circuit $\lambda$ strictly shorter than $\kappa$ containing $z$ but,  as there are no other prefacial circuits between them,  it must be $\lambda=\phi_z$, which contradicts  
 assumption \ref{en.tech_1}. Hence,  we conclude $c_2=c_0$ and $z'$ is a face of length $2$.
By removing/gluing the edge closing $z'$, we obtain a map $M_{\kappa}'$ with no new  prefacial cross-circuits between $\phi_z$ and $\kappa$ satisfying the hypotheses of Lemma~\ref{lem:technical_1}. By induction, $M_{\kappa}'$ can be mapped to $\kappa$ by a homomorphism. Thus, the composition of first removing/gluing the edge from $M_{\kappa}$ to $M_{\kappa}'$ and then mapping $M_{\kappa}'$ to $ \kappa$ gives us the result.

\vspace{0.2cm}

  \emph{Case 1.2: If the vertex adjacent to  $c_2$ belongs to $\kappa$ and is different than the vertex of $c_0$,} then $\kappa$ is of the form $(\;\;c_0\;\;\; c_1'\;\;\;\cdots\;\;\; c_i'\;\;\; c_{2}\;\;\; c_{i+2}'\;\;\;\cdots\;\;\; c_{\ell(\kappa)-1}'\;\;)$ and it can be partitioned into the crosses $\{c_1',\ldots, c_i'\}$ and $\{c_2,c_{i+2}',\ldots,c_{\ell(\kappa)-1}',c_0\}$. Then, either the  prefacial cross-circuit formed by $\alpha_2\alpha_0c_1$ and the crosses $\{c_1',\ldots, c_i'\}$, or the one formed by $c_1$ and $\{c_2,c_{i+2}',\cdots,c_{\ell(\kappa)-1}',c_0\}$  contains face $z$. Hence, this prefacial cross-circuit must be $\phi_z$ as there are no prefacial cross-circuits between $\kappa$ and $\phi_z$ (and it is not $\kappa$). Since the vertices of $c_0$ and $c_2$ are distinct, the length of $\phi_z$ is strictly smaller than that of $\kappa$, contradicting assumption \ref{en.tech_1}.

\vspace{0.2cm}

  \emph{Case 1.3: If the vertex adjacent to $c_2$ does not belong to $\kappa$,} then we glue the vertices of $c_0$ and  $c_2$ along the face $z'$ via $M_{\kappa}^{(c_0\;c_2)}$, and remove/glue the edge of $c_2$. We obtain a new map $M_{\kappa}'$ with $\kappa_{\ltimes}$ in the outer face, one edge and one vertex less, and all the faces having even length (as $z'$ had even length and we have shrink it by two units). Further, we have not created any prefacial cross-circuit between $\kappa$ and $\phi_z$, as it would correspond to a (perhaps several) contractible cross-circuit between $\kappa$ and $\phi_z$ in $M_{\kappa}$. No other face length has been altered (in particular, the length of $z$). Thus, the assumptions of Lemma~\ref{lem:technical_1} are met, and we can apply the result by induction. The result follows by composing the vertex gluing $M_{\kappa}^{(c_0\;c_2)}$ with gluing the now duplicate edge containing $c_2$ and the mapping from $M_{\kappa}'$ to $\kappa$. (Note that one of the glued vertices does not belong to $\kappa$, and neither does the edge where $c_2$ belongs, thus the condition that the mapping is the identity over $\kappa$ is preserved.)

\vspace{0.2cm}

  \emph{Case 2: There exists a  prefacial cross-circuit $\lambda$ with $\phi_z\preceq\lambda\preceq\kappa$ and $\phi_z\neq \lambda\neq \kappa$.} 
  Let $\lambda_0$ be the prefacial cross-circuit that is the ``closest'' to $\phi_z$ having the same length as $\kappa$ (perhaps being $\kappa$ itself), that is, 
\begin{itemize}
    \item $\lambda_0\neq \phi_z$ and $\phi_z\prec\lambda_0\preceq\kappa$;
    \item $\ell(\kappa)= \ell(\lambda_0)$ and $\ell(\lambda_0)<\ell(\lambda)$ for every $\lambda$ with $\phi_z\prec\lambda\preceq\lambda_0$.
\end{itemize}
The prefacial cross-circuit $\lambda_0$ induces a graph cycle in $M_{\kappa}$; otherwise, by choosing $\lambda_0'$ to be the cross-circuit of the (minimal) graph cycle all whose crosses belong to $\lambda_0$ and that contains face $z$, we would have $\phi_z\preceq\lambda_0'\preceq\kappa$ and $\ell(\lambda_0')<\ell(\kappa)$, which contradicts assumption \ref{en.tech_4}.\footnote{Note that we are not claiming that such $\lambda_0'$ is inside $\lambda_0$; for instance, if  $\lambda_0'$ is a cycle and $\lambda_0$ is that cycle with a pendant edge inside it, then $\lambda_0$ is inside $\lambda_0'$.}


\vspace{0.2cm}

  \emph{Case 2.1: If $\lambda_0\neq \kappa$,} then $M_{\lambda_0}$ is a strict submap of $M_{\kappa}$ with the same signed genus. Indeed, both maps are plane and $\lambda_0$  induces a graph cycle, and so cutting around $\lambda_0$ in $M_{\kappa}$ is the same as removing all the edges and vertices outside $\lambda_0$ in $M_{\kappa}$. Thus, by induction, there is an endomorphism $M_{\lambda_0}\to \lambda_0$ that maps $\phi_z$ to $\lambda_0$, and is the identity over $\lambda_0$.  Lemma~\ref{lem:lift_homomorphism} allows us to extend this endomorphism to a mapping $h_{\lambda_0}:M_{\kappa}\to M_{\kappa}$ (restricted to $M_{\lambda_0}$ is the endomorphism $M_{\lambda_0}\to \lambda_0$, and is the identity elsewhere). Since $\lambda_0\neq \kappa$, the image $h_{\lambda_0}(M_{\kappa})$ is a strict submap of $M_{\kappa}$ and has $\lambda_0$ as a facial walk. We can apply induction to obtain an endomorphism $h_{\kappa}:h_{\lambda_0}(M_{\kappa})\to \kappa$ that maps $\lambda_0$ to $\kappa$, and is the identity over $\kappa$. The composition $h_{\lambda_0}\circ h_{\kappa}: M_{\kappa} \to \kappa$ has the desired properties.

\vspace{0.2cm} 

  \emph{Case 2.2: Assume now that $\lambda_0=\kappa$.} 
We perform a similar procedure of gluing a vertex and an edge (as in Case~1.3), or simply glue an edge (as in  Case~1.1).
Let $\lambda_1$ be a prefacial cross-circuit with $\phi_z\prec \lambda_1 \prec \kappa$ (it exists by the assumption of Case~2) and such that there is no prefacial cross-circuit $\lambda_2$ with $\lambda_1 \prec \lambda_2\prec \kappa$ (can be assumed as the  the number of prefacial walks is finite). We have $\ell(\lambda_1)>\ell(\kappa)$ (by the properties of $\lambda_0$ and that $\lambda_0=\kappa$ in our case). A small reduction argument (as done for Case~1)
shows that $\lambda_1$ is of the form $(\;\;c_0\;\;\;c_1\;\;\;\cdots\;\;\;c_i\;\;\;c_{i+1}'\;\;\;\cdots\;\;\;c_t'\;\;)$ where the $c_j$'s are crosses of $\kappa$ (assuming we start with $c_0$ for convenience) and $c_{i+1}', \ldots ,c_t'$ are not. The vertex of $c_{i+1}'$ belongs to $\kappa$, yet the vertex adjacent to $c_{i+2}'$ does not (otherwise it would be $\ell(\lambda_1)\leq \ell(\kappa)$). The face of $c_{i+1}$ is also the face of $\alpha_0\alpha_2 c_{i+1}'$ (in fact, $\phi \alpha_0\alpha_2 c_{i+1}'=c_{i+1}$), otherwise there would exist another prefacial cycle between $\lambda_1$ and $\kappa$. Moreover, the face of $c_{i+1}$ is not $z$ as it is not inside $\lambda_1$. Now, we glue the vertices adjacent to $c_{i+2}$ and $c_{i+2}'$ by the riffling $M_{\kappa}^{(\alpha_0\alpha_2 c_{i+1}\; \phi c_{i+1})}$. Observe that, for this new map:
\begin{itemize}
    \item Condition \ref{en.tech_1} holds, as $\kappa$ and $\phi_z$ have not been modified.
    \item The parities of the prefacial cross-circuits between $\phi_z$ and $\kappa$ are the same as $\kappa$ (Observation~\ref{obs.some_obs_contract}).
    \item Condition \ref{en.tech_3} holds since we do a gluing through a face outside $z$ of two vertices at distance two in a face. Further, the parity of the prefacial cross-circuits between $\phi_z$ and $\kappa$ remains the same (using Observation~\ref{obs.some_obs_contract} on both the new and the old map).
    \item Some prefacial cross-circuits between $\phi_z$ and $\kappa$ may have been removed or  modified by reducing their size. The reduction would be of (exactly) two units, and thus  the new prefacial cross-circuits between $\phi_z$ and $\kappa$ satisfy condition \ref{en.tech_4}  with $\leq$ (the hypothesis on Case~2 asked for a strict inequality between their lengths).
    
\end{itemize}
Therefore, we can apply Lemma~\ref{lem:technical_1} inductively to the new map $M_{\kappa}^{(\alpha_0\alpha_2 c_{i+1}\; \phi c_{i+1})}$. Since the two mappings are the identity over $\kappa$ (we are using here Lemma~\ref{lem:lift_homomorphism} for the vertex gluing that we have performed) the composition of gluing a vertex and afterwards performing the inductive step is also a mapping that is the identity over $\kappa$, thus the result follows.
\end{proof}

\subsection{Characterization of map cores}\label{sec:char_cores}

We bring together the notions and  results of the previous sections on cores in the following theorem.

\begin{theorem}\label{thm:core_char}
A connected map $M$ is a core if and only if  each of its prefacial cross-circuits $\lambda$ either is a~facial walk, or~contains more than one face of odd degree, or contains no face whose degree is greater than or equal to~$\ell(\lambda)$ and has the same parity as~$\ell(\lambda)$.

\end{theorem}


\begin{proof}
We argue with the contrapositive in both implications.
From left to right, assume that there is a prefacial  cross-circuit $\lambda$ of $M$ such that: (i) is not a facial walk, (ii)  contains at most one face of odd degree, (iii)  contains a face $z$ whose degree is at least $\ell(\lambda)$ and has the same parity as $\ell(\lambda)$.
    
If $\lambda$ has odd length, by (ii) and (iii), $z$ is the only face of odd degree inside $\lambda$, and any prefacial cross-circuit $\mu$ with $\phi_z\preceq\mu\preceq\lambda$ has odd length (see Observation~\ref{obs.some_obs_contract}).
If $\lambda$ has even length, then all the faces inside $\lambda$ have even degree (otherwise, by Observation~\ref{obs.parity_faces}, there would be at least 
two faces of odd degree inside $\lambda$).

Let $\kappa$ be a prefacial cross-circuit such that $\phi_z\prec\kappa\preceq\lambda$ and $\ell(\kappa)=\min\{\ell(\mu):\phi_z\preceq\mu\preceq\lambda\}$
(perhaps being $\lambda$ itself).
The conditions of Lemma~\ref{lem:technical_1} are  met with face $z$ and prefacial cross-circuit $\kappa$. There is thus a homomorphism $h_{\kappa}:M_{\kappa}\to \kappa$ that is the identity over $\kappa$. Since $\kappa$ is a prefacial cross-circuit, by Lemma~\ref{lem:cont_submap}, we can conclude that the submap $M'$ obtained by removing the interior of $\kappa$ (which is the largest submap of $M$ where $\kappa$ is a facial walk)  has the same signed genus as $M$. Further, $\kappa\neq \phi_z$ which implies that $M'$ is a proper submap of $ M$.
Lemma~\ref{lem:lift_homomorphism} (applied to $h_{\kappa}$) then gives a homomorphism $h:M\to M'\subsetneq M$; therefore, $M$ is not a core.

Now, we prove the implication from right to left using the contrapositive again. Assume that $M$ is not a core; we want  to find a prefacial cross-circuit satisfying  conditions (i)--(iii) stated above. 

 Let $N$ be the core of $M$, and let $h:M\to N$ be a homomorphism, whose restriction to $N$ is the identity (see Proposition~\ref{p.prop_cores}\ref{en.core4}).
Since $N$ is a proper submap of $M$ with the same signed genus, Lemma~\ref{lem:cont_submap} asserts that there is a facial walk of $N$ that is a prefacial cross-circuit $\kappa$ of $M$ but not a facial walk of $M$.
Consider again the submap $M'$ of $M$ (defined as above), 
and let  $h':M\to M'$ be the mapping obtained from $h$ as follows:
\begin{itemize}
    \item Vertices $u,v$ are glued if $h$ glues them and either they are in the interior of $\kappa$, or $u$ is in the interior of $\kappa$ and $v$ is in $\kappa$.
    \item Edges $e,e'$ are glued if $h$ glues them and either they are in the interior of $\kappa$, or $e$ is in the interior of $\kappa$ and $e'$ belongs to $\kappa$.
\end{itemize}
Once we cannot perform any of the two previous operations, we stop and obtain $h'$.
Observe that we never glue two vertices or two edges of $\kappa$ as $h$ is the identity over $N$.

We now construct a layered and labelled graph $T$ by matching the above described gluing instructions; the layers will represent the state, at that stage of the process, of the faces in the interior of $\kappa$, which are denoted by $\{f_1,\ldots,f_r\}$.  In the first layer of $T$,  we place $r$ vertices labelled $f_1,\ldots,f_r$ (one vertex per face).
Now, if a vertex gluing subdivides a face (into two new faces), 
we place two new vertices (in the following layer) adjacent to the vertex that represents the original face (the parent); the edges are directed, always pointing towards the new vertices, which are both labelled with the label of the parent. 
If we glue two edges, the two faces that they bound are merged into one; this is represented by a vertex in the new layer, which is adjacent to the parent vertices by directed edges (again pointing towards the ``new''). This new vertex takes the label of the face with the largest degree 
among the two that we are merging (if there is a tie between two faces of degree $2$, we choose the label of one of them arbitrarily). For the remaining vertices (representing faces, including those that have not been subdivided), we place a new vertex in the new layer with the same label as the parent, and with a directed edge pointing from the old vertex to the new one.


The preceding procedure ensures that all the vertices of $T$ bear one of the labels of the original faces inside $\kappa$, that is, $\{f_1,\ldots,f_r\}$. Further, from one layer of $T$ to the next, the process involve some subdivision of faces (when two vertices are glued) and some destruction of faces (when two edges are glued); at the end, only one face remains, represented by a unique vertex in the last layer of $T$, whose label indicates the original face that has not been destroyed/merged-into-another-one.
We next discuss how the parity of the degree of the faces evolve in this process.

When a face of odd degree is subdivided,  a new face of odd degree and a face of even degree appear; if we subdivide a face of even degree, two faces of odd degree or two faces of even degree appear.  Then, no faces of odd degree are created when subdividing even degree faces, as there is only one face at the end of the process, and when an odd degree face is created, actually two are created, and none of them can be removed (we can only remove faces of degree $2$). Further, the vertices of any path in $T$ traversing different layers (from higher to lower layers) have  the property that if their label is the same,  the sequence of degrees of their associated faces is weakly decreasing. Let $f_i$ be the label of the unique vertex in the last layer, and consider a path with all vertices labelled $f_i$ that goes from the first to the last layer; we conclude that: 
\begin{itemize}
    \item the parity of $f_i$ is the same as $\kappa$ (it is maintained through the process, and such facial walk is modified until it becomes $\kappa$);
    \item the degree of the other faces in $M$ is even, and the degree of $f_i$ is larger or equal than $\ell(\kappa)$.
    \end{itemize}
Thus, $\kappa$ is a prefacial cross-circuit of $M$ satisfying conditions (i)--(iii); 
this finishes the contrapositive of the implication from right to left and completes the proof of the whole statement. 
\end{proof}

\paragraph{Application of the characterization to four classes of maps.} A {\em bouquet} is a map with one vertex, all edges being loops. Dually, a {\em quasi-tree} is a map with one face, also known as a unicellular map. (Bouquets and quasi-trees are necessarily connected.)
Theorem~\ref{thm:core_char} says that: a bouquet is a core if and only if it has no duplicate edges; a quasi-tree is a core if and only if either it is plane and consists of one edge or one isolated vertex, or it is not plane and has no vertices of degree one.
(These two facts can be shown directly.)
For the other two classes of connected maps, we exploit the parity conditions for $\kappa$ in Theorem~\ref{thm:core_char} to conclude the following.
A connected map with no vertices of degree one and all faces of odd degree is a core. 
Any plane connected core map satisfies exactly one of the following properties: (i) consists on a single edge or a single isolated vertex (when all faces have even degree and thus the underlying graph is bipartite by Observation~\ref{obs.some_obs_contract}), or (ii) is a cycle of odd length (thus having precisely two faces, both of odd degree), or (iii) has at least four odd degree faces (four instead of three using Observation~\ref{obs.parity_faces}).

\paragraph{Graph cores vs.~map cores.}
Since a map homomorphism is a graph homomorphism, any embedding of a graph core is a map core. 
However, there are map cores for which the underlying graph is not a graph core, for example, the quasi-tree embedding of a vertex with two loops on a torus. 
Further, the characterization of quasi-trees as map cores (stated in the previous paragraph) gives a sharp contrast between map cores and graph cores. Indeed, the fact that any connected graph has a quasi-tree embedding in some surface, possibly non-orientable~\cite[Page~123, Third paragraph]{edmonds1965surface}, 
yields that any connected graph without vertices of degree one can be the underlying graph of a map core by embedding it in an appropriate surface.


On the other side, identifying two vertices by a graph homomorphism can be seen as gluing two vertices by a map homomorphism, using an appropriate embedding (placing a face between the two vertices to be glued); to deal with parallel edges, the graph is embedded so that they are duplicated. Therefore, any graph homomorphism can be viewed as the composition of several map homomorphisms (perhaps by changing the embeddings of the corresponding underlying graphs).


This discussion prompts a natural question on whether map cores can be used to identify graph cores:
for every connected graph $\Gamma$ that is not a core, is  it possible to find a map $M$ with underlying graph $\Gamma$ such that $M$ is not a (map) core?

We answer this question in the negative. Let $\Gamma$ be the graph resulting from identifying two copies of the complete graph $K_5$ by a vertex (it has 9 vertices and no additional edges). The core of $\Gamma$ is $K_5$. In fact, all graph homomorphisms from $\Gamma$ leading to a subgraph (not necessarily a core) are obtained by mapping one of the copies of $K_5$ into the other. Now,  suppose that there exists a map $M$ with underlying graph $\Gamma$ that is not a map core. We can apply the arguments used in the proof of Theorem~\ref{thm:core_char} to conclude that there exist a prefacial cross-circuit $\kappa$ (which is not a facial walk) and a homomophism from $M$ to a strict submap of $M$ that maps the interior of $\kappa$ to $\kappa$. In particular, this homomorphism is a graph homomorphism from $\Gamma$ to a subgraph that must be $K_5$. Therefore, the underlying graph of the interior of $\kappa$ should contain $K_5$, which is a contradiction as the interior of $\kappa$ is a plane map.

\subsection{The poset of map cores}\label{sec:cores_poset}

Having now described how to tell whether a map is a core, we introduce and study the poset of map cores, comparing its properties with those of the poset of graph cores.

For fixed $g\in \mathbb{Z}$, let $\mathcal{M}(g)$ and $\mathcal{M}_{c}(g)$ be the sets of, respectively, maps and cores of signed genus~$g$. For $M,M'\in\mathcal{M}(g)$, we say that $M\leq M'$ if there exists a map homomorphism from $M$ to $M'$, that is, if $M\to M'$.
The associated strict order relation, $M<M'$, holds if $M\to M'$ and $M'\not \to M$. 
The relation $\leq$ is defined analogously on $\mathcal{M}_{c}(g)$; indeed, it is obtained from the preorder $(\mathcal{M}(g), \leq)$ upon quotienting out by the equivalence relation $\cong$ defined by $M\cong M'$ if $M\leq M'$ and $M'\leq M$.  

\begin{lemma}
The relation $\leq$ defines a preorder on $\mathcal{M}(g)$ and a partial order on $\mathcal{M}_{c}(g)$.
\end{lemma}

\begin{proof}
The preorder part follows by Proposition \ref{prop:epi_mono_iso} and the fact that the composition of homomorphisms between maps in $\mathcal{M}(g)$ is again a homomorphism.  For cores, Proposition~\ref{p.cu} ensures that the preorder is in fact a partial order, since we are considering maps up to  isomorphism.
\end{proof}

The first property that we establish is that the poset  $(\mathcal{M}_{c}(g), \leq)$ is connected, by first showing that the preorder $(\mathcal{M}(g), \leq)$ is connected. Recall that a preorder $(P,\leq)$ is connected if for every pair of elements $a,b$ there is a sequence $a=c_0, c_1, \dots, c_{k-1}, c_k=b$ such that $c_i\leq c_{i+1}$ or $c_{i+1}\leq c_i$ for each $i=0,1,\dots, k-1$. 

\begin{lemma} \label{lem:core_connection}
Given connected maps $M$ and $M'$ of the same signed genus and with at least one edge, 
$M'$ can be obtained from $M$ by a sequence of vertex gluings,  
vertex splittings (Definition~\ref{def:vertex_splitting} with the extra condition of Lemma~\ref{lem:riff_signedgenus}\ref{en:riff_signedgenus_2}(i)), 
edge splittings, 
and duplicate edge gluings. 
\end{lemma}
\begin{proof}
Vertex gluing and vertex splitting under the conditions indicated are inverse operations, as are duplicate edge gluing and edge splitting, and all these operations preserve signed genus and connectivity.   A connected map with more than one face has some edge that is a dual link; dually, a connected map with more than one vertex has some edge that is a link. 
For a dual link  $e\equiv (\;c \;\;\alpha_2\alpha_0 c\;)\;(\;\alpha_2 c\;\; \alpha_0\alpha_2c\;)$, crosses $c$ and $\alpha_2c$ belong to different faces. Splitting the vertex containing crosses $c$ and $\alpha_2 c$ by riffling $c$ and $\tau c$  
has the effect of adding a vertex and removing a face (so that all the crosses of $e$ now belong to a single face), while the signed genus and the connectivity are preserved, so the number of faces is reduced. 
This procedure can be repeated until the map becomes a quasi-tree. If this quasi-tree has more than one vertex, by contracting links we can reduce the number of vertices, while preserving signed genus (Lemma~\ref{lem:effect_del_sg}\ref{en:effect_del_sg_2}) and connectivity, thus producing no new faces; the quasi-tree can thus be made simultaneously into a bouquet by a sequence of such link contractions. Such a map with one vertex and one face is called by Tutte a {\em unitary map}. If the map has at least two edges, contracting a link $e\equiv (\;c \;\;\alpha_2\alpha_0 c\;)\;(\;\alpha_2 c\;\; \alpha_0\alpha_2c\;)$ can be realized by a sequence of vertex gluings/vertex splittings (Definition~\ref{def:vertex_splitting} with the extra condition of Lemma~\ref{lem:riff_signedgenus}\ref{en:riff_signedgenus_2}(i)) and duplicate edge gluings/edge splittings as follows: (1) glue the two vertices incident with $e$ by riffling crosses $c$ and $\phi c$, making $e$ become a loop; then (2) split the vertex incident with this loop by riffling $c$ and $\alpha_2\alpha_0 c$, thereby creating a leaf vertex; then (3) glue this leaf vertex to the vertex containing $\phi^{-1}c$ by riffling $\alpha_0\alpha_2 c$ and $\phi^{-1}c$, thereby creating a pair of duplicate edges; and finally (4) glue these duplicate edges. 

 We have just seen how any connected map with at least one edge can be taken to a unitary map by a sequence of vertex gluings/vertex splittings and duplicate edge gluings/edge splittings. 
The  \emph{orientable canonical map} of genus $g$ is the unitary map whose vertex permutation has a cycle of the form\footnote{When $g=0$ the canonical map is an isolated vertex and the vertex permutation consists of a pair of empty cycles; as Tutte does not allow edgeless maps, he defines the canonical map of genus zero to be the map consisting of a single link.}
$(\;\;c_1\;\;\; c_2 \;\;\; \alpha_0\alpha_2 c_2\;\;\; \alpha_0\alpha_2 c_1 \;\;\cdots \;\;c_{2g-1}\;\;\; c_{2g} \;\;\; \alpha_0\alpha_2 c_{2g}\;\;\; \alpha_0\alpha_2 c_{2g-1}\;\;),$
and the \emph{non-orientable canonical map} of genus $g$ has vertex permutation with a cycle of the form
\[(\;\;c_1\;\;\; \alpha_0 c_1 \;\;\;c_2\;\;\; \alpha_0 c_2 \;\;\cdots \;\; c_g\;\;\; \alpha_0 c_g\;\;).\]


As remarked after Definition~\ref{def:vertex_splitting}, Tutte defines an operation of vertex splitting~\cite[Figs. X.7.1 and X.7.2]{tutte01} related to ours in following our vertex splitting operation by the insertion of a link;  
this makes his operation inverse to that of contracting a link. Thus, to realize Tutte's vertex splitting operation we take the inverse sequence of four operations above that realized the operation of link contraction. In other words, start with the inverse of step (4) above (split an edge to make a duplicate edge), follow this by the inverse of step (3) (split an endpoint of the duplicate edge to make a leaf vertex), then carry out the inverse of step (2) (glue the leaf vertex to the vertex incident to it), finishing with the inverse of step (1) (split the endpoint of the loop).  
Tutte showed \cite[Theorem~X.37]{tutte01} that, using vertex splittings in his sense and link contractions, one can transform any unitary map into a canonical map while maintaining the signed genus.\footnote{The crosses belonging to a \emph{cross cap} (twisted loop appearing as a pair $c, \alpha_0 c$ in a cycle of the vertex permutation) or to a \emph{handle} (two  interlaced non-twisted loops $a, b, \alpha_0\alpha_2 a, \alpha_0\alpha_2 b$ in a cycle of the vertex permutation) appear consecutively in the canonical map, or, to use Tutte's term, they are \emph{assembled}: Tutte in this way gives a combinatorial  proof of the classification theorem for compact surfaces.}
Hence the same is true when using vertex gluings and splittings in our sense, along with edge splittings and duplicate edge gluings, as these operations can simulate the two operations used by Tutte in the way described above. 

Finally, then, we compose the sequence of operations that reduces $M$ to the canonical map of its signed genus with the inverse of the sequence of operations reducing $M'$ to the same canonical map, which takes the canonical map to $M'$. In this way we move from $M$ to $M'$ via the canonical map of their signed genus by a sequence of operations 
of the type described in the lemma statement. 
\end{proof}

\begin{theorem} \label{thm:cores_connected}
The poset  $(\mathcal{M}_{c}(g), \leq)$ is connected.
\end{theorem}
\begin{proof}
By Lemma~\ref{lem:core_connection}, for each pair of maps $M, M'\in\mathcal M(g)$ with at least one edge there is a sequence of maps $M_0=M, M_1, M_2, \dots, M_k=M'$ such that there is either a homomorphism from $M_i$ to $ M_{i+1}$ (when $M_{i+1}$ is obtained from $M_i$ by vertex gluing or duplicate edge gluing) or from $M_{i+1}$ to $M_i$
(when $M_{i+1}$ is obtained from $M_i$ by vertex splitting or edge splitting); all the operations are the ones described in Lemma~\ref{lem:core_connection}.
With the additional observation that the single vertex can be mapped everywhere in the case of $g=0$, this shows that
the preorder $(\mathcal{M}(g), \leq)$ is connected; then, by considering the corresponding homomorphisms between their cores, we obtain that the poset $(\mathcal{M}_{c}(g), \leq)$ is also connected.
\end{proof}

We recall some further definitions from poset theory. Let $a,b$ be elements of a poset $(P,<)$ such that $a<b$. The poset $P$ is \emph{dense} between $a$ and $b$
if for any $c,d\in P$ with $a\leq c < d\leq b$, there exists an element $e$ such that $c<e<d$; the pair $a$, $b$ forms a \emph{gap} in $P$ if there is no element $c$ with $a<c<b$ (alternatively, $b$ covers  $a$, i.e., $b$ is an immediate successor of $a$).
In the poset of (cores of) graphs with order $G\leq H$ defined by the existence of a graph homomorphism $G\to H$, each comparable pair of cores (with the exception of the one between the vertex and $K_2$) defines an interval that is a dense total order (actually, such interval is universal, in the sense that every countable partial order can be seen as one of its suborders \cite{FIALA2017101}).
The following theorem shows that the partial order of map cores is drastically different from that of graph cores.

\begin{theorem}\label{thm:no_dense_order_cores}
In the partial order $(\mathcal{M}_{c}(g), \leq)$, there is no pair of cores $N_1, N_2$ with $N_1<N_2$ and a dense total order between them.
\end{theorem}

\begin{proof}
Suppose for a contradiction that there is a dense total order between two cores $N_1, N_2\in \mathcal{M}_{c}(g)$, and assume first that both are connected. Thus, there are infinitely many cores in between $N_1$ and $N_2$, and there must exist two cores $N_3,N_4\in \mathcal{M}_{c}(g)$ with $N_1<N_3<N_4<N_2$ such that every core between them (including them) have the same number of odd-degree faces\footnote{$N_2$ has a finite  number of odd-degree faces, and the total number of odd-degree faces is non-decreasing when taking homomorphic images: removal of parallel edges preserves their degree, while vertex gluing always produces one odd-degree face from an odd-degree face, and either zero or two odd-degree faces from an even-degree face.}, and there is a dense total order between them. 
 Let $\mathcal{N}_{3,4}$ be the set of cores between $N_3$ and $N_4$.

\vspace{0.15cm}

\noindent \textbf{Claim:} There exist two cores $N_5, N_6\in \mathcal{N}_{3,4}$ with $N_5<N_6$ such that:
\begin{itemize}
    \item there is a dense total order of cores from $\mathcal{N}_{3,4}$ in between $N_5$ and $N_6$, and
    \item for each pair of cores $N'<N''$ in between $N_5$ and $N_6$ (including them) in the total order, there is a homomorphism $N'\to N''$ that induces a one-to-one correspondence between faces of odd degree and preserves their degree.
\end{itemize}

\begin{proof}[Proof of the claim] Let $N,\widetilde{N}\in \mathcal{N}_{3,4}$ with $N<\widetilde{N}$, and let $h: N\to \widetilde{N}$. The image under $h$ of each odd-degree face of $N$  is a face that induces a prefacial cross-circuit $\kappa$ in $\widetilde{N}$ of length smaller or equal than its degree (here we use that $h(N)$ is a submap of $\widetilde{N}$ and Lemma~\ref{lem:cont_submap}; note that the image of a face does not induce a prefacial cross-circuit, but rather a union of prefacial cross-circuits, we select one of them). Each of these $\kappa$'s contains exactly one face $z_{\kappa}$  of odd degree in $\widetilde{N}$ as, by Observation~\ref{obs.parity_faces}, it should contain at least one, and it cannot contain more 
since the number of odd faces in $\widetilde{N}$ would be strictly larger than that in $N$, a contradiction with $N,\widetilde{N}\in \mathcal{N}_{3,4}$.
Since $\widetilde{N}$ is a core, by Theorem~\ref{thm:core_char}, either $\phi_{z_{\kappa}}=\kappa$,  or the degree of $z_{\kappa}$  is strictly smaller than $\ell(\kappa)$.
As there are infinitely many cores in $\mathcal{N}_{3,4}$, totally ordered,  these degrees
 stabilize between two cores, which are the desired $N_5,N_6$. These cores $N_5,N_6$ can be further assumed to contain a dense total order between them, as there is a dense total order between $N_3$ and $N_4$.
\end{proof}

\noindent Let $\mathcal{N}_{5,6}$ be the set of cores totally ordered in between $N_5$ and $N_6$. We next prove that there are two cores in $\mathcal{N}_{5,6}$ with finitely many cores in between them, which contradicts the assumption that there is a dense total order between $N_1$ and $N_2$.

 Let $h: N\to N'$ where $N<N'$ are two cores in $\mathcal{N}_{5,6}$. 
Each even-length facial walk $\lambda$ of $N$ turns into several even-length facial walks in $h(N)\subseteq N'$ whose total length is at most $ \ell(\lambda)$ (it may happen that the face completely disappears); this follows from the fact that no even-degree face can be subdivided (applying the sequence of vertex and duplicate edge gluings) to create two odd-degree faces. These even-degree faces in $h(N)$ become prefacial cross-circuits $\kappa$ in the core $N'$. If $\kappa$ is not a facial walk, by Theorem~\ref{thm:core_char}, the faces in the interior of $\kappa$ have even degree strictly smaller than $\ell(\kappa)$. Therefore, when we consider a homomorphism from $N_5$ into a core in $\mathcal{N}_{5,6}$, its image is obtained by either reducing the number of faces (gluing duplicated edges) or subdividing  even-degree faces into smaller even-degree faces (gluing vertices); one can place within these faces a plane submap with all the faces of strictly smaller even degree  (we use that the codomain is another core, and Theorem~\ref{thm:core_char} where the prefacial cross-circuit $\kappa$ is a part of a facial walk from $N_5$).
Since there is an absolute minimum on the degree of a face, this process cannot be performed infinitely many times.
\footnote{When fixing the image of a homomorphism, the number of faces themselves is finite, even though we could fill a given face with arbitrarily many configurations that are cores using faces of strictly smaller degree; this arbitrary number will be fixed again, and all those inner faces will have strictly smaller degree. Thus, this procedure cannot be repeated infinitely many times once that arbitrary number is fixed and a particular collection of cores is chosen.} Hence, at some point, the number and the degree of the even-degree faces is stabilized. Since we have also stabilized the number of odd-degree faces and their respective degree, and there is only a finite number of maps with such parameters, we conclude that the interval between $N_5$ and $N_6$ is not dense.  This finishes the proof when the cores $N_1,N_2$ are connected.

Suppose now that the cores $N_1,N_2$ are not connected; we argue similarly on the connected components as they are cores and a map homomorphism can only merge them. Since the signed genus is fixed and the Euler genus is additive along connected components, up to plane connected components, we can find two cores $N_3, N_4$ with $N_1<N_3<N_4<N_2$ such that the profile of connected components with different signed genuses (different from $0$) is the same, and it is the same with all the cores in-between. Regarding the plane connected components, we next use an argument on the number of odd-degree faces. 

As mentioned in the applications of Theorem~\ref{thm:core_char}, each of the plane connected components either has at least three odd-degree faces or it is a plane odd cycle or it is bipartite (all faces of even length). In the latter case, the core is an edge, which can be mapped anywhere else. We also know that there cannot be more than one plane odd cycle among the connected components since otherwise one could be mapped into another and the whole map would not be a core. 
Furthermore, a homomorphism between cores of $\mathcal{N}_{3,4}$ cannot merge one plane connected component
into a non-plane one, since that would strictly increase the number of odd-degree faces of the non-plane component (each plane component besides, possibly, the odd cycle, has at least three odd-degree faces, while merging two connected components can reduce the total number of odd-degree faces by at most two).
Thus, as the overall number of odd-degree faces is bounded in $N_4$, the number of plane connected components is also bounded for each of the cores in $\mathcal{N}_{3,4}$. Therefore, there exist $N_5, N_6\in \mathcal{N}_{3,4}$ with $N_5<N_6$ that have the same profile of connected components and odd-degree faces (the genuses of the components, and the number of odd-degree faces on each of the connected components), and with a dense total order between them. Applying the same reasoning as in the connected case, we find two cores in $\mathcal{N}_{5,6}$ with finitely many cores in between them.  
\end{proof}

\begin{corollary}\label{cor:no_dense_order_between_pair}
In the preorder $(\mathcal{M}(g), \leq)$, there is no pair of maps $M_1, M_2$ with $M_1<M_2$ and a dense total order between them.
\end{corollary}

\begin{proof}
This follows from the fact that a dense total order between two maps would imply, by Proposition~\ref{p.prop_cores} part \ref{en.core3}, a dense total order between their cores; 
here we  also use that if $M_1<M_2$ then their cores are different.
\end{proof}


Let us now exhibit an infinite chain of map cores. For the remainder of this section we use $C_i$ to denote the plane map of a graph cycle of length $i$.
Let $T_{i,j}$ be the plane map obtained  by adding to  $C_4$ a path of length $i$ between the first and the third vertex, and a path of length $j$ between the second and the fourth vertex (each of these paths subdivides one of the faces induced by the $C_4$ on the plane). Thus, $T_{i,j}$ has four faces, two of length $i+2$ and two of length $j+2$. Since $T_{2k+1, \, 2s+1}$ has all faces of odd degree and no vertices of degree one, then  $T_{2k+1, \, 2s+1}$ is a core for all $k,s\geq 0$. By folding one edge appropriately, we have $T_{2k+1, \, 2s+1} \to T_{2k+1, \, 2(s-1)+1}$ for $s\geq 1$ and $k\geq 0$; symmetrically, $T_{2k+1, \, 2s+1} \to T_{2(k-1)+1, \, 2s+1}$ for $s\geq0$ and $k\geq 1$. As $T_{i,3}$ has a face of length $5$, we conclude that $
C_5\to \cdots \to T_{2s+1,3} \to T_{2(s-1)+1,3} \to \cdots \to T_{3,3} \to T_{1,3}$
is an infinite (ordered) chain of cores.

We can also find an infinite chain of gaps by
using that a plane connected map that is a core is bipartite (an edge or an isolated vertex) or a cycle of odd length or has at least four faces of odd degree, and that the codomain of a map homomorphism must have at least as many odd-degree faces as the domain.

\begin{theorem}\label{thm:gap_odd_cycles}
Let $C_{i}$ denote the plane cycle of length $i$. For each $k\geq 1$, $C_{2k+3}\to C_{2k+1}$ and
there is a gap between $C_{2k+3}$ and $C_{2k+1}$.
\end{theorem}


For each signed genus~$g$, we obtain an infinite antichain in $\mathcal{M}_c(g)$ and $\mathcal{M}(g)$ by considering maps with an increasing number of odd-degree faces, and an increasing degree for each of these faces. 
In \cite{hell_core_1992} the authors consider, for each graph $H$, the family 
$\mathcal{G}(H)=\{G\:|\: G\text{ is a graph and }G\to H\}$
and show the existence of a graph $H_0$ for which $\mathcal{G}(H_0)$ contains an infinite antichain \cite[Theorem~6]{hell_core_1992}
(the argument uses only planar graphs); we next present a similar construction to find arbitrarily large antichains in the posets given by
$\mathcal{F}(H)=\{M\:|\: M\text{ is a map and }M\to H\}$.
\begin{theorem}\label{thm:fin_indep_cores}
    For every even $n$, and odd $k\geq 3$, there exist plane maps $B$ and $A_1,\ldots,A_n$ such that
    \begin{itemize}
        \item the underlying graphs of $A_1,\ldots,A_n$ are graph cores, and $B$ is a map core,
        \item each of  $B$ and $A_1,\ldots,A_n$ have underlying graph of odd girth $2k+1$,
        \item $A_i\to B$ for each $i\in [n]$, and $A_i \nrightarrow A_j$ for distinct $i,j\in [n]$.
    \end{itemize}
\end{theorem}

\begin{proof}
We slightly adapt the arguments from \cite[Theorem~6]{hell_core_1992}; see Figures~\ref{fig:uandti} and~\ref{fig:t1tom1} for illustrations of the constructions.

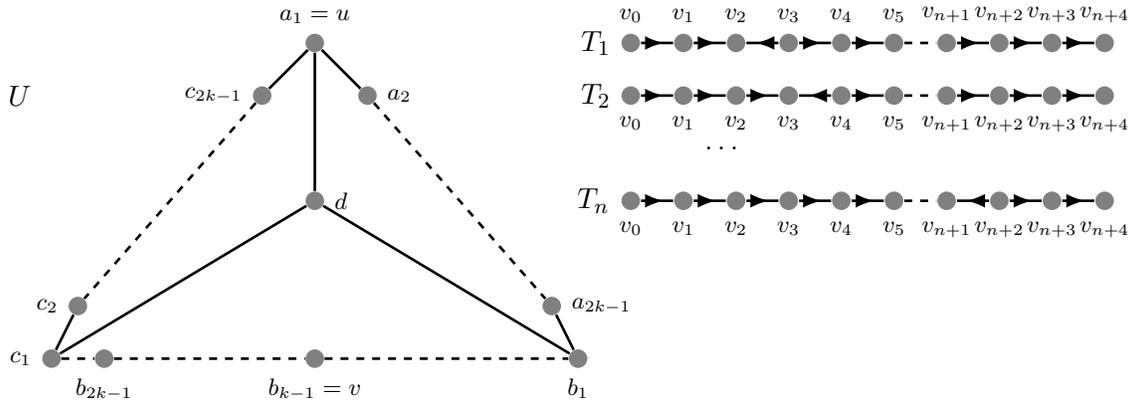
\begin{figure}[htb]
	\begin{center}
		\begin{tikzpicture}[scale=0.70] 
		
		\tikzstyle{vertex}=[circle,fill=black!50,minimum size=7pt,inner sep=0pt]

        \node[vertex,label=left:{\footnotesize{$c_1$}}] (c1) at (-5,4) {};
        \node[vertex,label=below:{\footnotesize{$b_{2k-1}$}}] (b2gm1) at (-4,4) {};
        \node[vertex,label=below:{\footnotesize{$b_{k-1}=v$}}] (v) at (0,4) {};
        \node[vertex,label=below:{\footnotesize{$b_1$}}] (b1) at (5,4) {};

        \node[vertex,label=left:{\footnotesize{$c_2$}}] (c2) at (-4.5,5 ) {};
        \node[vertex,label=left:{\footnotesize{$c_{2k-1}$}}] (c2gm1) at (-1,9) {};

         \node[vertex,label=above:{\footnotesize{$a_{1}=u$}}] (a1) at (0,10) {};

         \node[vertex,label=right:{\footnotesize{$a_{2}$}}] (a2) at (1,9) {};
         \node[vertex,label=right:{\footnotesize{$a_{2k-1}$}}] (a2gm1) at (4.5,5) {};

         \node[vertex,label=right:{\footnotesize{$d$}}] (d) at (0,7) {};

        \draw[line width= 1pt,dashed,-] (c1) -- (b2gm1);
        \draw[line width= 1pt,dashed,-] (b2gm1) -- (v);
        \draw[line width= 1pt,-] (b1) -- (a2gm1);
        \draw[line width= 1pt,dashed,-] (a2) -- (a2gm1);
        
        \draw[line width= 1pt,-] (a2) -- (a1);
        \draw[line width= 1pt,-] (a1) -- (c2gm1);
        \draw[line width= 1pt,dashed,-] (c2gm1) -- (c2);
        \draw[line width= 1pt,-] (c2) -- (c1);
        \draw[line width= 1pt,dashed,-] (v) -- (b1);

        \draw[line width= 1pt,-] (d) -- (c1);
        \draw[line width= 1pt,-] (d) -- (b1);
        \draw[line width= 1pt,-] (d) -- (a1);

        \node[label=left:{$U$}] (U) at (-5,9) {};

        \node[label=left:{$T_1$}] (T1) at (6,10) {};
        \node[label=left:{$T_2$}] (T2) at (6,9) {};

\begin{scope}[thick,decoration={
    markings,
    mark=at position 0.7 with {\arrow[scale=1.2]{latex}}}
    ] 

        \node[vertex,label=above:{\footnotesize{$v_0$}}] (vt11) at (6,10) {};
        \node[vertex,label=above:{\footnotesize{$v_1$}}] (vt12) at (7,10) {};
        \node[vertex,label=above:{\footnotesize{$v_2$}}] (vt13) at (8,10) {};
        \node[vertex,label=above:{\footnotesize{$v_3$}}] (vt14) at (9,10) {};
        \node[vertex,label=above:{\footnotesize{$v_4$}}] (vt15) at (10,10) {};
        \node[vertex,label=above:{\footnotesize{$v_5$}}] (vt16) at (11,10) {};
        \node[vertex,label=above:{\footnotesize{$v_{n+1}$}}] (vt17) at (12,10) {};
        \node[vertex,label=above:{\footnotesize{$v_{n+2}$}}] (vt18) at (13,10) {};
        \node[vertex,label=above:{\footnotesize{$v_{n+3}$}}] (vt19) at (14,10) {};
        \node[vertex,label=above:{\footnotesize{$v_{n+4}$}}] (vt110) at (15,10) {};

        \draw[line width= 1pt,-, postaction={decorate}] (vt11) -- (vt12);
        \draw[line width= 1pt,-, postaction={decorate}] (vt12) -- (vt13);
        \draw[line width= 1pt,-, postaction={decorate}] (vt14) -- (vt13);
        \draw[line width= 1pt,-, postaction={decorate}] (vt14) -- (vt15);
        \draw[line width= 1pt,-, postaction={decorate}] (vt15) -- (vt16);
        \draw[line width= 1pt,dashed,-] (vt16) -- (vt17);
        \draw[line width= 1pt,-, postaction={decorate}] (vt17) -- (vt18);
        \draw[line width= 1pt,-, postaction={decorate}] (vt18) -- (vt19);
        \draw[line width= 1pt,-, postaction={decorate}] (vt19) -- (vt110);

        \node[vertex,label=below:{\footnotesize{$v_0$}}] (vt21) at (6,9) {};
        \node[vertex,label=below:{\footnotesize{$v_1$}}] (vt22) at (7,9) {};
        \node[vertex,label=below:{\footnotesize{$v_2$}}] (vt23) at (8,9) {};
        \node[vertex,label=below:{\footnotesize{$v_3$}}] (vt24) at (9,9) {};
        \node[vertex,label=below:{\footnotesize{$v_4$}}] (vt25) at (10,9) {};
        \node[vertex,label=below:{\footnotesize{$v_5$}}] (vt26) at (11,9) {};
        \node[vertex,label=below:{\footnotesize{$v_{n+1}$}}] (vt27) at (12,9) {};
        \node[vertex,label=below:{\footnotesize{$v_{n+2}$}}] (vt28) at (13,9) {};
        \node[vertex,label=below:{\footnotesize{$v_{n+3}$}}] (vt29) at (14,9) {};
        \node[vertex,label=below:{\footnotesize{$v_{n+4}$}}] (vt210) at (15,9) {};
        
        \draw[line width= 1pt,-, postaction={decorate}] (vt21) -- (vt22);
        \draw[line width= 1pt,-, postaction={decorate}] (vt22) -- (vt23);
        \draw[line width= 1pt,-, postaction={decorate}] (vt23) -- (vt24);
        \draw[line width= 1pt,-, postaction={decorate}] (vt25) -- (vt24);
        \draw[line width= 1pt,-, postaction={decorate}] (vt25) -- (vt26);
        \draw[line width= 1pt,dashed,-] (vt26) -- (vt27);
        \draw[line width= 1pt,-, postaction={decorate}] (vt27) -- (vt28);
        \draw[line width= 1pt,-, postaction={decorate}] (vt28) -- (vt29);
        \draw[line width= 1pt,-, postaction={decorate}] (vt29) -- (vt210);

        \node[label=right:{\ldots}] at (7,8) {};

        \node[label=left:{$T_n$}] (Tn) at (6,7) {};

        \node[vertex, label=below:{\footnotesize{$v_0$}}] (vtn1) at (6,7) {};
        \node[vertex, label=below:{\footnotesize{$v_1$}}] (vtn2) at (7,7) {};
        \node[vertex,label=below:{\footnotesize{$v_2$}}] (vtn3) at (8,7) {};
        \node[vertex,label=below:{\footnotesize{$v_3$}}] (vtn4) at (9,7) {};
        \node[vertex,label=below:{\footnotesize{$v_4$}}] (vtn5) at (10,7) {};
        \node[vertex,label=below:{\footnotesize{$v_5$}}] (vtn6) at (11,7) {};
        \node[vertex,label=below:{\footnotesize{$v_{n+1}$}}] (vtn7) at (12,7) {};
        \node[vertex,label=below:{\footnotesize{$v_{n+2}$}}] (vtn8) at (13,7) {};
        \node[vertex,label=below:{\footnotesize{$v_{n+3}$}}] (vtn9) at (14,7) {};
        \node[vertex,label=below:{\footnotesize{$v_{n+4}$}}] (vtn10) at (15,7) {};

        \draw[line width= 1pt,-, postaction={decorate}] (vtn1) -- (vtn2);
        \draw[line width= 1pt,-, postaction={decorate}] (vtn2) -- (vtn3);
        \draw[line width= 1pt,-, postaction={decorate}] (vtn3) -- (vtn4);
        \draw[line width= 1pt,-, postaction={decorate}] (vtn4) -- (vtn5);
        \draw[line width= 1pt,-, postaction={decorate}] (vtn5) -- (vtn6);
        \draw[line width= 1pt,dashed,-] (vtn6) -- (vtn7);
        \draw[line width= 1pt,-, postaction={decorate}] (vtn8) -- (vtn7);
        \draw[line width= 1pt,-, postaction={decorate}] (vtn8) -- (vtn9);
        \draw[line width= 1pt,-, postaction={decorate}] (vtn9) -- (vtn10);

        \end{scope}

	\end{tikzpicture}
\end{center}
\caption{Copy of the graph $U$, and the dimaps $T_i$.}\label{fig:uandti}
\end{figure}

\begin{figure}[htb]
	\begin{center}
		\begin{tikzpicture}[scale=0.55] 
		
		\tikzstyle{vertex}=[circle,fill=black!50,minimum size=7pt,inner sep=0pt]

        \node[label=right:{$A_1$}] (M1) at (-3,8) {};
        \node[vertex, label=above:{\footnotesize{$b_1$}}] (b11) at (0,8) {};
        \node[vertex, label=above:{\footnotesize{$b_1$}}] (b12) at (4,8) {};
        \node[vertex,label=below:{\footnotesize{$b_1$}}] (b13) at (5,4) {};
        \node[vertex,label=above:{\footnotesize{$b_1$}}] (b14) at (12,8) {};
        \node[vertex,label=above:{\footnotesize{$b_1$}}] (b16) at (20,8) {};

        \node[vertex, label=below:{\footnotesize{$c_1$}}] (c11) at (0,0) {};
        \node[vertex, label=below:{\footnotesize{$c_1$}}] (c12) at (4,0) {};
        \node[vertex,label=above:{\footnotesize{$c_1$}}] (c13) at (5,12) {};
        \node[vertex,label=below:{\footnotesize{$c_1$}}] (c14) at (12,0) {};
        \node[vertex,label=below:{\footnotesize{$c_1$}}] (c16) at (20,0) {};

        \node[vertex, label=above:{\footnotesize{$u$}}] (a11) at (-4,4) {};
        \node[vertex, label=right:{\footnotesize{$u$}}] (a12) at (0,6) {};
        \node[vertex,label=left:{\footnotesize{$u$}}] (a13) at (8,6) {};
        \node[vertex,label=right:{\footnotesize{$u$}}] (a14) at (8,6) {};
        \node[vertex,label=right:{\footnotesize{$u$}}] (a15) at (12,6) {};
        \node[vertex,label=right:{\footnotesize{$u$}}] (a16) at (16,6) {};

        \node[vertex, label=above:{\footnotesize{$d$}}] (d1) at (-2,4) {};
        \node[vertex, label=above:{\footnotesize{$d$}}] (d2) at (2,5) {};
        \node[vertex,label=left:{\footnotesize{$d$}}] (d3) at (6,8) {};
        \node[vertex,label=above:{\footnotesize{$d$}}] (d4) at (10,5) {};
        \node[vertex,label=above:{\footnotesize{$d$}}] (d6) at (18,5) {};

        \node[vertex, label=left:{\footnotesize{$v$}}] (v1) at (0,6) {};
        \node[vertex, label=left:{\footnotesize{$v$}}] (v2) at (4,6) {};
        \node[vertex,label=right:{\footnotesize{$v$}}] (v3) at (4,6) {};
        \node[vertex,label=left:{\footnotesize{$v$}}] (v4) at (12,6) {};
        \node[vertex,label=left:{\footnotesize{$v$}}] (v5) at (16,6) {};
        \node[vertex,label=left:{\footnotesize{$v$}}] (v6) at (20,6) {};

        \draw[line width= 1pt,dashed,-] (a11) -- (b11);
        \draw[line width= 1pt,dashed,-] (b11) -- (v1);
        \draw[line width= 1pt,dashed,-] (v1) -- (c11);
        \draw[line width= 1pt,dashed,-] (a11) -- (c11);
        \draw[line width= 1pt,-] (d1) -- (c11);
        \draw[line width= 1pt,-] (d1) -- (b11);
        \draw[line width= 1pt,-] (d1) -- (a11);

        \draw[line width= 1pt,dashed,-] (a12) -- (b12);
        \draw[line width= 1pt,dashed,-] (b12) -- (v2);
        \draw[line width= 1pt,dashed,-] (v2) -- (c12);
        \draw[line width= 1pt,dashed,-] (a12) -- (c12);
        \draw[line width= 1pt,-] (d2) -- (c12);
        \draw[line width= 1pt,-] (d2) -- (b12);
        \draw[line width= 1pt,-] (d2) -- (a12);

        \draw[line width= 1pt,dashed,-] (a13) -- (b13);
        \draw[line width= 1pt,dashed,-] (b13) -- (v3);
        \draw[line width= 1pt,dashed,-] (v3) -- (c13);
        \draw[line width= 1pt,dashed,-] (a13) -- (c13);
        \draw[line width= 1pt,-] (d3) -- (c13);
        \draw[line width= 1pt,-] (d3) -- (b13);
        \draw[line width= 1pt,-] (d3) -- (a13);

        \draw[line width= 1pt,dashed,-] (a14) -- (b14);
        \draw[line width= 1pt,dashed,-] (b14) -- (v4);
        \draw[line width= 1pt,dashed,-] (v4) -- (c14);
        \draw[line width= 1pt,dashed,-] (a14) -- (c14);
        \draw[line width= 1pt,-] (d4) -- (c14);
        \draw[line width= 1pt,-] (d4) -- (b14);
        \draw[line width= 1pt,-] (d4) -- (a14);

        \node[label=right:{\ldots}] at (13,6) {};

               \draw[line width= 1pt,dashed,-] (a16) -- (b16);
        \draw[line width= 1pt,dashed,-] (b16) -- (v6);
        \draw[line width= 1pt,dashed,-] (v6) -- (c16);
        \draw[line width= 1pt,dashed,-] (a16) -- (c16);
        \draw[line width= 1pt,-] (d6) -- (c16);
        \draw[line width= 1pt,-] (d6) -- (b16);
        \draw[line width= 1pt,-] (d6) -- (a16);

	\end{tikzpicture}
\end{center}
\caption{Exchanging edges of $T_1$ by copies of $U$ to create $A_1$. The map $B$ is obtained by gluing the first $u$ and the last $v$, along the exterior face.}\label{fig:t1tom1}
\end{figure}
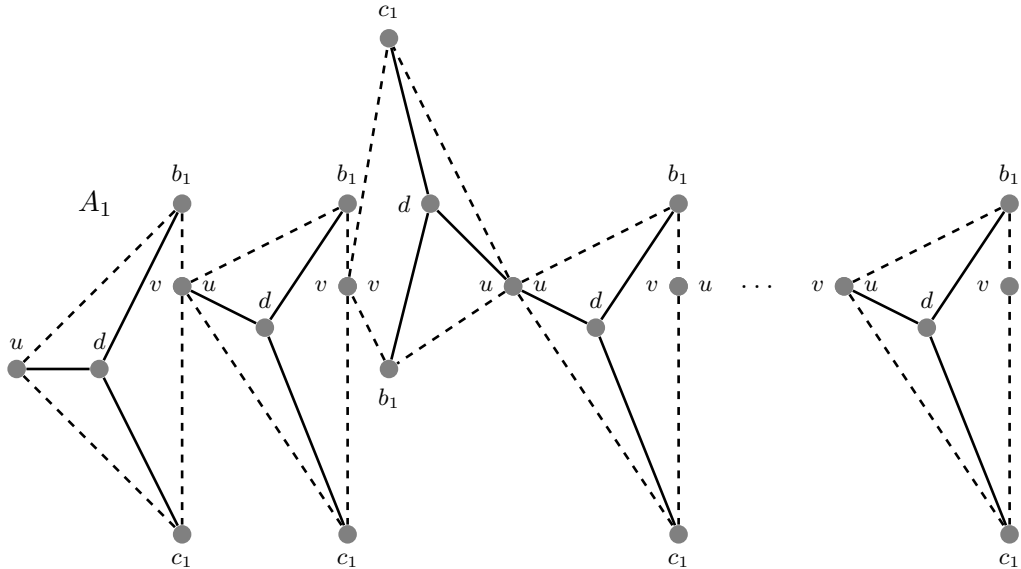

Let $U$ be a graph with vertices $\{a_1,a_2,\allowbreak \dots,\allowbreak a_{2k-1},b_1,b_2,\allowbreak \ldots,\allowbreak b_{2k-1},\allowbreak c_1,c_2,\ldots,c_{2k-1},d\}$ where there is a cycle  $(a_1,a_2,\allowbreak \dots,\allowbreak a_{2k-1},b_1,b_2,\allowbreak \ldots,\allowbreak b_{2k-1},\allowbreak c_1,c_2,\ldots,c_{2k-1}, a_1)$ and $d$ is adjacent to $b_1,c_1,a_1$. Now use the label $u$ for $a_1$ and $v$ for $b_{k-1}$. A plane embedding of $U$ is a map core since all the faces have odd degree. It is also a graph core as each odd cycle should be mapped to an odd cycle, and a cycle of length the odd girth should be mapped to a cycle of length the odd girth; two of the cycles cannot be mapped to a single one as otherwise the third odd-girth-length cycle would be shortened.

Let $T_i$, $1\leq i\leq n$, be a path with $n+4$ directed edges, all of them oriented in one direction but the $i+2$-th edge that is oriented in the opposite direction. Let $G_i$ be the graph obtained by replacing each edge of $T_i$ by a copy of $U$, placing vertex $u$ as the tail of the edge and $v$ as the head.
Now, consider the plane embeddings of the graphs $G_i$, denoted by $A_i$, 
 where all the vertices $b_1$ of the respective triangles are  ``facing up'' except the $b_1$ corresponding to the copy of $U$ given by the backward edge, which is ``facing down''.
The plane map $B$ is formed by gluing the first vertex $u$ with the last vertex $v$. Thus, $B$ consists of $n+4$ copies of $U$ with all but one vertex $b_1$ to the outside of the circle, and one of the vertices $b_1$ to the inside. It is clear that $A_i\to B$, but by the argument in  \cite[Theorem~6]{hell_core_1992} all the graphs $G_i$ are graph cores and $A_i\nrightarrow  A_j$ (indeed, assume not, since $U$ is a core, each copy of $U$ in $A_i$ should be mapped to a copy of $U$, vertices $u$ in different copies should be mapped among them as well as vertices $v$, since $v$ is located slightly to one side along an odd-girth-length cycle; we would thus have a graph homomorphism between $T_i$ and $T_j$, a contradiction) and thus $A_i$ are map cores.
Further, for $n$ even, each face of $B$ has odd degree, and thus it is a map core. This follows from the fact that each face of a plane embedding of $U$ has odd degree, and so the path $uv$ has odd length on one side and even length on the other; in particular, by joining an even number of $U$'s together, and merging the first $u$ and the last $v$ as in $A_i$, we obtain two faces, both of odd degree (the facial walk of one face is obtained as the addition of an odd number of odd-length paths plus an even-length path, and the facial walk of the other face is obtained as the addition of an odd number of even-length paths plus an odd-length path).
\end{proof}

In \cite[Theorem~6]{hell_core_1992}, the graph codomain of the antichain  is the odd cycle of length $2k+1$, which is a (graph) core and does not depend on $n$, but only on $k$. However, since the number of odd faces of $A_1,\ldots,A_n$ increases with $n$, when using the same construction, the map core $B$, codomain of the antichain,  also depends on $n$. 
Thus, we can find arbitrarily large antichains, but not an infinite one. This raises the following question: 
\begin{quote}
Is there a map $B$ for which $\mathcal{F}(B)$ contains an infinite antichain?
\end{quote}

\small
\bibliographystyle{abbrv}
\bibliography{biblio.bib}

\end{document}